\documentclass[12pt]{amsart}
\usepackage{amssymb}
\usepackage[all]{xy}
\theoremstyle{plain} %% This is the default, anyway
\newtheorem{theor}[equation]{Theorem}
\newtheorem{cor}[equation]{Corollary}
\newtheorem{lem}[equation]{Lemma}
\newtheorem{conjecture}[equation]{Conjecture}
\newtheorem{proposition}[equation]{Proposition}

\theoremstyle{definition}
\newtheorem{defin}[equation]{Definition}

\theoremstyle{remark}
\newtheorem{rem}[equation]{Remark}
\newtheorem{ex}[equation]{Example}
\newtheorem{notation}[equation]{Notation}

\newtheorem{propriete}[equation]{Property}

\newcommand{\deuxcat}{\mathcal S}
\newcommand{\propBD}{BD}

\def\build#1_#2^#3{\mathrel{\mathop{\kern0pt#1}\limits_{#2}^{#3}}}
\title{String topology of classifying spaces}% of Lie groups
\author{David Chataur and  Luc Menichi }

\begin{document}

\maketitle 

\centerline{\it Dedicated to Micheline Vigu\'e-Poirrier for her 60th birthday}

\begin{abstract}  Let $G$ be a finite group or a compact connected Lie group and let $BG$ be its classifying space. Let $\mathcal{L}BG:=map(S^1,BG)$ be the free loop space of $BG$ i.e. the space of continuous maps from the circle $S^1$ to $BG$. The purpose of this paper is to study the singular homology $H_*(\mathcal LBG)$ of this loop space. We prove that when taken with coefficients in a field the homology of $\mathcal LBG$ is a homological conformal field theory. As a byproduct of our main theorem, we get a Batalin-Vilkovisky algebra structure on the cohomology $H^*(\mathcal LBG)$. We also prove an algebraic version of this result by showing that the Hochschild cohomology $HH^*(S_* (G),S_*(G))$ of the singular chains of $G$ is a Batalin-Vilkovisky algebra.

\end{abstract}
% 55R12 Transfer
    \vspace{5mm}\noindent {\bf AMS Classification} :  18D50, 55N91, 55P35, 55P48, 55R12, 55R35, 55R40, 57R56, 58D29, 81T40, 81T45

    \vspace{2mm}\noindent {\bf Key words} : free loop space, Hochschild cohomology, props, string topology, topological field theories

\vspace{5mm}

%%%%%%%%%%%%%%%%%%%%%%%%%%%%%%%%%%%%%%%%%%%%%%%%%%%%%%%%%%%%%%%%%%%%%%%%%%%%%%%%%%%%%%%%%%%%%%%%%%%%%%%%%%%%%%%%%%%
\centerline{\bf Introduction}
$\footnotetext{The first author is supported in part by ANR grant JCJC06 OBTH.}$
\\
\\
Popularized by M. Atiyah and G. Segal~\cite{Atiyah:TQFTsIHES,Segal:defCFT}, topological quantum field theories, more generally quantum field theories and their conformal cousins have entered with success the toolbox of algebraic topologists.
\\
%If these theories were initially used in low dimensional topology to give topological invariants of knots, links and $3$-manifolds. They were also an essential ingredient to understand the representation theory of loop groups. And following an intuition of G. Segal they are used to try to build geometrical cocycles called "elliptic objects" for elliptic comohology theories and for the theory of topological modular forms.
%\\
\\
Recently they appeared in a fundamental way in the study of the algebraic and differential topology of loop spaces. Let $M$ be a compact, closed, oriented $d$-dimensional manifold. Let $\mathcal LM$ be the free loop space of $M$. By definition $\mathcal LM$ is the space of continuous maps of the circle into $M$. In their foundational paper "String topology" \cite{Chas-Sullivan:stringtop} M. Chas and D. Sullivan introduced a new and very rich algebraic structure on the singular homology $H_*(\mathcal LM,\mathbb Z)$ and its circle equivariant version $H_*^{S^1}(\mathcal LM,\mathbb Z)$. This is a fascinating generalization to higher dimensional manifolds of W. Goldman's bracket \cite{Goldmancrochet} which lives on the $0$-th space $H_0^{S^1}(\mathcal L\Sigma,\mathbb Z)$ of the equivariant homology of the loops of a closed oriented surface $\Sigma$ (the free homotopy classes of curves in $\Sigma$). Thanks to further works of M. Chas and D. Sullivan \cite{Chas-Sullivan:closedLiebialgebra} $2$-dimensional topological field theories,  because they encode the ways strings can interact, became the classical algebraic apparatus to understand string topology operations (see also R. Cohen and V. Godin's paper \cite{1095.55006}).  
\\
Let us also mention that on the geometric side this theory is closely related to Floer homology of the cotangent bundle of $M$ and symplectic field theory. On its more algebraic side the subject relates to Hochschild cohomology and various generalizations of the Deligne's conjecture. 
\\
\\
More recently the theory was successfully extended to topological orbifolds \cite{Lupercio-Uribe-Xicotencatl:Orbistring} and topological stacks~\cite{Behrend-Ginot-Noohi-Xu}. K. Gruher and P. Salvatore also studied a Pro-spectrum version of string topology for the classifying space of a compact Lie group \cite{gruher-salvatore:genstringtopop}. When twisted topological complex $K$-theory is applied to this Pro-spectrum, the Pro-cohomology obtained is related to Freed-Hopkins-Teleman theory of twisted K-theory and the Verlinde algebra \cite{Gruher:duality}. With these last developments string topology enters the world of equivariant topology.
In this paper we explore the string topology of the classifying space of finite groups and of connected compact Lie groups, advocating that string topology could be applied very naturally in this setting.  
\\
Our main theorem is about the field theoretic properties of $\mathcal LBG$. We prove that the homology of $\mathcal LBG$ is a homological conformal field theory. Let us notice that V. Godin~\cite{Godin:higherstring} proved an analoguous result for $\mathcal LM$, but the techniques used in this setting (fat graphs, embeddings) are completely transverse to those of this paper.
%This result about $\mathcal LBG$ unifies many interesting algebraic structures, we explain how to relate some of them to known structures on the Hochschild cohomology of the singular chains $C_*(G)$. In the next section we give the plan of the paper and review our main results.   
%\\
%\\
%The authors are aware that such a generalization of string topology to $BG$ treated from the point of view of topological field theories has also been understood by several other people in particular K. Costello whose paper \cite{Cos} was a great inspirational source, and also D. Freed, M. Hopkins and C. Teleman whose work on the Verlinde algebra uses part of this structure and of the yoga of correspondences.   

%%%%%%%%%%%%%%%%%%%%%%%%%%%%%%%%%%%%%%%%%%%%%%%%%%%%%%%%%%%%%%%%%%%%%%%%%%%%%%%%%%%%%%%%%%%%%%%%%%%%%%%%%%%%%%%%%%%%
{\em Acknowledgment:}
The authors would like to thank all the people who helped us in this paper:
Michael Crabb, Yves F\'elix, J\'erome Scherer, especially Jean-Claude Thomas
and Nathalie Wahl for their helpful comments, Craig Westerland for explaining the second
author the cactus action on $H_*(\mathcal{L}M)$~\cite[2.3.1]{Cohen-Hess-Voronov:stringtopacyclhom}.
\section{Plan of the paper and Results}
Our main intuition behind all these field theoretic structures was completely algebraic. The second author proved in \cite{MenichiL:BValgaccoHa} that the Hochschild cohomology of a symmetric Frobenius algebra is a Batalin-Vilkovisky algebra.
The group ring ${\Bbbk}[G]$ of a finite group $G$ is one of the classical
example of symmetric Frobenius algebras (Example~\ref{example d'algebre symetrique} 2)).
Therefore the Hochschild cohomology of ${\Bbbk}[G]$
with coefficients in its dual,
$HH^*({\Bbbk}[G];{\Bbbk}[G]^\vee)$,
is a Batalin-Vilkovisky algebra (See the beginning of Section~\ref{hochschild cohomology} for more details).
As we know from results of Burghelea-Fiederowicz \cite{Burghelea-Fiedorowicz:chak} and Goodwillie \cite{Goodwillie:cychdfl} that
the Hochschild cohomology $HH^*({\Bbbk}[G];{\Bbbk}[G]^\vee)$ is isomorphic as ${\Bbbk}$-modules to $H^*(\mathcal LBG;{\Bbbk})$.
Therefore, we obtain
\\
%The singular homology $H_*(M)$ of a closed compact oriented manifold $M$ or more generally of a
%Poincar\'e duality space together with the intersection product is one of the classical examples of symmetric Frobenius algebra. When the manifold is formal, the Hochschild cohomology of $H_*(M)$ together with this BV-structure is related to the string topology operations built geometrically by M. Chas and D. Sullivan (see Cohen-Jones'paper \cite{CohJon:ahtrostringtopology}).  
%\\
%After recalling the definitions of Frobenius and symmetric Frobenius algebras we prove that the homology of a connected Lie group together with the Pontryagin product is a symmetric Frobenius algebra. In fact, we prove that this structure can be in some weak sense lifted at the chain level. Together with this lifting the Hochschild cohomology of $C_*(G)$ comes equipped with a square zero operator $\Delta$ induced by the Connes operator, we pove
%\\
\\
{\bf Inspirational Theorem.} {\it Let $G$ be a finite group.
Then the singular cohomology with coefficients in any commutative ring ${\Bbbk}$, $H^*(LBG;{\Bbbk})$, is a Batalin-Vilkovisky algebra.}
\\
\\
The question of the geometric incarnation of this structure is natural and will be partially answered in this paper.
Note that this Batalin-Vilkovisky algebra is highly non-trivial.
For example, the underlying algebra, the Hochschild cohomology ring
 $HH^*({\Bbbk}[G];{\Bbbk}[G])$ is studied and computed in some cases in~\cite{Siegel-Witherspoon:Hochschildcoh}.      
\\
\\
{\bf Section 2} In order to build our "stringy" operations on $\mathcal LBG$ we use correspondences and Umkehr maps :
\\
transfer maps for dealing with finite groups,
\\
integration along the fiber for compact Lie groups.
\\
We recall the basic definitions and properties of these maps.  
\\
\\
{\bf Section 3} As the preceding section this one is expository. We recall the concepts of quantum field theories as axiomatized by Atiyah and Segal. In order to play with these algebraic structures we use the notion of props. The Segal prop of Riemann surfaces and its associated props is described in some details. 
\\
\\
{\bf Section 4}, {\bf section 5}, {\bf section 6}, {\bf section 7} In these sections we build evaluation products for a propic action of the homology of the Segal prop on the homology of the free loop spaces of various groups and topological groups. The aim of these sections is to give a proof of the following theorem.  
\\
\\
{\bf Main Theorem.}
(Theorems~\ref{main theoreme homologie groupe fini} and~\ref{main theoreme homologie groupe de Lie}) {\it
\\
(1) Let $G$ be a discrete finite group or 
\\
(2) Let $G$ be a connected topological group
such that its singular homology $H_*(G,\mathbb F)$
with coefficient in a field is finite dimensional.
\\
Then the singular homology of $\mathcal LBG$ taken with coefficients
in a field, $H_*(\mathcal{L}BG;\mathbb{F})$, is an homological conformal field theory.}
\\
\\
The condition (2) in the main theorem deserves some comments. We have
originally proved the main theorem for a connected compact Lie group $G$. The proof being completely homological, the main theorem can be extended for free to (2) which is obviously a weaker condition. Condition (2) is satisfied for finite loop spaces and since the discovery of the Hilton-Roitberg criminal \cite{Hilton-Roitbergcriminal} one knows that not every finite loop space is homotopy equivalent to a compact Lie group.
\\

When $\mathbb F=\mathbb Q$ one knows that every odd sphere $S^{2n+1}$ is rationally equivalent to an Eilenberg-MacLane space $K(\mathbb Q,2n+1)$, therefore condition (2) holds for the group $K(\mathbb Q,2n+1)$. And when $\mathbb F=\mathbb Z/p$ the condition is satisfied by $p$-compact groups \cite{Dwyer-Wilkerson:hfpmLg}.    
\\
In fact for all the groups $G$ satisfying (2), we will show
that their singular homology $H_*(G;\mathbb F)$ is a symmetric Frobenius algebra, since it is a
finite dimensional cocommutative connected Hopf algebra.
\\
This theorem when restricted to the genus zero and operadic part of the prop of riemann surfaces with boundary gives the topological counterpart of our "inspirational theorem". 
\begin{cor}(Particular case of Corollaries~\ref{Structure BV en cohomologie groupe fini}
and~\ref{Structure BV en cohomologie groupe de Lie})
Let ${\Bbbk}$ be any principal ideal domain.
\\
(1) Let $G$ be a finite group. Then $H^{*}(\mathcal LBG,{\Bbbk})$ is a Batalin-Vilkovisky algebra.
\\
(2) Let $G$ be a connected compact Lie group of dimension $d$.
Then $H^{*+d}(\mathcal LBG,{\Bbbk})$ is a Batalin-Vilkovisky algebra.
\end{cor}

\noindent{\bf Section 8.} We recall some basic facts about Frobenius algebras and Hopf algebras. We prove that the homology of a connected Lie
group together with the Pontryagin product is a symmetric Frobenius algebra
(Theorem~\ref{homologie groupe de Lie algebre symetrique}), in fact we offer two different proofs.
The first is completely algebraic while the second is topological.
\\
\\
{\bf Section 9.} We extend our inspirational theorem from finite groups to Lie groups
(Theoremé\ref{Homologie de Hochschild groupe de Lie}):
{\it Let $G$ be a connected compact Lie group of dimension $d$.
Let $S_*(G)$ be the singular chains of $G$.
The Gerstenhaber algebra structure on the Hochschild cohomology
$HH^*(S_*(G),S_*(G))$ extends to a Batalin-Vilkovisky algebra.
There is an isomorphism of vector spaces
$HH^*(S_*(G),S_*(G))\cong H^{*+d}(\mathcal LBG)$.}]
\\
\\
{\bf Section 10.} In this section, we define
(Theorem~\ref{string bracket pour les classifiants}) a string bracket on
the $S^1$-equivariant cohomology $H^*_{S^1}(\mathcal{L}BG)$ when $G$ is a group
satisfying the hypotheses of our main theorem. We also define
(Theorem~\ref{string bracket sur la cohomologie cyclique}) a Lie bracket on the cyclic
cohomology $HC^*(S_*(G))$ of $S_*(G)$ when $G$ is a finite group or a connected compact Lie group.
\\
\\
Sections 8, 9 and 10 can be read independently from the rest of the paper. The interested reader might first consider the mod $2$ version of Theorem~\ref{main theoreme homologie groupe de Lie}. Indeed over $\mathbb{F}_2$, there is no sign and orientation issues for integration along the fiber.

%%%%%%%%%%%%%%%%%%%%%%%%%%%%%%%%%%%%%%%%%%%%%%%%%%%%%%%%%%%%%%%%%%%%%%%%%%%%%%%%%%%%%%%%%%%%%%%%%%%%%%%%%%%%%%%%%%%
\section{Wrong way maps}
\vspace{3mm}

An {\it Umkehr map} or wrong way map is a map $f_!$ in homology related to an original continuous map $f:X\rightarrow Y$ which reverses the arrow. Umkehr maps can also be considered in cohomology and some of them are refined to stable maps. 
\\
A typical example is given when one considers a continuous map $f:M^m\rightarrow N^n$ between two oriented closed manifolds. Then using Poincar\'e duality one defines the associated Umkehr, Gysin, wrong way, surprise or transfer map (depending on your prefered name)
$$f_!:H_*(N^n)\rightarrow H_{*+m-n}(M^m).$$
In this paper we will deal with two types of Umkehr maps : transfer and integration along the fiber, both types of Umkehr maps being associated to fibrations. 
\\
In the next two sections we review their constructions and in a third section we give a list of their common properties. We refer the reader to chapter $7$ of J. C. Becker and D. H. Gottlieb's paper \cite{Becker-Gottlieb:history} for a nice survey on Umkehr maps. 
  
\subsection{Transfer maps}
\subsubsection{Transfer for coverings}
Let $p:E\twoheadrightarrow B$ be a covering.
Following~\cite[Beginning of Section 1.3]{Hatcher:algtop}, we don't require that a covering is surjective.
Suppose that all the fibers, $p^{-1}(b)$, $b\in B$ are of finite cardinal.
As pointed by~\cite[p. 100]{Adams:infiniteloop}, ``there is no need to assume that they all have the same cardinal
if $B$ is not connected''.
For example, in~\cite[(4.3.4)]{Adams:infiniteloop}, Adams considers the example of an injective covering $p$ with $1$-point
fibers and $0$-point fibers.
Then one can define a map of spectra~\cite[Construction 4.1.1]{Adams:infiniteloop}
$$\tau_p:\Sigma^{\infty}B_+\longrightarrow \Sigma^{\infty}E_+$$
where $\Sigma^{\infty}X_+$ denotes the suspension spectrum of the topological space $X$ with a disjoint basepoint added. This map induces in singular homology the transfer map :
$$p_!:H_*(B)\rightarrow H_*(E).$$

\subsubsection{Becker-Gottlieb transfer maps}
Let $p:E\rightarrow B$ be a fibration over a path-connected base space $B$.
Up to homotopy, we have an unique fiber.
Suppose that the fiber $F$ of $p$ has the (stable) homotopy type of a finite complex then one has a stable map
$$\tau_p:\Sigma^{\infty}B_+\longrightarrow \Sigma^{\infty}E_+.$$
Originally constructed by Becker and Gottlieb for smooth fiber bundles with compact fibers \cite{Becker-Gottlieb:transferfiberbundle}, they generalize it to fibrations with finite fibres and finite dimensional base spaces~\cite{Becker-Gottlieb:transferfibration}. The finiteness condition on the basis has been removed by M. Clapp using duality theory in the category of ex-spaces \cite{Clapp:duality}. 

\subsubsection{Dwyer's transfer}\label{Dwyer transfer}
 At some point we will need to use W. Dwyer's more general version of the transfer \cite{Dwyer:transfer}. It has the same properties as M. Clapp's transfer (or any other classical version)
and is equal to it in the case of the sphere spectrum $S^0$.
\\
Let us consider a ring spectrum $R$.
By definition, a space $F$ is {\it $R$-small}~\cite[Definition 2.2]{Dwyer:transfer}
if the canonical map of spectra
$$
map(\Sigma^{\infty}F_+,R)\wedge \Sigma^{\infty}F_+
\rightarrow
map(\Sigma^{\infty}F_+,R\wedge\Sigma^{\infty}F_+)
$$
is an equivalence. If $R$ is the sphere spectrum $S^0$, a space $F$ is $R$-small if
$F$ is (stably) homotopy equivalent to a finite CW-complex.
If $R$ is  an Eilenberg-MacLane spectrum $H\mathbb Q$ or $H \mathbb F_p$ or a Morava $K$-theory spectrum $K(n)$, a space $F$ is $R$-small if  $\pi_*(R\wedge \Sigma^{\infty} F_+)$  is finitely generated as a $\pi_*(R)$-module~\cite[Exemple 2.15]{Dwyer:transfer}.
\\
Now let $p:E\twoheadrightarrow B$ be a fibration over a path-connected base $B$ and suppose that
the fiber $F$ is  $R$-small.
Then W. Dwyer has build a transfer map~\cite[Remark 2.5]{Dwyer:transfer}
$$\tau_p:R\wedge\Sigma^{\infty}B_+\rightarrow R\wedge\Sigma^{\infty}E_+.$$
 
\subsubsection{Transfer for non-surjective fibrations}
We would like that the Becker-Gottieb (or Dwyer's) transfer extends the transfer for coverings.
(Recall that a covering is a fibration).

Let $p:E\twoheadrightarrow B$ be a fibration. We don't require that $p$ is surjective.
Suppose that all the fibers, $p^{-1}(b)$, $b\in B$, are (stably) homotopy equivalent to a finite CW-complex.
Then we have a Becker-Gottlieb transfer map
$$\tau_p:\Sigma^{\infty}B_+\longrightarrow \Sigma^{\infty}E_+.$$
\begin{proof}
Let $\alpha\in\pi_0(B)$.
Denote by $B_\alpha$ the path-connected component of $B$ corresponding to $\alpha$.
Let $E_\alpha:=p^{-1}(B_\alpha)$ be the inverse image of $B_\alpha$ by $p$.
Let $p_\alpha:E_\alpha\twoheadrightarrow B_\alpha$ the restriction of $p$ to $E_\alpha$.
Either $p$ is surjective or $E_\alpha$ is the empty set $\emptyset$.
By pull-back, we have the two weak homotopy equivalence
$$
\xymatrix{
\coprod_{\alpha\in\pi_0(B)}E_\alpha
\ar[d]_{\coprod_{\alpha\in\pi_0(B)}p_\alpha}\ar[r]^-{\simeq}
&E\ar[d]^p\\
\coprod_{\alpha\in\pi_0(B)} B_\alpha
\ar[r]^-\simeq
&B
}
$$
Since $p_\alpha:E_\alpha\twoheadrightarrow B_\alpha$ is a fibration over a path-connected basis,
whose fiber is (stably) homotopy equivalent to a finite CW-complex,
we have a Becker-Gottlieb transfer
$$\tau_{p_\alpha}:\Sigma^{\infty}B_{\alpha+}\longrightarrow \Sigma^{\infty}E_{\alpha+}.$$
We define the transfer of $p$ by
$$\tau_{p}:=\vee_{\alpha\in\pi_0(B)}\tau_{p_\alpha}:
\Sigma^{\infty}B_{+}\simeq\vee_{\alpha\in\pi_0(B)}\Sigma^{\infty}B_{\alpha+}
\longrightarrow \vee_{\alpha\in\pi_0(B)}\Sigma^{\infty}E_{\alpha+}\simeq \Sigma^{\infty}E_{+}.$$
In singular homology, we have the linear map of degree $0$
$$p_!:=\oplus_{\alpha\in\pi_0(B)}p_{\alpha!}:
H_*(B)\cong\oplus_{\alpha\in\pi_0(B)}H_*(B_{\alpha})
\longrightarrow \oplus_{\alpha\in\pi_0(B)}H_*(E_{\alpha})\cong H_*(E).$$
\end{proof}
\subsection{Integration along the fiber} 

If we have a smooth  oriented fiber bundle $p:E\rightarrow B$, the integration along the fiber $F$ can be defined at the level of the de Rham cochain complex by integrating differential forms on $E$ along $F$. This defines a map in cohomology 
$$p^!:H^*(E,\mathbb R)\rightarrow H^*(B,\mathbb R)$$ 
this was the very first definition of integration along the fiber and it goes back to A. Lichnerowicz \cite{Lichnerowicz:integrationfibre}. 
\\
We review some well-known generalizations of  this construction, for our purpose we need to work with fibrations over an infinite dimensional basis. As we just need to work with singular homology (at least in this paper), we will use Serre's spectral sequence.% We also review a more geometric construction which can be generalized to an equivariant setting, this equivariant version will help us to compare our constructions with others built by means of equivariant techniques.     

\subsubsection{A spectral sequence version}\label{integration suite spectrale}
(\cite[Section 2]{Gottlieb:bundleEuler}, \cite[Chapter 2, Section 3]{thesedeKitchloo}, \cite[p. 106]{Adams:infiniteloop} or \cite[Section 4.2.3]{Morita:characteristic})
Let $F\hookrightarrow E\stackrel{p}{\twoheadrightarrow} B$ be a fibration over a path-connected base $B$. We suppose that the homology of the fiber $H_*(F,\mathbb{F})$
is concentrated in degree less than $n$ and
has a top non-zero homology group $H_n(F,\mathbb{F})\cong \mathbb{F}$. Let us
assume that the action of the fundamental group $\pi_1(B)$ on $H_n(F,\mathbb{F})$ induced by the holonomy is trivial. 
Let $\omega$ be a generator of $H_n(F,\mathbb{F})$ i.e. an orientation class.
We shall refer to such data as an {\it oriented fibration}. 
\\
Using the Serre spectral sequence, one can define the integration along the fiber as a map
$$p_!:H_*(B)\rightarrow H_{*+n}(E).$$ 
Let us recall the construction, we consider the spectral sequence with local coefficients~\cite{Mimura-Toda:topliegroups}.
%$$E^2_{l,m}=H_l(B,\mathcal H_m(F_b,\mathbb{F}))\Rightarrow H_*(E,\mathbb{F})$$
As the Serre spectral sequence is concentrated under the $n$-th line, the filtration on the abutment
$H_{l+n}(E)$ is of the
form
$$0=F^{-1}=F^0=\cdots=F^{l-1}\subset F^{l}\subset F^{l+1}\subset\cdots\subset  F^{l+n}=H_{l+n}(E).$$
As the local coefficients are trivial by hypothesis,
the orientation class $\omega$ defines an isomorphism of local coefficients
$\tau:\mathbb{F}\rightarrow \mathcal H_n(F_b,\mathbb{F})$. By definition $p_!$ is the composite
$$
p_!:H_l(B,\mathbb{F})\buildrel{H_l(B;\tau)}\over\rightarrow H_l(B,\mathcal H_n(F_b,\mathbb{F}))=E^2_{l,n}\twoheadrightarrow E^\infty_{l,n}=\frac{F^l}{F^{l-1}}=F^l\subset H_{l+n}(E,\mathbb{F}).
$$

\subsection{Properties of Umkehr maps}\label{proprietes umkehr maps}

Let us give a list of properties that are satisfied by transfer maps and integration along the fibers. In fact all reasonable notion of Umkehr map must satisfy this Yoga. We write these properties for integration along the fiber taking into account the degree shifting, we let the reader do the easy translation for transfer maps. 
\\
\\
{\bf Naturality~\cite[p. 29]{thesedeKitchloo}}: Consider a commutative diagram
$$\xymatrix{
E_1\ar[r]^{g}\ar@{>>}[d]^{p_1}
&E_2\ar@{>>}[d]^{p_2}\\
B_1\ar[r]^{h}
&B_2
}$$
where $p_1$ is a fibration over a path-connected base
and $p_2$ equipped with the orientation class $w_2\in H_n(F_2)$ is an oriented
fibration. Let $f:F_1\rightarrow F_2$ the map induced between the fibers.
Suppose that $H_*(f)$ is an isomorphism.
Then the fibration
$p_1$ equipped with the orientation class $w_1:=H_n(f)^{-1}(w_2)$
is an oriented fibration
and the following diagram commutes
$$\xymatrix{
H_{*+n}(E_1)\ar[r]^{H_*(g)}
&H_{*+n}(E_2)\\
H_*(B_1)\ar[r]^{H_*(h)}\ar[u]^{p_1!}
&H_*(B_2)\ar[u]^{p_2!}
}$$
\begin{proof}
With respect to the morphism of groups $\pi_1(h):\pi_1(B_1)\rightarrow \pi_1(B_2)$, $H_n(f):H_n(F_1)\rightarrow H_n(F_2)$ is an isomorphism of $\pi_1(B_1)$-modules.
Since $\pi_1(B_2)$ acts trivially on $H_n(F_2)$,
$\pi_1(B_1)$ acts trivially on $H_n(F_1)$.
By definition of the orientation class $w_1$ and by naturality of the Serre
spectral sequence, the diagram follows. 
\end{proof}
We will apply the naturality property in the following two cases
\\
{\bf Naturality with respect to pull-back:}~\cite[chapter 9.Section 2.Theorem 5 b)]{Spanier:livre}
 Suppose that we have a pull-back,
then $f$ is an homeomorphism.
\\
{\bf Naturality with respect to homotopy equivalences:} Suppose that
$g$ and $h$ are homotopy equivalences, then $f$ is a homotopy equivalence.
\\
\\
{\bf Composition:}
Let $f:X\twoheadrightarrow Y$ be an oriented fibration with path-connected
fiber $F_f$ and orientation class $w_f\in H_m(F_f)$.
Let $g:Y\twoheadrightarrow Z$ be a second
oriented fibration with path-connected
fiber $F_g$ and orientation class $w_g\in H_n(F_g)$.
Then the composite $g\circ f:X\rightarrow Z$
is an oriented fibration with path-connected
fiber $F_{g\circ f}$.
By naturality with respect to pull-back,
we obtain an oriented fibration
$f':F_{g\circ f}\twoheadrightarrow F_g$ with orientation class
$w_f\in H_m(F_f)$.
By definition, the orientation class of $g\circ f$
is
$w_{g\circ f}:=f'_!(w_g)\in H_{n+m}(F_{g\circ f})$.
Then we have the commutative diagram
$$\xymatrix{
& H_{*+n}(Y)\ar[dr]^{f!}\\
H_*(Z)\ar[ur]^{g!}\ar[rr]^{(g\circ f)!}
& & H_{*+m+n}(X).
}$$
{\bf Product:}
Let $p:E\twoheadrightarrow B$ be an oriented fibration with 
fiber $F$ and orientation class $w\in H_m(F)$.
Let $p':E'\twoheadrightarrow B'$ be a second oriented fibration with 
fiber $F'$ and orientation class $w'\in H_n(F')$.
Then if one work with homology with field coefficients,
$p\times p':E\times E'\twoheadrightarrow B\times B'$ is a third oriented fibration with 
fiber $F\times F'$ and orientation class $w\times w'\in 
H_{m+n}(F\times F')$ and 
one has for $a\in  H_*(B)$ and $b\in  H_*(B')$, 
$$(p\otimes p')_!(a\otimes b)=(-1)^{\vert a\vert n}p_!(a)\otimes p'_!(b).$$
Notice that since $p'_!$ is of degree $n$, the sign
$(-1)^{\vert a\vert n}$  agrees with the Koszul rule.
\\
{\bf Borel construction:}
Let $G$ be a topological group acting continuously on two topological spaces $E$ and $B$,
we also suppose that we have a continuous $G$-equivariant map
$$p:E\rightarrow B$$
the induced map on homotopy $G$-quotients (we apply the Borel functor $EG\times_G-$ to $p$) is denoted by
$$p_{hG}:E_{hG}\rightarrow B_{hG}.$$ 

We suppose that the action of $G$ on $B$ has a fixed point $b$
and that $p:E\twoheadrightarrow B$ is an oriented fibration with 
fiber $F:=p^{-1}(b)$ and orientation class $w\in H_n(F)$.
This fiber $F$ is a sub $G$-space of $E$.
Then we suppose that the action of $G$ preserves the orientation, to be more precise we suppose that the action of $\pi_0(G)$ on $H_n(F)$ is trivial. Then $$p_{hG}:E_{hG}\twoheadrightarrow B_{hG}$$
is locally an oriented (Serre) fibration and therefore
is an oriented (Serre)
fibration~\cite[Chapter 2.Section 7.Theorem 13]{Spanier:livre}
 with fiber $F$ and orientation class $w\in H_n(F)$.
Note that under the same hypothesis, $p_!$ is $H_*(G)$-linear (Compare with Lemma~\ref{shriek commute avec delta}).
\subsection{The yoga of correspondences}
Let us finish this section by an easy lemma on Umkehr maps, once again we give it for integration along the fiber and let the reader translate it for transfers. 
\\
We will formulate it in the language of oriented correspondences, the idea of using correspondences in string topology is due to S. Voronov \cite[section 2.3.1]{Cohen-Hess-Voronov:stringtopacyclhom}. We introduce a category of oriented correspondences denoted by ${\bf Corr_{or}}$ :
\\
- the objects of our category will be path connected spaces with the homotopy type of a CW-complex
\\
- $Hom_{\bf Corr_{or}}(X,Y)$ is given by the set of oriented correspondences between $X$ and $Y$ a correspondence will be a sequence of continuous maps :
$$Y\stackrel{r_2}{\leftarrow} Z\stackrel{r_1}{\rightarrow} X$$ 
such that $r_1$ is an oriented fibration. 
\\
Composition of morphisms is given by pull-backs, if we consider two correspondences
$$Y\stackrel{r_2}{\leftarrow} Z\stackrel{r_1}{\rightarrow} X$$ 
and 
$$U\stackrel{r'_2}{\leftarrow} T\stackrel{r'_1}{\rightarrow} Y$$
then the composition of the correspondences is 
$$U\stackrel{r"_2}{\leftarrow} T\times_Y Z\stackrel{r"_1}{\rightarrow} X.$$
We have to be a little bit more careful in the definition of morphisms, composition as defined above is not strictly associative (we let the reader fix the details). 
\\
\\

\begin{lem}\label{composition lemma}
{\bf Composition lemma.} {\it The singular homology with coefficients in a field defines a symmetric monoidal functor 
$$\mathbb H(-,\mathbb F):{\bf Corr_{or}}\longrightarrow \mathbb F-{\bf vspaces}$$
to a morphism $Y\stackrel{r_2}{\leftarrow} Z\stackrel{r_1}{\rightarrow} X$ it associates $(r_2)_*\circ(r_1)_!:H_*(X,\mathbb F)\rightarrow H_{*+d}(Y,\mathbb F)$ (where $d$ is the homological dimension of the fiber of $r_1$).}
\end{lem}
\begin{proof}
The fact that the functor is monoidal follows from the product property of the integration along the fiber.
\\
Let $c$ and $c'$ be two oriented correspondences the fact that $\mathbb H(c'\circ c,\mathbb F)=\mathbb H(c',\mathbb F)\circ\mathbb H(c,\mathbb F)$ follows from an easy inspection of the following diagram
$$\xymatrix{&&X\\
&\ar[ld]_{r"_2}\ar[d]_{f_2}T\times_Y Z\ar[ru]^{r"_1}\ar[r]^{f_1}& Z\ar[d]^{r_2}\ar[u]_{r_1}\\
U&\ar[l]^{r'_2}T\ar[r]_{r'_1}& Y}$$
in fact we have $(r'_2)\circ(r'_1)_!\circ(r_2)_*\circ(r_1)_!=(r'_2)_*\circ(f_2)_*\circ(f_1)_!\circ (r_1)_!$ by the naturality property and by the composition property we have $(f_1)_!\circ(r_1)_!=(r"_1)_!$.
\end{proof}

%%%%%%%%%%%%%%%%%%%%%%%%%%%%%%%%%%%%%%%%%%%%%%%%%%%%%%%%%%%%%%%%%%%%%%%%%%%%%%%%%%%%%%%%%%%%%%%%%%%%%%%%%%%%%%%%%%%%%%%%

\section{Props and field theories}
The aim of this section is to introduce the algebraic notions that encompass the "stringy" operations acting on $\mathcal LBG$. This section is mainly expository.
\subsection{props} We use props and algebras over them as a nice algebraic framework in order to deal with $2$-dimensional field theories. We could in this framework use the classical tools from algebra and homological algebra exactly as for algebras over operads.  

\begin{defin}~\cite[Definition 54]{Markl:operadprop}
 A {\it prop} is a symmetric (strict) monoidal category $\mathcal P$
~\cite[3.2.4]{Kock:Frob2TQFT} whose set of objects is identified with the set $\mathbb Z_+$ of nonnegative numbers. The tensor law on objects should be given by addition of integers $p\otimes q=p+q$. Strict monoidal means that 
the associativity and neutral conditions are the identity. 
\end{defin}
We thus have two composition products on morphisms a horizontal one given by the tensor law :
$$-\otimes -:\mathcal P(p,q)\otimes \mathcal P(p',q')\longrightarrow \mathcal P(p+p',q+q'),$$
and a vertical one given by composition of morphisms :
$$-\circ-:\mathcal P(q,r)\otimes \mathcal P(p,q)\longrightarrow \mathcal P(p,r).$$  
\begin{ex} Let $V$ be a fixed vector space.
A fundamental example of prop is given by the {\it endomorphims prop of $V$} denoted $\mathcal End_V$. The set of morphisms is defined as $\mathcal End_V(p,q)=Hom(V^{\otimes p},V^{\otimes q})$. The horizontal composition product is just the tensor while the vertical is the composition of morphisms.
\end{ex}
A morphism of props is a symmetric (strict) monoidal
functor~\cite[3.2.48]{Kock:Frob2TQFT} $F$ such that $F(1)=1$.
Let $\mathcal P$ be a linear prop i.e. we suppose that $\mathcal P$ is enriched in the category of vector spaces (graded, differential graded) and $V$ be a vector space (graded, differential graded,....). 
\begin{defin}~\cite[Definition 56]{Markl:operadprop}
The vector space $V$ is said to be a
{\it $\mathcal P$-algebra} if there is a morphism of linear props
$$F:\mathcal P\longrightarrow \mathcal End_V.$$
\end{defin}
This means that we have a 
a family of linear morphisms $$F:\mathcal{P}(m,n)
\rightarrow\text{Hom}(V^{\otimes m},V^{\otimes n})$$
such that
\begin{itemize}
%\item $F(n)=V^{\otimes n}$.
\item[(monoidal)]$
F(f\otimes g)=
F(f)\otimes F(g)
$
for $f\in \mathcal P(m,n)$ and $g\in \mathcal P(m',n')$.
\item[(identity)] The image $F(id_n)$
of the identity morphism $id_n\in\mathcal{P}(n,n)$,
is equal to the identity morphism of $V^{\otimes n}$.
\item[(symmetry)]$
F(\tau_{m,n})=\tau_{V^{\otimes m},V^{\otimes n}}
$.
Here $\tau_{m,n}:m\otimes n\rightarrow n\otimes m$ denotes the natural
twist isomorphism of $\mathcal{P}$.
And for any graded vector spaces $V$ and $W$,
$\tau_{V,W}$ is the isomorphism $\tau_{V,W}:V\otimes W\rightarrow W\otimes V$,
$v\otimes w\mapsto (-1)^{\vert v\vert\vert w\vert}w\otimes v$.
\item[(composition)]
$
F(g\circ f)=F(g)\circ F(f)
$
for $f\in \mathcal P(p,q)$ and $g\in \mathcal P(q,r)$.
\end{itemize}

By adjunction, this morphism determines evaluation products 
$$\mu:\mathcal P(p,q)\otimes V^{\otimes p}\longrightarrow V^{\otimes q}.$$

{\bf Remarks.} One can also notice that the normalized singular chain functor sends topological props to props in the category of differential graded modules. If the homology of a topological prop is also a prop, this is not always the case for algebras over props (because props also encode coproducts), one has to consider homology with coefficients in a field. This is the main reason, why in this
paper, we have chosen to work over an arbitrary field $\mathbb{F}$.

\subsection{Topological Quantum Field Theories}
Let $F_1$ and $F_2$ be two smooth closed oriented $n$-dimensional manifolds not necessarily path-connected.
\begin{defin}~\cite[p. 201]{Milnor-Stasheff}\label{oriented cobordism}
An {\it oriented cobordism} from $F_1$ to $F_2$ is a $n$-dimensional smooth compact oriented manifold $F$
not necessarily path-connected  with boundary $\partial F$, equipped with an orientation preserving
diffeomorphism $\varphi$ from the disjoint union $-F_1\coprod F_2$ to $\partial F$
(The orientation on $\partial F$ being the one induced by the orientation of $F$).
\end{defin}
We call {\it in-coming boundary} map $in:F_1\hookrightarrow F$ the composite of the restriction of $\varphi$ to $F_1$ and of
the inclusion map of $\partial F$ into $F$.
We call {\it out-coming boundary} map $out:F_2\hookrightarrow F$ the composite of the restriction of $\varphi$ to $F_2$ and of
the inclusion map of $\partial F$ into $F$.
\\
Let $F$ and $F'$ be two oriented  cobordisms from $F_1$ to $F_2$.
\begin{defin}~\cite[1.2.17]{Kock:Frob2TQFT}\label{cobordisms equivalents}
We say that $F$ and $F'$ are {\it equivalent} if there is an orientation preserving diffeomorphism $\phi$ from $F$ to $F'$
such that the following diagram commutes
$$\xymatrix{
& F\ar[dd]^{\phi}_\cong\\
F_1\ar[ru]^{in}\ar[rd]_{in}
&&F_2\ar[lu]_{out}\ar[ld]^{out}\\
&F'
}$$
\end{defin}
\begin{defin}~\cite[1.3.20]{Kock:Frob2TQFT}
The category of oriented cobordisms $n-Cob$ is the discrete category whose objets are smooth closed oriented,
 not necessarily path-connected, $n-1$-dimensional manifolds. The morphisms from $F_1$ to $F_2$ are the set
of equivalent classes of oriented cobordisms from $F_1$ to $F_2$.
Composition is given by gluing cobordisms.
Disjoint union gives $n$-cob a structure of symmetric monoidal category.
\end{defin}
\begin{defin}~\cite[1.3.32]{Kock:Frob2TQFT}
A $n$-dimensional Topological Quantum Field Theory ($n$-TQFT) as axiomatized by Atiyah is a symmetric monoidal
functor from the category of oriented cobordisms  $n$-Cob to the category of vector spaces.
\end{defin}
\subsection{The linear prop defining 2-TQFTs $\mathbb{F}[sk(2-Cob)]$}\label{prop des tqfts}
We now restrict to $n=2$, i. e. to $2$-TQFTs.
By~\cite[1.4.9]{Kock:Frob2TQFT}, the skeleton of the category $2-Cob$, $sk(2-Cob)$,
 is the full subcategory of  $2-Cob$ whose objects are disjoint union of circles $\coprod_{i=1}^n S^1$, $n\geq 0$.
Therefore $sk(2-Cob)$ is a discrete prop.
For any set $X$, denote by $\mathbb{F}[X]$ the free vector space with basis $X$.
By applying the functor  $\mathbb{F}[-]$ to $sk(2-Cob)$, we obtain a linear
prop $\mathbb{F}[sk(2-Cob)]$. And now a vector space $V$ is a $2$-Topological Quantum Field Theory if and only if $V$ is an algebra over this prop
$\mathbb{F}[sk(2-Cob)]$.

\subsection{Segal prop of Riemann surfaces $\mathfrak M$}(~\cite[Example 2.1.2]{Cohen-Hess-Voronov:stringtopacyclhom}, compare with~\cite[p. 24 and p. 207]{Markl-Shnider-Stasheff:opeatp})\label{la prop de Segal}
Let $p$ and $q\geq 0$.

A ``complex cobordism'' from the disjoint union $\coprod_{i=1}^p S^1$ to
  $\coprod_{i=1}^q S^1$ is a closed complex curve $F$, not necessarily
path connected equipped with two holomorphic embeddings of
disjoint union of closed disks into $F$,
$in:\coprod_{i=1}^p D^2\hookrightarrow F$
and $out:\coprod_{i=1}^q D^2\hookrightarrow F$.

Let $F$ and $F'$ be two ``complex cobordisms'' from $\coprod_{i=1}^p S^1$
to $\coprod_{i=1}^q S^1$.

We say that $F$ and $F'$ are {\it equivalent} if there is a biholomorphic map
$\phi$ from $F$ to $F'$
such that the following diagram commutes
$$\xymatrix{
& F\ar[dd]^{\phi}_\cong\\
\coprod_{i=1}^p D^2\ar[ru]^{in}\ar[rd]_{in}
&&\coprod_{i=1}^q  D^2\ar[lu]_{out}\ar[ld]^{out}\\
&F'
}$$
The Segal prop $\mathfrak M$ is the topological category
whose objects are disjoint union of circles $\coprod_{i=1}^n S^1$, $n\geq 0$,
identified with non-negative numbers.
The set of morphisms from $p$ to $q$, denoted $\mathfrak M(p,q)$, is
the set of equivalent classes of ``complex cobordisms'' from $\coprod_{i=1}^p S^1$
to $\coprod_{i=1}^q S^1$. The moduli space $\mathfrak M(p,q)$ is equipped
with a topology difficult to define~\cite[p. 207]{Markl-Shnider-Stasheff:opeatp}.

By applying the singular homology functor with coefficients in a field, $H_*(-)$ to the Segal topological prop
 $\mathfrak M$, one gets a graded linear prop $H_*(\mathfrak M)$.
Explicitly 
$$H_*(\mathfrak M)(p,q):=H_*(\mathfrak M(p,q)).$$
\begin{defin}~\cite[3.1.2]{Cohen-Hess-Voronov:stringtopacyclhom}\label{unital counital HCFT}
A graded vector space $V$ is a {\it (unital counital) homological conformal field theory} or HCFT for short if
$V$ is an algebra over the graded linear prop $H_*(\mathfrak M)$.
\end{defin}
\subsection{The props isomorphism $\pi_0(\mathfrak{M})\cong sk(2-Cob)$}\label{composantes connexes de la Segal prop}
Let
$$\coprod_{i=1}^p D^2\buildrel{in}\over\hookrightarrow F\buildrel{out}\over\hookleftarrow\coprod_{i=1}^q D^2$$ be a ``complex cobordism'' from $\coprod_{i=1}^p S^1$
to $\coprod_{i=1}^q S^1$.
By forgetting the complex structure and removing the interior of the $p+q$ disks,
we obtain an oriented cobordism $F-\coprod_{i=1}^{p+q} Int D^2$ from  $-\coprod_{i=1}^p S^1$
to $\coprod_{i=1}^q S^1$.
Indeed the restrictions of $in$ to  $\coprod_{i=1}^p S^1$, $in_{|\coprod_{p} S^1}$,  and the restriction of $out$ to
$\coprod_{i=1}^p S^1$, $out_{|\coprod_{q} S^1}$, are orientation preserving diffeomorphims.
By composing $in_{|\coprod_{p} S^1}$ with a reversing orientation diffeomorphism, $F-\coprod_{i=1}^{p+q} Int D^2$ becomes an  oriented cobordism from  $\coprod_{i=1}^p S^1$
to $\coprod_{i=1}^q S^1$.
Therefore, we have defined a morphism of props
$Forget:\mathfrak{M}\rightarrow sk(2-Cob)$.

The topology on the moduli spaces $\mathfrak{M}(p,q)$ is defined such that
this morphism of props $Forget$ is continuous and can be identified with the canonical
surjective morphism of topological props
$\mathfrak{M}\twoheadrightarrow\pi_0(\mathfrak{M})$ from the Segap prop
to the discrete prop obtained by taking its path components.

\subsection{Tillmann prop.}\label{Tillmann prop}
 Following U. Tillman's topological approach to the study of the Moduli spaces of complex curves we introduce a topological prop
$\propBD$ homotopy equivalent to a sub prop of Segal prop of Riemann surfaces $\mathfrak M$.

%\subsubsection{The prop of surfaces}
For $p$ and $q$, consider the groupoid $\mathcal{E}(p,q)$~\cite[p. 69]{thesedewahl}. An object of $\mathcal{E}(p,q)$ is an oriented cobordism $F$ from $\amalg_{i=1}^p S^1$ to $\amalg_{i=1}^p S^1$
(Definition~\ref{oriented cobordism}).
The set of morphisms from  $F_1$ to $F_2$, is
$$\text{Hom}_{\mathcal{E}(p,q)}(F_1,F_2):=\pi_0 Diff^+(F_1;F_2; \partial)$$
the connected components of orientation preserving diffeomorphisms $\phi$
that fix the boundaries of $F_1$ and $F_2$ pointwise:
$\phi\circ in=in$ and $\phi\circ out=out$.
Remark that if $F_1=F_2$ is a connected surface $F_{g,p+q}$,
then $\text{Hom}_{\mathcal{E}(p,q)}(F_{g,p+q},F_{g,p+q})$
is the mapping class group
$$
\Gamma_{g,p+q}:=\pi_0 Diff^+(F_{g,p+q}; \partial).
$$

 In \cite{Tillmann:htpysmcg} U. Tillman studies the homotopy type of a surface $2$-category
$\deuxcat$.
Roughly speaking, the objects of $\deuxcat$ are natural numbers representing circles.
The enriched set of
morphisms from $p$ to $q$ is the category $\mathcal{E}(p,q)$.
The composition in  $\deuxcat$ is the functor induced by gluing
  $\mathcal{E}(q,r)\times\mathcal{E}(p,q)\rightarrow\mathcal{E}(p,r)$.
U. Tillmann's surface category  $\deuxcat$
has the virtue to be a symmetric strict monoidal 2-category where the tensor product is given by disjoint union of cobordisms.

Let $\mathcal{C}$ be a small category.
The nerve of $\mathcal{C}$ is a simplicial set $N(\mathcal{C})$.
The classifying space of $\mathcal{C}$, $B(\mathcal{C})$, is by definition
the geometric realization of this simplicial set, $N(\mathcal{C})$.
Applying the classifying space functor to $\deuxcat$,
one gets a topological symmetric (strict) monoidal category.
We thus have a topological prop $\mathcal B\deuxcat$.
By definition, $\mathcal B\deuxcat(p,q):=B(\mathcal{E}(p,q))$. 

These categories and their higher dimensional analogues have been studied extensively because of their fundamental relationship with conformal field theories and Mumford's conjecture~\cite{galatiusmadsentillmannweiss,Madsen-Tillmann:stablemapping}.

Recall that if $\mathcal{C}$ is a groupoid, then $\pi_0(B\mathcal{C})$
is the set of isomorphisms classes of $\mathcal{C}$.
Therefore the skeleton of the category of oriented cobordisms, $sk(2-Cob(p,q))$ (Section~\ref{prop des tqfts}),
is exactly the discrete prop obtained by taking the path-components
of the topological prop $\mathcal B\deuxcat$ (Compare with Section~\ref{composantes connexes de la Segal prop}).
In particular,
$$sk(2-Cob(p,q)):=\pi_0(\mathcal B\deuxcat(p,q))=\pi_0(B(\mathcal{E}(p,q)).$$
By considering the skeleton of a groupoid $\mathcal{C}$,
we have the homotopy equivalence
$$
\coprod_x B(\text{Hom}(x,x))
\buildrel{\thickapprox}\over\rightarrow B\mathcal{C}.
$$
where the disjoint union is taken over a set of representatives $x$
of isomorphism classes in $\mathcal{C}$
and  $B(\text{Hom}(x,x))$ is the classifying space of the discrete group
$\text{Hom}(x,x)$~\cite[5.10]{Dwyer-Henn:barcelone}.
Therefore the morphism spaces $B\deuxcat(p,q)$ have a connected component
for each oriented cobordism class $F$ (Definition~\ref{cobordisms equivalents}).
The connected component corresponding to $F$ has the homotopy type
of $B(\pi_0(Diff^+(F,\partial)))$~\cite[p. 264]{Tillmann:htpysmcg}.

The cobordism $F_{g,p+q}$ is the disjoint union of its path components:
$$
F_{g,p+q}\cong F_{g_1,p_1+q_1}\amalg\dots\amalg F_{g_k,p_k+q_k}.
$$
Here $F_{g_i,p_i+q_i}$ denotes a surface of genus $g_i$ with $p_i$ incoming and $q_i$ outgoing circles target.
We have $g=\sum_i g_i$, $p=\sum_i p_i$ and $q=\sum_i q_i$.
We suppose that each path component $F_{g_i,p_i+q_i}$ has at least
one boundary component.
That is, we suppose that $\forall 1\leq i\leq k$, $p_i+q_i\geq 1$.
Since a diffeomorphism fixing the boundaries pointwise cannot exchange
the path-components of $F$, we have the isomorphism of topological groups
$$
Diff^+(F; \partial)\cong
Diff^+(F_{g_1,p_1+q_1};\partial)\times\dots\times
Diff^+(F_{g_k,p_k+q_k}; \partial)
$$
Again since $p_i+q_i\geq 1$, by~\cite{Earle-Schatz},
the canonical surjection from
$$Diff^+(F_{g_i,p_i+q_i};\partial)
\buildrel{\thickapprox}\over\longrightarrow
\Gamma_{g_i,p_i+q_i}:=\pi_0 Diff^+(F_{g_i,p_i+q_i}; \partial).
$$
with the discrete topology is a homotopy equivalence.
Therefore the canonical surjection:
$$Diff^+(F;\partial)
\buildrel{\thickapprox}\over\longrightarrow
\pi_0 Diff^+(F; \partial)
$$
is also a homotopy equivalence: Earle and Schatz result~\cite{Earle-Schatz} extends to
a non-connected surface $F$ if each component has at least one boundary
component.
Therefore we have partially recovered the following proposition:

\begin{proposition}~\cite[3.2]{TillmannICM02}\label{equivalence des trois props}
For $p\in\mathbb{Z}^+$ and $q\in\mathbb{Z}^+$,
define the collection of topological spaces
$$\propBD(p,q):=\coprod_{F_{p+q}} BDiff^+(F; \partial).$$
Here the disjoint union is taken over a set of representatives $F_{p+q}$
of the oriented cobordism classes from $\amalg_{i=1}^p S^1$ to
$\amalg_{i=1}^q S^1$ (Definition~\ref{cobordisms equivalents}).
This collection $\propBD$
of spaces forms a topological prop up to homotopy (it is a prop in the homotopy category of spaces).
If we consider only cobordisms $F$ whose path components  have
at least one outgoing-boundary component, i. e. $q_i\geq 1$,
(this is the technical condition of~\cite[p. 263]{Tillmann:htpysmcg}),
the resulting three sub topological props of $\propBD$, Tillmann's prop $B\deuxcat$ and
Segal prop $\mathfrak{M}$ are all homotopy equivalent.
\end{proposition}
\subsection{Non-unital and non-counital homological conformal field theory}\label{unital counital ou pas}
\begin{defin}(Compare with definition~\ref{unital counital HCFT})
A graded vector space $V$ is a {\it unital non-counital} homological conformal field theory if
$V$ is an algebra over the graded linear prop obtained by applying singular homology
to one of the three sub topological props defined in the previous Proposition.
\end{defin}
In~\cite{Godin:higherstring}, Godin uses the term homological conformal field theory {\it with positive boundary}
instead of unital non-counital.

If instead, we consider only cobordisms $F$ whose  path components  have
at least one in-boundary component, i. e. $p_i\geq 1$, we will say that we have a
 {\it counital non-unital} homological conformal field theory.

In this paper, we will deal mainly with  {\it non-unital non-counital}
homological conformal field theory: this is when we consider only cobordisms $F$
whose path components  have at least
one in-boundary component and also at least one outgoing-boundary component,
i. e. $\forall 1\leq i\leq k$, $p_i\geq 1$ and $q_i\geq 1$.
%%%%%%%%%%%%%%%%%%%%%%%%%%%%%%%%%%%%%%%%%%%%%%%%%%%%%%%%%%%%%%%%%%%%%%%%%%%%%%%%%%%%%%%%%%%%%%%%%

\section{Definition of the operations}\label{Definition des evaluation products}

The goal of this section is to define the evaluation products of the propic action of the homology of the Segal prop.

\subsection{The definition assuming Propositions~\ref{cofibre de in}, \ref{action du pi1 sur in est trivial} 
and Theorem~\ref{action du mapping class group est trivial}}
Let $F_{g,p+q}$ be an oriented cobordism from $\amalg_{i=1}^p S^1$ to
  $\amalg_{i=1}^q S^1$.
Let $k$ be its number of path components and let $g$ be the sum of the genera of its path
components.
Let $$\chi(F)=2k-2g-p-q$$ 
be its Euler characteristic.
Let $Diff^+(F;\partial)$ be the group of orientation preserving diffeomorphisms
that fix the boundaries pointwise.

\begin{defin}\label{definition evaluation product}
Let $X$ be a simply-connected space such that its pointed loop homology $H_*(\Omega X)$
is a finite dimensional vector space.
Denote by $d$ the top degree such that $H_d(\Omega X)$ is not zero.

Suppose that every path component of  $F_{g,p+q}$ has at least one in-boundary component
and at least one outgoing-boundary component.
Then we can define the {\it evaluation product} associated to $F_{g,p+q}$.
$$
\mu(F):H_*(BDiff^+(F;\partial))\otimes H_*(\mathcal{L}X)^{\otimes p}\rightarrow H_*(\mathcal{L}X)^{\otimes q}.
$$  
It is a linear map of degre $-d\chi(F)$.
\end{defin}
\begin{proof}[Proof of Definition~\ref{definition evaluation product}]
Let $F_{g,p+q}$ be an oriented cobordism (not necessarily path-connected) with $p+q$ boundary components equipped with a given ingoing map
$$in:\partial_{in} F_{g,p+q}=\amalg_{i=1}^p S^1\hookrightarrow F_{g,p+q}$$
and an outgoing map 
$$out:\partial_{out}F_{g,p+q}=\amalg_{i=1}^q S^1\hookrightarrow F_{g,p+q}.$$
These two maps are cofibrations.
The cobordism $F_{g,p+q}$ is the disjoint union of its path components:
$$
F_{g,p+q}\cong F_{g_1,p_1+q_1}\amalg\dots\amalg F_{g_k,p_k+q_k}.
$$
Recall that we suppose
that $\forall 1\leq i\leq k$, $p_i\geq 1$ and $q_i\geq 1$.
Therefore, by Proposition~\ref{cofibre de in}, one obtain that the cofibre of $in$, $F/\partial_{in} F$,
is homotopy equivalent to the wedge
$\vee_{-\chi(F)} S^1$ of $-\chi(F)=2g+p+q-2k$ circles.
When we apply the mapping space $map(-,X)$ to the cofibrations $in$ and $out$, one gets the two fibrations
$$map(in,X):map(F_{g,p+q},X)\twoheadrightarrow map(\partial_{in}F_{g,p+q},X)\cong \mathcal LX^{\times p}$$
and
$$map(out,X):map(F_{g,p+q},X)\twoheadrightarrow map(\partial_{out}F_{g,p+q},X)\cong \mathcal LX^{\times q}.$$
The fiber of the continuous map $map(in,X)$ is the pointed mapping space $map_*(F/\partial_{in} F,X)$
and is therefore homotopy equivalent to the product of pointed loop spaces $\Omega X^{-\chi(F)}$.

Since $H_*(\Omega X)$ is a Hopf algebra and is finite dimensional, by~\cite[Proof of Corollary 5.1.6 2)]{Sweedler:livre}, $H_*(\Omega X)$ is a Frobenius algebra: i. e. there exists an isomorphism 
 $$H_*(\Omega X)\buildrel{\cong}\over\rightarrow H_*(\Omega X)^\vee\cong  H^*(\Omega X)$$
of left $H_*(\Omega X)$-modules.
Since $H_*(\Omega X)$ is concentrated in degre between $0$ and $d$, and $H_0(\Omega X)$ and  $H_d(\Omega X)$
are not trivial vector spaces, this isomorphism is of lower degre $-d$:
 $$H_p(\Omega X)\buildrel{\cong}\over\rightarrow H_{d-p}(\Omega X)^\vee\cong  H^{d-p}(\Omega X).$$
In particular, since $X$ is simply connected,
$H_d(\Omega X)\cong  H_0(\Omega X)^\vee\cong  H^0(\Omega X)\cong\mathbb{F}$
is of dimension $1$.

Therefore the homology of the fibre of $map(in,X)$ is concentrated in  degre less or equal than $-d\chi(F)$
and $H_{-d\chi(F)}(map_*(F/\partial_{in} F,X))\cong\mathbb{F}$ is also of dimension $1$.

By Proposition~\ref{action du pi1 sur in est trivial}, we have that in fact
$map(in,X):map(F_{g,p+q},X)\twoheadrightarrow \mathcal LX^{\times p}$
is an oriented fibration. 

We set $D_{g,p+q}:=Diff^+(F_{g,p+q},\partial)$.
We also denote the Borel construction
$$\mathcal M_{g,p+q}(X):=(map(F_{g,p+q},X))_{hD_{g,p+q}}.$$
Applying the Borel construction $(-)_{hD_{g,p+q}}$ to the
fibrations $map(in,X)$ and $map(out,X)$ yields the following two fibrations
$$\rho_{in}:=map(in,X)_{hD_{g,p+q}}:\mathcal M_{g,p+q}(X){\longrightarrow}BD_{g,p+q}\times\mathcal LX^{\times p}.$$
     $$\rho_{out}:\mathcal M_{g,p+q}(X)\stackrel{map(out,X)_{hD_{g,p+q}}}{\longrightarrow}BD_{g,p+q}\times\mathcal LX^{\times q}
\buildrel{proj_2}\over\longrightarrow LX^{\times q}.$$
Here $proj_2$ is the projection on the second factor.

By Theorem~\ref{action du mapping class group est trivial}, $\pi_0(D_{g,p+q})$ acts trivially on
$H_{-d\chi(F)}(map_*(F/\partial_{in} F,X))$.
Under this condition, the Borel construction $(-)_{hD_{g,p+q}}$
preserves oriented (Serre) fibration.
Therefore $\rho_{in}$ is also an oriented fibration with fibre  $map_*(F/\partial_{in} F,X)$.
After choosing an orientation class (Proposition~\ref{choix classe orientation})
$$\omega_{F}\in H_{-d\chi(F)}(map_*(F/\partial_{in} F,X)),$$
we have a well defined integration along the fiber map for $\rho_{int}$:
$$\rho_{in}!:H_*(BD_{g,p+q}\times \mathcal LX^{\times p})
\longrightarrow H_{*-d\chi_F}(\mathcal M_{q,p+p}(X)).$$
By composing with $H_*(\rho_{out}):H_*(\mathcal M_{g,p+q}(X)){\longrightarrow}H_*(\mathcal LX^{\times q})$, one gets a map
$$\mu(F_{g,p+q}):H_l(BD_{g,p+q})\otimes H_{m_1}({\mathcal LX})\otimes\cdots
\otimes H_{m_p}(\mathcal LX)\rightarrow H_{l+m_1+\cdots+m_p-d\chi_F}({\mathcal LX}^{\times q}).$$
As we restrict ourself either
to homology with coefficients in a field, one
finally gets an evaluation product of degree $-d\chi_F=d(2g+p+q-2k)$
\\
\begin{center}
\begin{tabular}{|c|}
\hline 
$\mu(F_{g,p+q}):H_*(BD_{g,p+q})\otimes H_*({\mathcal LX})^{\otimes p}
\rightarrow H_*({\mathcal LX})^{\otimes q}$\tabularnewline
\hline
\end{tabular}
\end{center}
%They satisfy $\mu(F_{g,p+q},\theta)=(-1)^d\mu(F_{g,p+q},\theta^{op})$.
\end{proof}

\subsection{Orientability of the fibration $map(in,X):map(F,X)\twoheadrightarrow\mathcal{L}X^{\times p}$}

\begin{proposition}\label{cofibre de in}
Let $F_{g,p+q}$ be the path-connected compact oriented surface of genus $g$ with $p$
incoming boundary circles and $q$ outgoing boundary circles.
Denote by $F/\partial_{in} F$, the cofibre of
$$in:\partial_{in} F_{g,p+q}=\amalg_{i=1}^p S^1\hookrightarrow F_{g,p+q}.$$
If $p\geq 1$ and $q\geq 1$ then  $F/\partial_{in} F$ is homotopy equivalent to the wedge
$\vee_{-\chi(F)} S^1$ of $-\chi(F)=2g+p+q-2$ circles.
\end{proposition}
\begin{proof}
We first show that  $F/\partial_{in} F$ and  the wedge of $-\chi(F)$ circles have the same
homology.
If  $p\geq 1$ then $\tilde{H}_0(F/\partial_{in} F)=\{0\}$.
By excision and Poincar\'e duality~\cite[Theorem 3.43]{Hatcher:algtop}, if $q\geq 1$,

 $$\tilde{H}_2(F/\partial_{in} F)\cong H_2(F,\partial_{in} F)\cong H^0(F,\partial_{out} F)\cong
\tilde{H}^0(F/\partial_{out} F)=\{0\}.$$
Using the long exact sequence associated to the pair $(\partial_{in} F, F)$, we obtain
that $\chi(H_*(F))=\chi(H_*(\partial_{in} F))+\chi(H_*(F,\partial_{in} F))$.
Therefore the Euler characteristics  $\chi(H_*(F))$ of  $H_*(F)$ is equal to
the Euler characteristics  $\chi(H_*(F,\partial_{in} F))$ of  $H_*(F,\partial_{in} F)$.
Therefore $\tilde{H}_1(F/\partial_{in} F)=\tilde{H}_*(F/\partial_{in} F)$ is of dimension
$-\chi(F)$.

We show that $F/\partial_{in} F$ is homotopy equivalent to a wedge of circles.
The surface $F_{g,n-1}$, $n\geq 1$, can be constructed from a full polygon with
$4g+n-1$ sides by identifying pairs of the $4g$ edges~\cite[p. 5]{Hatcher:algtop}.
Therefore the surface with one more boundary component, $F_{g,n}$, $n\geq 1$, is the mapping cylinder
of the attaching map from $S^1$ to $F_{g,n-1}^{(1)}$, the $1$-skeleton of $F_{g,n-1}$ (in the case $n=1$, see~\cite[Example 1B.14]{Hatcher:algtop}). Therefore any surface with boundary deformation retracts onto a connected graph~\cite[Example 1B.2]{Hatcher:algtop}.

By filling the $p$ incoming circles $S^1$ by $p$ incoming disks $D^2$, we have the push-out
$$\xymatrix{
\coprod_{i=1}^p S^1 \ar^{in}[r]\ar[d]
& F_{g,p+q}\ar[d]\\
\coprod_{i=1}^p D^2\ar[r] & F_{g,0+q}.
}
$$
So we obtain that $$\frac{F_{g,p+q}}{\coprod_{i=1}^p S^1}\cong \frac{F_{g,0+q}}{\coprod_{i=1}^p D^2}$$
is homotopy equivalent to the mapping cone of an application from $p$ distinct points to a path-connected graph.
Therefore $\frac{F_{g,p+q}}{\coprod_{i=1}^p S^1}$ is homotopy equivalent to a path-connected graph and therefore
to a wedge of circles~\cite[Example 0.7]{Hatcher:algtop}.
\end{proof}
Now we prove that the fibration is oriented by studying the action of the fundamental group of the basis on the homology of the fiber.
\begin{proposition}\label{action du pi1 sur in est trivial}
Let $F_{g,p+q}$ be a path-connected cobordism from
$\coprod_{i=1}^p S^1$ to $\coprod_{i=1}^q S^1$.
Let $X$ be a simply connected space such that $H_*(\Omega X,\mathbb F)$ is finite dimensional.
If $p\geq 1$ and $q\geq 1$ then the fibration obtained by restriction to the in-boundary
components 
$$map(in,X):map(F_{g,p+q},X)\twoheadrightarrow \mathcal LX^{\times p}$$ 
is $H_*(-,\mathbb F)$-oriented.
\end{proposition}
\begin{proof}
The key for the proof of Proposition~\ref{cofibre de in} was that
a surface $F_{g,p+q}$ with boundary is homotopy equivalent to a graph.
For the proof of this Proposition, we are more precise: we use a special kind
of graphs, called the Sullivan Chord diagrams.
If $p\geq 1$ and $q\geq 1$ then
the surface $F_{g,p+q}$ is homotopy equivalent to a Sullivan's chord diagram $c_{g,p+q}$ of type $(g;p,q)$. 

A Sullivan's Chord diagram $c_{g,p+q}$ of type $(g;p,q)$~\cite[Definition 2]{1095.55006} consists of an union of $p$ disjoint circles together with the disjoint union
of path-connected trees. The endpoints of the trees are joined at
distincts points to the $p$ circles.

Denote by $\sigma(c)$ this set of path-connected trees.
The endpoints of the trees lying on the $p$ circles are called the circular vertices and the set of circular vertices of  $c_{g,p+q}$ is denoted $v(c)$.
For each tree $v\in\sigma(c)$, denoted by $\mu(v)$ the endpoints of the tree $v$
which are on the $p$ circles. We have the disjoint union
$v(c)=\coprod_{v\in\sigma(c)} \mu(v)$.
Therefore, with the notations introduced
(that follows the notations of Cohen and
Godin~\cite[Section 2]{1095.55006}), we have the push-out
$$
\xymatrix{
v(c)=\coprod_{v\in\sigma(c)} \mu(v)\ar[r]\ar[d]
& \coprod_{i=1}^p S^1\ar[d]^{in}\\
\coprod_{v\in\sigma(c)} v\ar[r]
& c_{g,p+q}
}
$$
Up to a homotopy equivalence between  $F_{g,p+q}$ and $c_{g,p+q}$,
the $p$ circles in $c_{g,p+q}$ represent the incoming boundary components of $F_{g,p+q}$, i. e.
we have the commutative triangle
$$\xymatrix{
\coprod_{i=1}^p S^1\ar[rr]^{in}\ar[dr]_{in}
&& c_{g,p+q} \\
& F_{g,p+q}\ar[ur]_{\thickapprox}
}
$$
Let $\#\mu(v)$, $\# v(c)$ and $\#\sigma(c)$ denote the cardinals of
the sets  $\mu(v)$, $v(c)$ and $\sigma(c)$.
By applying the mapping space $map(-,X)$, we have the commutative diagram
where the square is a pull-back.
\begin{equation}\label{pull-back diagramme de cordes}\xymatrix{
map(F_{g,p+q},X)\ar[dr]_{map(in,X)}
&map(c_{g,p+q},X)\ar[d]\ar[d]^{map(in,X)}\ar[r]\ar[l]_{\thickapprox}
&\prod_{v\in\sigma(c)}map(v,X)\ar[d]\\
&({\mathcal LX})^{\times p}\ar[r]
&\prod_{v\in\sigma(c)} X^{\#\mu(v)}=X^{\# v(c)}}
\end{equation}
As $X^{\# v(c)}$ is simply connected,
the fibration
$$\prod_{v\in\sigma(c)}map(v,X)\twoheadrightarrow\prod_{v\in\sigma(c)} X^{\#\mu(v)}=X^{\# v(c)}$$
is oriented. As orientation is preserved by pull-back and homotopy equivalence,
the fibration $map(in,X):map(F_{g,p+q},X)\twoheadrightarrow({\mathcal LX})^{\times p}$ is also oriented:  
the action of $\pi_1(\mathcal LX^{\times p})$ preserves the orientation class
 $\omega_{F,\theta}\in H_{-d\chi(F)}(map_*(F/\partial_{in} F,X))$.

Remark that using Sullivan Chord diagrams, we can give a second but
more complicated proof of Proposition~\ref{cofibre de in}:
Since the path-connected trees $v\in\sigma(c)$ are contractile
and since the cardinal of $\mu(v)$, $\#\mu(v)$ is always not zero,
the cofiber of the cofibration $\mu(v)\hookrightarrow v$ is homotopy
equivalent to a wedge $\vee_{\#\mu(v)-1} S^1$ of $\#\mu(v)-1$ circles.
Therefore the fiber of the fibration 
$$\prod_{v\in\sigma(c)}map(v,X)\twoheadrightarrow\prod_{v\in\sigma(c)} X^{\#\mu(v)}=X^{\# v(c)}$$
is homotopy equivalent to the product
$\prod_{v\in\sigma(c)} \Omega X^{\#\mu(v)-1}=\Omega X^{\# v(c)-\#\sigma(c)}.
$
By Mayer-Vietoris long exact sequence, we have the additivity formula
for the Euler characteristics
$$
\chi(F_{g,p+q})=\chi(c_{g,p+q})=\chi(\coprod_{i=1}^p S^1)
+\chi(\coprod_{v\in\sigma(c)} v)
-\chi(\coprod_{v\in\sigma(c)} \mu(v))=0+\#\sigma(c)-\# v(c).
$$
Since fibers are preserved by pull-backs, we recover that the fiber
of  $map(in,X)$ is homotopy equivalent to $\Omega X^{-\chi(F_{g,p+q})}$.
\end{proof}
{\bf Examples :} 
\\
1) In order to illustrate the proof of the preceding proposition let us consider the fundamental example of the pair of pants $P$, viewed as a cobordism between two ingoing circles and one outgoing circle.
In this particular case one has $c_{0,2+1}=O-O$. One replaces the space $map(P,X)$ by $map(O-O,X)$ and as $\sharp\sigma(O-O)=1$ and $\sharp v(O-O)=2$ we have to deal with the pull-back diagram
$$\xymatrix{
map(O-O,X)\ar[r]\ar[d]_{map(in,X)}& map(I,X)\ar[d]^{(ev_0,ev_1)}\\
({\mathcal LX})^{\times 2}\ar[r]_{ev_0\times ev_0}
&X^{\times 2}.}
$$
The product on $H_*({\mathcal LX})$ is the composite
$$H_*({\mathcal LX}\times {\mathcal LX})
\buildrel{map(in,X)_!}\over\rightarrow H_{*+d}(map(O-O,X))
\buildrel{H_*(out)}\over\rightarrow H_{*+d}({\mathcal LX}).$$
\\
2) Let us consider again the pair of pants $P$, viewed this time
as a cobordism between one ingoing circle and two outgoing circles.
In this case, $c_{0,1+2}=\O$ and we have the pull-back diagram
$$\xymatrix{
map(\O,X)\ar[r]\ar[d]_{map(in,X)}&map(I,X)\ar[d]^{(ev_0,ev_1)}\\
{\mathcal LX}\ar[r]_{(ev_0,ev_{1/2})}&X^{\times 2}.}
$$

The coproduct on $H_*({\mathcal LX})$ is the composite
$$H_*({\mathcal LX})
\buildrel{map(in,X)!}\over\rightarrow H_{*+d}(map(\O,X))
\buildrel{H_*(out)}\over\rightarrow H_{*+d}({\mathcal LX}\times {\mathcal LX}).$$
Our pull-back square~(\ref{pull-back diagramme de cordes}) is the same as the one
considered in~\cite[section 3]{1095.55006}, except that we did not collapse the trees $v\in\sigma(c)$ (or the ``ghost
edges'' of $c_{g,p+q}$ with the
terminology of~\cite[section 3]{1095.55006}), since we want oriented fibrations and Cohen and Godin wanted embeddings
in order to have shriek maps.

For example, the two Sullivan's Chord diagrams $O-O$ and $\O$ give both after
collapsing the unique edge of their unique tree (or the unique ``ghost edge'')
the famous figure eight $\infty$, considered by Chas
and Sullivan.

Therefore up to this collapsing, our product on $H_*(\mathcal{L}X)$ is defined as the Chas-Sullivan loop
product on $H_*(\mathcal{L}M)$ for manifolds.
Our coproduct on $H_*(\mathcal{L}X)$ is defined as the Cohen-Godin loop
coproduct on $H_*(\mathcal{L}M)$.
\subsection{Orientability of the fibration $\rho_{in}$}
\begin{theor}\label{action du mapping class group est trivial}
Let $F_{g,p+q}$ be a path-connected cobordism from $\coprod_{i=1}^p S^1$
to $\coprod_{i=1}^q S^1$. Assume that $p\geq 1$ and that $q\geq 1$ and
let $\chi(F)=2-2g-p-q$ be the Euler characteristics of $F_{g,p+q}$.
Let $X$ be a simply connected space such that $H_*(\Omega X)$ is a finite
dimensional vector space.
Denote by $d$ the top degree such that $H_d(\Omega X)\neq \{0\}$.
Then the action of $Diff^+(F,\partial)$ on
$H_{-d\chi(F)}(map_*(F/\partial_{in}F,X))$ is trivial.
\end{theor}
To prove this Theorem, we will need the following Propositions~\ref{orientation et determinant}
and~\ref{action sur la top class}.
\begin{propriete}(Compare with~\cite[Lemma 1 p. 1176]{Conant-Vogtmann:theokontsevich})\label{determinant et suite exacte}
Consider a commutative diagram of exact sequences of abelian groups
$$
\xymatrix{
0\ar[r] & A\ar[r]\ar[d]_a & B\ar[r]\ar[d]_b  & C\ar[r]\ar[d]_c & D\ar[r]\ar[d]_d & 0\\
0\ar[r] & A\ar[r] & B\ar[r] & C\ar[r] & D\ar[r] & 0.
}$$
If $A$, $C$ and $D$ are free and finitely generated then $B$
is also free and finitely generated and the determinants of the vertical
morphisms satisfy the equality
$$
\text{det}(a)\text{det}(c)=\text{det}(b)\text{det}(d).
$$
\end{propriete}
\begin{proof}
First consider the case $A=\{0\}$ where we have short exact sequences.
A splitting $D\rightarrow C$, a basis of $B$ and a basis of $D$ gives
a basis of $C$ where the matrix of $c$ is a triangular by block matrix of
the form$
\left(\begin{array}{cc}
b & ?\\
0 & d
\end{array}\right)
$
Therefore $\text{det}(c)=\text{det}(b)\text{det}(d)$.
The general case follows by splitting $\xymatrix@1{
0\ar[r] & A\ar[r] & B\ar[r] & C\ar[r] & D\ar[r] & 0}$ into two short exact sequences.
\end{proof}
\begin{proposition}\label{orientation et determinant}
Let $F_{g,p+q}$ be a path-connected cobordism from
$\coprod_{i=1}^p S^1$ to $\coprod_{i=1}^q S^1$ with $p\geq 0$ and $q\geq 0$.
Let 
 $f\in Homeo^+(F,\partial)$ be an homeomorphism from
$F_{g,p+q}$ to $F_{g,p+q}$ that fixes the boundary pointwise and
preserve the orientation (for example if  $f\in Diff^+(F,\partial)$).
Then the induced isomorphism in singular cohomology
$$H^1(f;\mathbb{Z}):H^1(F,\partial_{in} F;\mathbb Z)\rightarrow H^1(F,\partial_{in} F;\mathbb Z)$$
is of determinant $+1$.
\end{proposition}
Remark that if $p+q\geq 1$, by the exact sequence
$$0=H_2(F_{g,p+q};\mathbb Z)\rightarrow H_2(F_{g,p+q},\coprod_{p+q} S^1;\mathbb Z)
\rightarrow H_1(\coprod_{p+q} S^1;\mathbb Z),$$ continous maps
$f:(F_{g,p+q},\coprod_{p+q} S^1)\rightarrow (F_{g,p+q},\coprod_{p+q} S^1)$
that fix the boundary $\coprod_{p+q} S^1$ are automatically orientation
preserving.
\begin{proof}
By filling the boundary of $ F_{g,p+q}$ with $p+q$ closed disks $ D^2$, we have the push out
$$
\xymatrix{
\coprod_{p+q} S^1\ar[r]\ar[d]
& F_{g,p+q}\ar[d]\ar[dr]^f\\
\coprod_{p+q} D^2\ar[r]\ar[dr]_{id}
& F_g \ar@{.>}[dr]^{\tilde{f}}&  F_{g,p+q}\ar[d]\\
& \coprod_{p+q} D^2\ar[r]
&  F_g
}$$
Denote by $Homeo^+(F_g,\coprod_{p+q} D^2)$ the group of orientation preserving
homeomorphisms from $F_g$ to $F_g$ that fix the $p+q$ embedded disks pointwise.
By universal property of push outs, any $f\in  Homeo^+(F,\partial)$
can be extended to an unique $\tilde{f}\in Homeo^+(F_g,\coprod_{p+q} D^2)$.
Since $$H_2(F_g)\cong H_2(F_g,\coprod_{p+q} D^2)\cong H_2(F_{g,p+q},\coprod_{p+q} S^1),$$ $H_2(\tilde{f}):H_2(F_g)\rightarrow H_2(F_g)$
is the identity if and only if
$$H_2(f):H_2(F_{g,p+q},\coprod_{p+q} S^1)\rightarrow H_2(F_{g,p+q},\coprod_{p+q} S^1)$$ is also the identity.
Therefore $\tilde{f}$ preserves the orientation if and only if
$f$ also preserves the orientation.
Since the restriction of  $\tilde{f}$ to $F_{g,p+q}$ is f, the two groups
$Homeo^+(F_g,\coprod_{p+q} D^2)$ and $Homeo^+(F,\partial)$ are isomorphic.
Observe that with diffeomorphisms instead of homeomorphisms, we don't have
an isomorphism between $Diff^+(F_g,\coprod_{p+q} D^2)$ and $Diff^+(F,\partial)$ althought the two groups are usually identified~\cite[p. 169]{Morita:characteristic}.

Since $H^*(\tilde{f})$ preserves the cup product $\cup$ and the fundamental
class $[F_g]\in Hom_\mathbb{Z}(H^2(F_g),\mathbb{Z})$,  $H^1(\tilde{f})$
preserves the symplectic bilinear form on $H^1(F_g)\cong\mathbb{Z}^{2g}$
defined by $<a,b>:=[F_g](a\cup b)$ for $a$ and $b\in H^1(F_g)$.
Since symplectic automorphisms are of determinant $+1$,
$H^1(\tilde{f})$ is of determinant $+1$ and the
well known~\cite[Definition of Torelli group]{Morita:characteristic} case $p=q=0$
is proved.

We now prove by induction that $\forall i$ such that $0\leq i\leq q$,
$H^1(\tilde{f}_{|F_{g,0+i}}):H^1(F_{g,0+i})\rightarrow H^1(F_{g,0+i})$
is of determinant $+1$. This will prove the case $p=0$.
If $i\geq 1$, consider the commutative diagram of long exact sequences
$$
\xymatrix{
0=\tilde{H}^0(S^1)\ar[r] 
& H^1(F_{g,0+i-1})\ar[r]\ar[d]_{H^1(\tilde{f}_{|F_{g,0+i-1}})}
& H^1(F_{g,0+i})\ar[r]\ar[d]_{H^1(\tilde{f}_{|F_{g,0+i}})} 
& H^1(S^1)\ar[r]\ar[d]_{H^1(f_{|S^1})}
&  H^2(F_{g,0+i-1})\ar[r]\ar[d]_{H^2(\tilde{f}_{|F_{g,0+i-1}})}
& H^2(F_{g,0+i})=0\\
0=\tilde{H}^0(S^1)\ar[r] & H^1(F_{g,0+i-1})\ar[r] & 
H^1(F_{g,0+i})\ar[r]  & H^1(S^1)\ar[r] &  H^2(F_{g,0+i-1})\ar[r] & H^2(F_{g,0+i})=0.
}$$
Since the restriction of $f$ to $S^1$, $f_{|S^1}$, is the identity morphism, 
$H^2(\tilde{f}_{|F_{g,0+i-1}})$ is also the identity morphism.
Therefore by Property~\ref{determinant et suite exacte},
the determinant of $H^1(\tilde{f}_{|F_{g,0+i}})$ is equal to the determinant
of $H^1(\tilde{f}_{|F_{g,0+i-1}})$ which is by induction hypothesis $+1$.
This finishes the induction.

Suppose now that $p\geq 1$ and $q\geq 0$.
Consider the commutative diagram of long exact sequences
$$
\xymatrix{
0=H^0(F_{g,0+q},\coprod_p D^2)\ar[d]
&H^0(F_{g,0+q},\coprod_p D^2)=0\ar[d]\\
H^0(F_{g,0+q})\ar[d]\ar[r]^{H^0(\tilde{f}_{|F_{g,0+q}})}
&H^0(F_{g,0+q})\ar[d]\\
H^0(\coprod_p D^2)\ar[d]\ar[r]^{H^0(\tilde{f}_{|\coprod_p D^2})}&
H^0(\coprod_p D^2)\ar[d]\\
H^1(F_{g,0+q},\coprod_p D^2)\ar[d]\ar[r]^{H^1(\tilde{f}_{|F_{g,0+q}})}
&H^1(F_{g,0+q},\coprod_p D^2)\ar[d]\\
H^1(F_{g,0+q})\ar[d]\ar[r]^{H^1(\tilde{f}_{|F_{g,0+q}})}
&H^1(F_{g,0+q})\ar[d]\\
0=H^1(\coprod_p D^2)&H^1(\coprod_p D^2)=0
}
$$
Since the restriction $\tilde{f}_{|\coprod_p D^2}$ is the identity,
$H^0(\tilde{f}_{|F_{g,0+q}})$ is also the identity.
We have proved that
$$ H^1(\tilde{f}_{|F_{g,0+q}}):H^1(F_{g,0+q})\rightarrow H^1(F_{g,0+q})$$
is of determinant $+1$. Therefore,
by Property~\ref{determinant et suite exacte},
$$H^1(\tilde{f}_{|F_{g,0+q}}):H^1(F_{g,0+q},\coprod_p D^2)
\rightarrow H^1(F_{g,0+q},\coprod_p D^2)$$
is also of determinant $+1$. Since $H^1(F_{g,p+q},\partial_{in}F)
\cong  H^1(F_{g,0+q},\coprod_p D^2)$, the determinants of
$$H^1(f):H^1(F_{g,p+q},\partial_{in} F)\rightarrow
H^1(F_{g,p+q},\partial_{in} F)$$ and $$H^1(\tilde{f}_{|F_{g,0+q}}):H^1(F_{g,0+q},\coprod_p D^2)\rightarrow H^1(F_{g,0+q},\coprod_p D^2)$$ are equal.
\end{proof}
\begin{proposition}\label{action sur la top class}
Let $n\geq 0$ be a non-negative integer.
Let $h:\vee_n S^1\buildrel{\thickapprox}\over\rightarrow \vee_n S^1$ be a pointed homotopy
equivalence from a wedge of $n$ circles to itself.
Let $X$ be a simply connected space such that $H_*(\Omega X)$ is finite dimensional.
Let $d$ be the top degre such that $H_d(\Omega X)\neq\{0\}$.
Then in the top degre,
$$H_{dn}(map_*(h,X);\mathbb{F}):H_{dn}(map_*(\vee_n S^1,X);\mathbb{F})
\rightarrow H_{dn}(map_*(\vee_n S^1,X);\mathbb{F})
$$
is the multiplication by $(det H_1(h;\mathbb{Z}))^d$, the $d$-th power of the determinant
of $$H_1(h;\mathbb{Z}):H_1(\vee_n S^1\mathbb{Z})=\mathbb{Z}^n\rightarrow H_1(\vee_n S^1\mathbb{Z})=\mathbb{Z}^n.$$
\end{proposition}
\begin{proof}
If $n=0$ then $h$ is the identity map and the proposition follows since
$H_1(h;\mathbb{Z}):\{0\}\rightarrow \{0\}$ is of determinant $+1$.
Denote by
$h\vee S^1:\left(\vee_n S^1\right)\vee S^1÷\buildrel{\thickapprox}\over\rightarrow\left(\vee_n S^1\right)\vee S^1$ the homotopy equivalence extending $h$ and the identity of $S^1$.
Note that $\text{det } H_1(h\vee S^1;\mathbb{Z})=\text{det } H_1(h;\mathbb{Z})$.
Since $H_d(\Omega X)$ is of dimension $1$ and $map_*(h\vee S^1,X)=map_*(h,X)÷\times\Omega X$,
by naturality of Kunneth theorem, we have the commutative diagram of vector spaces
$$
\xymatrix{
H_{dn}map_*(\vee_n S^1,X)\ar[r]^-{\cong}\ar[d]|{H_{dn}(map_*(h,X))}
& H_{dn}map_*(\vee_n S^1,X)\otimes H_d(\Omega X)
\ar[r]^-{\cong}\ar[d]|{H_{dn}(map_*(h,X))\otimes H_d(\Omega X) }
& H_{d(n+1)}map_*(\vee_{n+1} S^1,X)\ar[d]|{H_{d(n+1)}(map_*(h\vee S^1,X))}\\
H_{dn}map_*(\vee_n S^1,X)\ar[r]_-{\cong}
& H_{dn}map_*(\vee_n S^1,X)\otimes H_d(\Omega X)\ar[r]_-{\cong}
& H_{d(n+1)}map_*(\vee_{n+1} S^1,X)
}
$$
where the horizontal morphisms are isomorphisms. So if $H_{d(n+1)}(map_*(h\vee S^1,X);\mathbb{F})$
is the multiplication by $(\text{det }H_1(h\vee S^1;\mathbb{Z}))^d$, then
$H_{dn}(map_*(h,X);\mathbb{F})$ is also the multiplication by
$(\text{det }H_1(h\vee S^1;\mathbb{Z}))^d=(\text{det }H_1(h;\mathbb{Z}))^d$.
Therefore if the proposition is proved for an integer $n+1$, the proposition is also proved
for the previous integer $n$.

It remains to prove the proposition for large $n$: we assume now that $n\geq 3$.

Let $\text{aut}_*\vee_n S^1$ be the monoid of pointed self equivalences of the wedge of $n$ circles.
Recall that $\pi_1(\vee_n S^1)$ is isomorphic to the free group $F_n$ on $n$ letters.
Denote by $\text{Aut}F_n$ the groups of automorphisms of $F_n$.
Since $\vee_n S^1$ is a $K(\pi,1)$, by~\cite[Chapter 1 Proposition 1B.9]{Hatcher:algtop} and Whitehead theorem,
the morphism of groups $\pi_1(-):\pi_0(\text{aut}_*\vee_n S^1)\rightarrow \text{Aut}F_n$ sending the homotopy class $[h]$
of $h\in\text{aut}_*\vee_n S^1$ to the group automorphism $\pi_1(h)$, is an isomorphism of groups.

The monoid $\text{aut}_*\vee_n S^1$ acts on the right on $map_*(\vee_n S^1,X)$ by composition.
Therefore, we have a morphism of monoids
$$map_*(-,X):\text{aut}_*\vee_n S^1\rightarrow map\left(map_*(\vee_n S^1,X),map_*(\vee_n S^1,X)\right)^{op}.$$
Here $op$ denotes the opposite monoid.
Passing to homology, we have a morphism of groups into the opposite of the general linear group
of the $\mathbb{F}$-vector space $H_{dn}map_*(\vee_n S^1,X)$. 
$$
H_{dn}(map_*(-,X)):\pi_0(\text{aut}_*\vee_n S^1)\rightarrow\text{GL}_{\mathbb{F}}(H_{dn}map_*(\vee_n S^1,X))^{op}.
$$
Since the vector space $H_{dn}(map_*(\vee_n S^1,X))$ is of dimension $1$, the trace map
$$\text{Trace}:\text{GL}_{\mathbb{F}}(H_{dn}map_*(\vee_n S^1,X))\buildrel{\cong}\over\longrightarrow \mathbb{F}-\{0\}
$$
is an isomorphism of abelian groups.
Denote by $(-)_{Ab}$ the Abelianisation functor from groups to abelian groups.
Since $\mathbb{F}-\{0\}$ is an abelian group, by universal property of abelianisation,
we have a commutative diagram of groups
\begin{equation}\label{trace et abelianisation}
\xymatrix{
\text{Aut}F_n\ar@{>>}[d]
&\pi_0(\text{aut}_*\vee_n S^1)\ar[l]_-{\pi_1(-)}^-{\cong}\ar[rr]^-{H_{dn}(map_*(-,X))}
&&\text{GL}_{\mathbb{F}}(H_{dn}map_*(\vee_n S^1,X))^{op}\ar[d]^{Trace}_{\cong}\\
(\text{Aut}F_n)_{Ab}\ar@{.>}[rrr]
&&& \mathbb{F}-\{0\}
}
\end{equation}
The abelianisation of $F_n$,  $(F_n)_{Ab}$ is isomorphic to $\mathbb{Z}^n$.
Denote by 
$$\text{Ab}(-):\text{Aut}F_n\rightarrow \text{Aut}\left((F_n)_{Ab}÷\right)=\text{GL}_n(\mathbb{Z}),$$
the morphism of groups sending $f:F_n\rightarrow F_n$ to  $f_{Ab}:\mathbb{Z}^n\rightarrow \mathbb{Z}^n$.
Since $\mathbb{Z}/2\mathbb{Z}$ is an abelian group, by universal property of abelianisation,
we have a commutative diagram of groups
\begin{equation}\label{determinant et abelianisation}
\xymatrix{
\text{Aut}F_n\ar@{>>}[r]^{\text{Ab(-)}}\ar@{>>}[d]
& \text{GL}_n(\mathbb{Z})\ar@{>>}[d]^{Det}\\
(\text{Aut}F_n)_{Ab}\ar@{.>}[r]_{\cong}
& \mathbb{Z}/2\mathbb{Z}
}
\end{equation}
Since both the determinant map $\text{Det}:\text{GL}_n(\mathbb{Z})\twoheadrightarrow \mathbb{Z}/2\mathbb{Z}$
and $\text{Ab}(-):\text{Aut}F_n\twoheadrightarrow\text{GL}_n(\mathbb{Z})$
are surjective~\cite[Chapter I.Proposition 4.4]{Lyndon-Schupp:livre},
the morphism $(\text{Aut}F_n)_{Ab}\twoheadrightarrow \mathbb{Z}/2\mathbb{Z}$ is surjective.
By~\cite[Section 5.1]{VogtmannICM06}, since $n\geq 3$, the abelianisation of $\text{Aut}F_n$, $(\text{Aut}F_n)_{Ab}$
is isomorphic to  $\mathbb{Z}/2\mathbb{Z}$.
Therefore this surjective morphism of abelian groups
$(\text{Aut}F_n)_{Ab}\buildrel{\cong}\over\twoheadrightarrow \mathbb{Z}/2\mathbb{Z}$
is in fact an isomorphism.

The composition
$\pi_0(\text{aut}_* \vee_n S^1)\buildrel{\pi_1(-)}\over\rightarrow
\text{Aut}F_n\buildrel{\text{Ab}(-)}\over\rightarrow\text{GL}_n(\mathbb{Z})$
coincides with the morphism of groups
$H_1(-;\mathbb{Z}):\pi_0(\text{aut}_* \vee_n S^1)\rightarrow\text{GL}_n(\mathbb{Z})$
sending the homotopy class $[h]$ of $h\in \text{aut}_* \vee_n S^1$ to the isomorphism of abelian groups
$H_1(h;\mathbb{Z}):\mathbb{Z}^n\buildrel{\cong}\over\rightarrow \mathbb{Z}^n$.
Therefore putting side by side diagram~(\ref{trace et abelianisation})
and diagram~(\ref{determinant et abelianisation}), we obtain the commutative diagram of groups
\begin{equation}\label{trace et determinant}
\xymatrix{
\text{GL}_n(\mathbb{Z})\ar@{>>}_{Det}[d]
&\pi_0(\text{aut}_*\vee_n S^1)\ar[l]_-{H_1(-;\mathbb{Z})}\ar[rr]^-{H_{dn}(map_*(-,X))}
&&\text{GL}_{\mathbb{F}}(H_{dn}map_*(\vee_n S^1,X))^{op}\ar[d]^{Trace}_{\cong}\\
\mathbb{Z}/2\mathbb{Z}\ar@{.>}[rrr]_{i}
&&& \mathbb{F}-\{0\}
}
\end{equation}
We want to compute the morphism of groups $i:\mathbb{Z}/2\mathbb{Z}\rightarrow  \mathbb{F}-\{0\}$.
In order to distinguish the circles in $\vee_{i=1}^n S^1$, we consider $\vee_{i=1}^n S^1$
as the quotient space $\frac{S^1\times\{1,\cdots,n\}}{*\times\{1,\cdots,n\}}$:
an element of $\vee_{i=1}^n S^1$ is the class of $(x,i)$, $x\in S^1$, $i\in\{1,\cdots,n\}$.

Any permutation $\sigma\in\Sigma_n$ induces a pointed homeomorphism
$\sigma\cdot -:\vee_{i=1}^n S^1\buildrel{\cong}\over\rightarrow \vee_{i=1}^n S^1$ defined by
$\sigma\cdot (x,i)=(x,\sigma(i))$.
The matrix of $H_1(\sigma\cdot -;\mathbb{Z})$, which is the image of $\sigma\cdot -$ by the morphism
$H_1(-;\mathbb{Z}):\pi_0(\text{aut}_*\vee_n S^1)\rightarrow \text{GL}_n(\mathbb{Z})$,
is the permutation matrix $M_\sigma:=\left(m_{ij}\right)_{i,j}$ defined
by $$m_{\sigma(i)j}=\begin{cases}
1&\text{Si $j=i$},\\
0& \text{Si $j\neq i$}.
\end{cases}$$
Since the determinant of the permutation matrix $M_\sigma$ is the sign of
$\sigma$, $\varepsilon(\sigma)$,
$$\text{Det}\circ  H_1(-;\mathbb{Z})(\sigma\cdot -)=\varepsilon(\sigma).$$
On the other hand,
$H_*(map_*(\sigma\cdot -,X)):H_*(\Omega X)^{\otimes n}\rightarrow
H_*(\Omega X)^{\otimes n}$ maps $w_1\otimes\cdots\otimes w_n$ to
 $\pm w_{\sigma(1)}\otimes\cdots\otimes w_{\sigma(n)}$ where $\pm$ is the
Koszul sign.
In particular, let $\tau$ be the transposition $(12)$ of $\Sigma_n$,
$H_{dn}(map_*(\tau\cdot -,X)):H_d(\Omega X)^{\otimes n}\rightarrow
H_d(\Omega X)^{\otimes n}$ is the multiplication by $(-1)^d$.
Since the sign of $\tau$ is $-1$, by commutativity of diagram~(\ref{trace et determinant}), we have $$i(-1)=i(\varepsilon(\tau))=i\circ\text{Det}\circ  H_1(-;\mathbb{Z})(\tau\cdot -)=\text{Trace}H_{dn}(map_*(\tau\cdot -,X))=(-1)^d.$$
\end{proof}
\begin{proof}[Proof of Theorem~\ref{action du mapping class group est trivial}]
Let $f\in Diff^+(F,\partial)$.
Since the diffeomorphism $f$ fixes the boundary pointwise, $f$ induces
a pointed homeomorphism
$\bar{f}:F/\partial_{in}F\buildrel{\cong}\over\rightarrow F/\partial_{in}F
$. The action of $f$ on $map_*(F/\partial_{in}F,X)$ is given by
$map_*(\bar{f},X)$.
Since $p\geq 1$ and $q\geq 1$, by Proposition~\ref{cofibre de in},
there exists a pointed homotopy equivalence
$g:\vee_{-\chi(F)} S^1\buildrel{\thickapprox}\over\rightarrow F/\partial_{in}F$.
Consider a pointed homotopy equivalence
$h:\vee_{-\chi(F)} S^1\buildrel{\thickapprox}\over\rightarrow\vee_{-\chi(F)} S^1$
such that $g\circ h$ is (pointed) homotopic to $\bar{f}\circ g$.
Since $map_*(-,X)$ preserves pointed homotopies, we have the commutative
diagram of $\mathbb{F}$-vector spaces
$$
\xymatrix{
H_{-d\chi(F)}(map_*(F/\partial_{in}F,X))\ar[rr]^{H_{-d\chi(F)}(map_*(\bar{f},X))}
\ar[d]_{H_{-d\chi(F)}(map_*(g,X))}^\cong
&& H_{-d\chi(F)}(map_*(F/\partial_{in}F,X))\ar[d]^{H_{-d\chi(F)}(map_*(g,X))}_\cong\\
H_{-d\chi(F)}(map_*(\vee_{-\chi(F)} S^1,X))\ar[rr]_{H_{-d\chi(F)}(map_*(h,X))}
&& H_{-d\chi(F)}(map_*(\vee_{-\chi(F)} S^1,X))}.
$$
Since by Proposition~\ref{action sur la top class} applied to $h$ with $n=-\chi(F)$, $H_{-d\chi(F)}(map_*(h,X))$ is the multiplication by
$(detH_1(h;\mathbb{Z}))^d$, $H_{-d\chi(F)}(map_*(\bar{f},X))$ is also
 the multiplication by $(detH_1(h;\mathbb{Z}))^d$.
Up to the isomorphisms
$$
H_1(F,\partial_{in}F)\buildrel{\cong}\over\rightarrow H_1(F/\partial_{in}F)
\build\rightarrow_{\cong}^{H_1(g)^{-1}} H_1(\vee_{-\chi(F)} S^1),$$
$H_1(f;\mathbb{Z}):H_1(F,\partial_{in}F)\rightarrow H_1(F,\partial_{in}F)$
coincides with
$H_1(h;\mathbb{Z}):H_1(\vee_{-\chi(F)} S^1)\rightarrow H_1(\vee_{-\chi(F)} S^1)$.
In particular, their determinants, $detH_1(f;\mathbb{Z})$ and $detH_1(h;\mathbb{Z})$, are equals.
Since $f$ is orientation preserving, by Proposition~\ref{orientation et determinant}, $detH_1(f;\mathbb{Z})$
is $+1$. So finally,
$$H_{-d\chi(F)}(map_*(\bar{f},X)):H_{-d\chi(F)}(map_*(F/\partial_{in}F,X))\rightarrow H_{-d\chi(F)}(map_*(F/\partial_{in}F,X))$$
is the identity morphism.
\end{proof}
\section{Prop structure}\label{propic}
In this section we prove that the action of evaluation products is propic, we thus have to prove that the evaluation products are compatible with the action of the symmetric group on the boundary components, with the gluing of surfaces along their boundaries and with the disjoint union of surfaces. 
%Let us notice that orientations classes $\theta$ are compatible with these operations, the situation is completely analoguous to the one encountered by V. Godin in \cite{Godin:higherstring}. In what follows for the sake of simplicity we will drop this orientation class from our notations.
\begin{proposition}
If $F_1$ and $F_2$ are two equivalent smooth cobordisms (Definition~\ref{cobordisms equivalents})
then the evaluation products $\mu(F_1)$ and $\mu(F_2)$ coincides.
\end{proposition}
\begin{notation}
Let $S_n$ denote the disjoint union of $n$ circles, $\coprod_{i=1}^n S^1$.
\end{notation}
\begin{proposition}\label{restriction evaluation product}
For any cobordism $F$, the restriction of the evaluation product
$$
\mu(F):H_*(BD(F))\otimes H_*(\mathcal LX)^{\otimes p}
\buildrel{H_*(\rho_{out})\circ \rho_{in}!}\over\rightarrow H_*(\mathcal LX)^{\otimes q}
$$
to $H_0(BD(F))$ coincides with the operation induced by $F$
$$
H_*(\mathcal{L}X)^{\otimes p}\buildrel{in!}\over\rightarrow
H_*(map(F,X))\buildrel{H_*(out)}\over\rightarrow
H_*(\mathcal{L}X)^{\otimes q}.
$$
\end{proposition}
\begin{proof}
Consider the following commutative diagram where all the squares
are pull-backs and the horizontal maps, fibrations.

\xymatrix{
BD(F)\times map(S_p,X)
& (map(F,X))_{hD(F)}\ar[l]_{\rho_{in}}\ar[r]
& BD(F)\times map(S_q,X)\\
ED(F)\times map(S_p,X)\ar[u]\ar[d]_{proj_2}^{\thickapprox}
& ED(F)\times map(F,X)\ar[l]^{Id\times in}
\ar[r]_{Id\times out}\ar[u]\ar[d]^{proj_2}_{\thickapprox}
& ED(F)\times map(S_q,X)\ar[u]\ar[d]^{proj_2}_{\thickapprox}\\
map(S_p,X)
& map(F,X)\ar[l]^{in}\ar[r]_{out}
& map(S_q,X)
}
Here $proj_2$ is the projection on the second factor.
By definition, $\rho_{out}$ fits into the commutative diagram

$$\xymatrix{
map(F,X)_{hD(F)}\ar[r]\ar[d]_{\rho_{out}}
& BD(F)\times map(S_q,X)\ar[dl]_{proj_2}\\
map(S_q,X)
& ED(F)\times map(S_q,X)\ar[l]^{proj_2}_{\thickapprox}\ar[u].
}$$
By naturatility of integration along the fibers with respect to pull-backs,
we obtain the proposition.
\end{proof}
For $\varepsilon=0$ or $1$,
let $i_\varepsilon:S_n\hookrightarrow S_n\times I$,
$x\mapsto (x,\varepsilon)$, be the two canonical inclusions of $S_n$ into
the cylinder $S_n\times I$.
\begin{proposition}
Let $\phi:S_n\rightarrow S_n$ be a diffeomorphism.
Consider the cylinder $S_n\times I$ as the cobordism equipped with
the in-boundary
$S_n\buildrel{i_0\circ \phi}\over\hookrightarrow S_n\times I$
and the out-boundary
$S_n\buildrel{i_1}\over\hookrightarrow S_n\times I$.
Following~\cite[1.3.22]{Kock:Frob2TQFT}, we denote this cobordism $C_{\phi}$.
The operation induced by $C_{\phi}$:
$$H_*(map(i_1,X))\circ map(i_0\circ\phi,X)!:
H_*(\mathcal{L}X)^{\otimes n}\rightarrow
%H_*(map(S_n\times I,X))\rightarrow
H_*(map(C_{\phi},X))\rightarrow
H_*(\mathcal{L}X)^{\otimes n}
$$
coincides with $H_*(map(\phi,X))^{-1}$.
\end{proposition}
\begin{proof}
Since $map(-,X)$ preserves homotopies, the fibration $map(i_0,X)$
is a homotopy equivalence. So by naturality of integration along
the fiber with respect to homotopy equivalences,
$$
H_*(map(i_0,X))\circ map(i_0,X)!=id!\circ id=id.
$$
But since $map(i_0,X)$ is homotopic to $map(i_1,X)$,
$H_*(map(i_0,X))$ is equal to $H_*(map(i_1,X))$.
Therefore we have proved the proposition in the case $\phi=id$.
The general case follows since
$map(\phi,X)!=H_*(map(\phi,X))^{-1}$ and
$map(i_0\circ\phi,X)!=map(i_0,X)!\circ map(\phi,X)!$.
\end{proof}
Let $id:S_n\rightarrow S_n$ be the identity diffeomorphism.
Denote by $id_n\in H_0(BD(C_{id}))$ the identity morphisms of the prop.
Let $\phi:S_m\coprod S_n\buildrel{\cong}\over\rightarrow S_n\coprod S_m$
be the twist diffeomorphism.
Denote by $\tau_{m,n}\in H_0(BD(C_{\phi}))$ the symmetry isomorphism of the
prop. From the previous two propositions, we immediately obtain
\begin{cor}
i) (identity) The evaluation product for the cylinder satisfies

$\mu(C_{id})(id_n\otimes a_1\otimes\dots\otimes a_n)=
a_1\otimes\dots\otimes a_n$
for $a_1$, \dots, $a_n\in H_*(\mathcal{L}X)$.

\noindent ii) (symmetry) The evaluation product for $C_{\phi}$ satisfies

$\mu(C_{\phi})(\tau_{m,n}\otimes v\otimes w)=(-1)^{\vert v\vert\vert w\vert}
w\otimes v$ for $v\in H_*(\mathcal{L}X)^{\otimes m}$ and
$w\in H_*(\mathcal{L}X)^{\otimes n}$.
\end{cor}
The following Lemma is obvious.
\begin{lem}\label{middle exchange}
Let $D$ and $D'$ two topological groups.
Let $X$ be a left $D$-space. Let $X'$ be a left $D'$-space.
Let $Y$ be a topological space that we considered to be a trivial 
left $D$-space and a trivial left $D'$-space.
Let $f:X\rightarrow Y$ be a $D$-equivariant map.
Let $f':X'\rightarrow Y$ be a $D'$-equivariant map.
Consider the pull-back
$$\xymatrix{
X'' \ar[r]\ar[d]
&X\ar[d]^f\\
X'\ar[r]_{f'} & Y
}$$
Then 

i) $X"$ is a sub $D\times D'$-space of $X\times X'$
and we have the pull-back

\[
\xymatrix{
(X'')_{hD\times D'} \ar[r]\ar[d]
&(X)_{hD}\ar[d]\\
(X')_{hD'}\ar[r] & Y.
}\]

ii) In particular, we have a natural homeomorphism

$$(X\times X')_{hD\times D'}\cong (X)_{hD}\times (X')_{hD'}.$$
\end{lem}
\begin{proposition}(Monoidal, i.e. disjoint union)
Let $F_{p+q}$ and $F'_{p'+q'}$ be two surfaces.
For $a\otimes v\in H_*(BD_{p+q}\times\mathcal{L}X^{\times p})$ and
  $b\otimes w\in H_*(BD'_{p+q}\times\mathcal{L}X^{\times p'})$
The evaluation product satisfies
$$\mu(F\amalg F')(a\otimes b\otimes v\otimes w)
=(-1)^{\vert b\vert\vert v\vert-d\chi_{F'}(\vert a\vert+\vert v\vert)}
\mu(F)(a\otimes v)\otimes \mu(F')(b\otimes w).$$
\end{proposition}
\begin{proof}
Consider the two cobordisms
$$\amalg_{i=1}^p S^1\buildrel{in}\over\hookrightarrow F
\buildrel{out}\over\hookleftarrow \amalg_{i=1}^q S^1
\quad\text{and}\quad
\amalg_{i=1}^{p'} S^1\buildrel{in'}\over\hookrightarrow F'
\buildrel{out'}\over\hookleftarrow \amalg_{i=1}^{q'} S^1.
$$
By definition, the disjoint union of these two cobordisms is
$$
\amalg_{i=1}^p S^1\amalg\amalg_{i=1}^{p'} S^1
\buildrel{in \amalg in'}\over\hookrightarrow F\amalg F'
\buildrel{out\amalg out'}\over\hookleftarrow
\amalg_{i=1}^q S^1\amalg\amalg_{i=1}^{q'} S^1.
$$
Since we assume that $F$ and $F'$ have each at least one boundary component,
we have the isomorphism of topological groups
$$
Diff^+(F\amalg F';\partial)\cong
Diff^+(F;\partial)\times Diff^+(F';\partial).
$$
Denote by $D$ the group $Diff^+(F;\partial)$ while $D':=Diff^+(F';\partial)$.
The homeomorphism $map(Y\coprod Z,X)\cong map(Y,Z)\times map(Z,X)$
is natural in $Y$ and $Z$.
Therefore using Lemma~\ref{middle exchange} ii), we have the commutative diagram
$$
\xymatrix{
(map(F,X))_{hD}\times (map(F',X))_{hD'}
\ar[r]^\cong\ar[d]_{\rho_{in}\times \rho_{in}'}
& (map(F\amalg F',X))_{hD\times D'}\ar[d]\\
BD\times\mathcal{L}X^{\times p}\times
BD'\times\mathcal{L}X^{\times p'}\ar[r]_\cong 
& ED\times ED'/D\times D'\times\mathcal{L}X^{\times p+p'} 
}
$$
where the horizontal maps are homeomorphisms.
The homotopy equivalence
$ED\times ED'\buildrel{\simeq}\over\rightarrow E(D\times D')$
induces the commutative diagram
$$
\xymatrix{
ED\times ED'\times_{D\times D'} map(F\amalg F',X)\ar[d]\ar[r]^{\simeq}
& E(D\times D')\times_{D\times D'} map(F\amalg F',X)\ar[d]^{\rho_{in}}
\\
 ED\times ED'/D\times D'\times\mathcal{L}X^{\times p+p'} \ar[r]_{\simeq}
& B(D\times D') \times\mathcal{L}X^{p+p'}
}
$$
where the horizontal maps are homotopy equivalences.
Using similar commutative diagrams for $\rho_{out}$,
by naturality of integration along the fiber with respect to
homotopy equivalences, we obtain the commutative diagram

$$
\xymatrix{
H_*(BD\times\mathcal{L}X^{p}\times BD'\times\mathcal{L}X^{p'})
\ar[r]^{\cong}\ar[d]_{(\rho_{in}\times\rho_{in}')_!}
& H_*(B(D\times D')\times\mathcal{L}X^{p+p'})\ar[d]^{\rho_{in}}\\
H_*((map(F,X))_{hD}\times (map(F',X))_{hD'})
\ar[r]^{\cong}\ar[d]_{H_*(\rho_{out}\times\rho_{out}')}
& H_*((map(F\amalg F',X))_{hD\times D'})
\ar[d]^{H_*(\rho_{out})}\\
H_*(\mathcal{L}X^{q}\times\mathcal{L}X^{q'})\ar[r]^{\cong}
& H_*(\mathcal{L}X^{q+q'})
}
$$
where the horizontal maps are isomorphisms.
\end{proof}
We now show that the evaluation products defined in the preceding section are compatible with gluing. Let us pick two surfaces $F_{g,p+q}$ and $F'_{g',q+r}$, gluing these surfaces one gets a third surfaces $F"_{g",p+r}$ and an inclusion of groups $D_{g,p+q}\times D_{g',q+r}\hookrightarrow D_{g",p+r}$.
Therefore the composite $gl$
$$BD_{g,p+q}\times BD_{g',q+r}\thickapprox
B(D_{g,p+q}\times D_{g',q+r})
\rightarrow BD_{g",p+r}$$
gives in homology the gluing morphism
$$gl:H_*(BD_{g,p+q})\otimes H_*(BD_{g',q+r})\rightarrow H_*(BD_{g",p+r}).$$
\begin{proposition}\label{compatible au gluing} (Composition, i.e. gluing). For any $a_i\in H_*(\mathcal LX)$ and any $m_1\in H_*(BD_{g,p+q})$ and $m_2\in H_*(BD_{g',q+r})$ one has 
$$(gl(m_1\otimes m_2))(a_1\otimes\ldots\otimes a_p)=m_2( m_1(a_1\otimes\ldots\otimes a_p)).$$
\end{proposition}
\begin{proof}
We have the two sequences of maps
$$BD_{g,p+q}\times BD_{g',q+r}\times \mathcal LX^{\times p}
\rightarrow BD_{g",p+r}\times\mathcal LX^{\times p}\leftarrow\mathcal M_{g",p+r}(X)\rightarrow\mathcal LX^{\times r}$$
and
$$BD_{g,p+q}\times BD_{g',q+r}\times \mathcal LX^{\times p}\leftarrow BD_{g',q+r}\times\mathcal M_{g,p+q}(X)\rightarrow BD_{g',q+r}\times\mathcal LX^{\times q}\leftarrow\mathcal M_{g',q+r}(X)\rightarrow\mathcal LX^{\times r}.$$
The first induces the operation $gl(m_1\otimes m_2)$ while the second induces $m_2\circ m_1$. In order to compare these two operations we introduce the following intermediate moduli space :
$$\mathcal M_{g,g',p+q+r}(X):=(map(F"_{g",p+r},X))_{hD_{g,p+q}\times D_{g',q+r}}.$$
Denote by $\rho_{in}(g,g',p+g+r):=$
$$(in)_{hD_{g,p+q}\times D_{g',q+r}}:\mathcal M_{g,g',p+q+r}(X)\twoheadrightarrow
BD_{g,p+q}\times BD_{g',q+r}\times \mathcal LX^{\times p}.$$
Denote by $proj_1$ the projection of the first factor.
we have the commutative diagram :
$$\xymatrix{
\mathcal M_{g,g',p+q+r}(X)\ar[r]^{glue}\ar[d]_{\rho_{in}(g,g',p+g+r)}
&\mathcal M_{g",p+r}(X)\ar[d]^{\rho_{in}}\\
BD_{g,p+q}\times BD_{g',q+r}\times \mathcal LX^{\times p}
\ar[r]_{gl\times Id}\ar[d]_{proj_1}
& BD_{g",p+r}\times \mathcal LX^{\times ^p}\ar[d]^{proj_1}\\
BD_{g,p+q}\times BD_{g',q+r}\ar[r]_{gl}& BD_{g",p+r}
}$$
The two composites of the vertical maps,
$proj_1\circ\rho_{in}(g,g',p+g+r)$ and $proj_1\circ\rho_{in}$
are fiber bundles with the
same fiber $map(F"_{g",p+r},X)$.
Therefore the total square is a pull-back.
Since the lower square is obviously a pull-back, by associativity of pull-back,
the upper square is also a pull-back.
Thus, by property of the integration along the fibers in homology we have :
$$\mu_{g,g',p+r}\circ(gl\times Id)_*=
(\rho_{out})_*\circ glue_*\circ(\rho_{in}(g,g',p+q+r))_!$$
thus $gl(m_1\otimes m_2)(a_1\otimes\ldots a_p)$ is equal to
$$(\rho_{out})_*\circ glue_*\circ(\rho_{in}(g,g',p+q+r))_!
(m_1\otimes m_2\otimes a_1\ldots\otimes a_p).$$

Now, we have to compare our second correspondence
(corresponding to $m_2\circ m_1$) to the correspondence
$$BD_{g,p+q}\times BD_{g',q+r}\times \mathcal LX^{\times p}
\buildrel{\rho_{in}(g,g',p+q+r)}\over
\longleftarrow\mathcal M_{g,g',p+q+r}(X)\rightarrow\mathcal LX^{\times r}.$$
By definition, $F"_{g",p+r}$ is given by the push-out
$$\xymatrix{
\coprod_{i=1}^q S^1 \ar[r]\ar[d]
&F_{g,p+q}\ar[d]\\
F'_{g,q+r}\ar[r] & F"_{g",p+r}.
}
$$
By applying $map(-,X)$, we obtain the pull-back
\begin{equation}~\label{produit fibre gluing}
\xymatrix{
map(F"_{g",p+r},X)\ar[r]^f\ar[d]
&map(F_{g,p+q},X)\ar[d]^{map(out,X)}\\
map(F'_{g,q+r},X)\ar[r]_-{map(in,X)} & \mathcal{L}X^{\times q}. 
}
\end{equation}
By applying the Lemma~\ref{middle exchange} i) to the previous pull-back,
we obtain the pull-backs,
$$
\xymatrix{
\mathcal M_{g,g',p+q+r}(X) \ar[r]\ar[d]
&BD_{g',q+r}\times\mathcal M_{g,p+q}(X)
\ar[d]^{BD_{g',q+r}\times\rho_{out}}\ar[r]_-{proj_2}
&\mathcal M_{g,p+q}(X)\ar[d]^{\rho_{out}}\\
\mathcal M_{g',q+r}(X)\ar[r]_{\rho_{in}} &
BD_{g',q+r}\times\mathcal LX^{\times q}\ar[r]_-{proj_2}
&\mathcal LX^{\times q}
}
$$
\noindent Finally, we have obtain the diagram
$$
\xymatrix{
& BD_{g',q+r}\times\mathcal M_{g,p+q}(X)\ar[r]_{BD_{g',q+r}\times\rho_{out}}
\ar[dl]_{BD_{g',q+r}\times\rho_{in}}
& BD_{g',q+r}\times\mathcal LX^{\times q}\\
BD_{g,p+q}\times BD_{g',q+r}\times \mathcal LX^{\times p}
\ar[d]_{gl\times\mathcal LX^{\times ^p}}
& \mathcal M_{g,g',p+q+r}(X)\ar[l]^-{\rho_{in}(g,g',p+g+r)}\ar[d]^{glue}\ar[r]\ar[u]
&\mathcal M_{g',q+r}(X)\ar[d]^{\rho_{out}}\ar[u]_{\rho_{in}}\\
BD_{g",p+r}\times \mathcal LX^{\times ^p}
& \mathcal M_{g",p+r}(X)\ar[l]^{\rho_{in}}\ar[r]_{\rho_{out}}
& \mathcal LX^{\times ^r}
}$$
\noindent where:

-the lower left square and the upper right square are pull-backs and

-the lower right square and the upper left triangle commutte.

The comparison follows from the Composition lemma, 
\end{proof}

\section{Results for connected topological groups}

\subsection{Main Theorem}
In Section~\ref{Definition des evaluation products}, we have defined for every oriented cobordism $F$
whose path-conponents have at least one in-boundary component and at least one out-going boundary component,
a linear map of degre $-d\chi(F)$,
$$\mu(F_{g,p+q}):H_*(BD_{g,p+q})\otimes H_*({\mathcal LX})^{\otimes p}
\rightarrow H_*({\mathcal LX})^{\otimes q}.$$
In Section~\ref{propic}, we have shown that these $\mu(F)$ define an action of the corresponding prop
$H_*(BD)$ on $H_*(\mathcal{L}X)$.
Therefore, by Proposition~\ref{equivalence des trois props}, we have our main Theorem:
\begin{theor}\label{main theoreme homologie groupe de Lie}
{\bf (Main Theorem)} Let $X$ be a simply connected topological space such that
the singular homology  of its based loop space with coefficient in a field,
$H_*(\Omega X,\mathbb F)$, is finite dimensional.
Then the singular homology of $\mathcal LX$ taken with coefficients in a field, $H_*(\mathcal{L}X;\mathbb{F})$,
is a non-unital non-counital homological conformal field theory. (See Section~\ref{unital counital ou pas} for the definition)
\end{theor}
\subsection{TQFT structure on $H_*(\mathcal{L}X)$}
In Sections~\ref{composantes connexes de la Segal prop} or~\ref{Tillmann prop}, we recalled that
there is an isomorphim of discrete props $\pi_0(\mathfrak M)\cong \pi_0(BD)\cong sk(2-Cob)$ between the path-components of the two topological props $\mathfrak M$, $BD$
and the skeleton of the category of oriented $2$-dimensional cobordisms.
Therefore, we have an inclusion of linear props
$$
\mathbb{F}[sk(2-Cob)]\cong H_0(BD)\hookrightarrow H_*(BD).
$$
A homological conformal field theory is an algebra over the linear prop $H_*(BD)$
(Section~\ref{unital counital ou pas}).
A $2$-dimensional topological quantum field theory is an algebra over the discrete prop $sk(2-Cob)$
(section~\ref{prop des tqfts}).
So, any (non-counital non-unital) homological conformal field theory is in particular a (non-counital non-unital)
topological quantum field theory.
Notice that (non-counital non-unital)
topological quantum field theory are exactly (non-counital non-unital) commutative Frobenius
algebra~\cite[Theorem 3.3.2]{Kock:Frob2TQFT}.
So finally Theorem~\ref{main theoreme homologie groupe de Lie} implies the following Corollary.
\begin{cor}\label{TQFT homologie groupe de Lie}
Let $X$ be a simply connected topological space such that
$H_*(\Omega X,\mathbb F)$ is finite dimensional.
Then $H_*(\mathcal{L}X;\mathbb{F})$ is a non-unital non-counital Frobenius algebra.
\end{cor}
As we have just explained, Corollary~\ref{TQFT homologie groupe de Lie} is a direct consequence of
Theorem~\ref{main theoreme homologie groupe de Lie}.
But let us notice that Corollary~\ref{TQFT homologie groupe de Lie} is much simpler to prove
than Theorem~\ref{main theoreme homologie groupe de Lie}, because the classifying space
$BD(F)$ is not needed to construct the operations associated to the topological quantum field theory structure
(Proposition~\ref{restriction evaluation product}).
\subsection{Main coTheorem}\label{section main cotheorem}
Let $V$ be a homological conformal field theory which is positively graded and of finite type in
each degree.
Then transposition gives a morphism of props :
$\mathcal End_V^{op}\rightarrow \mathcal End_{V^\vee}$
from the opposite prop of the endomorphims prop of $V$, to the endomorphisms prop of the linear dual of $V$.
Since the category of complex cobordisms, i. e. Segal prop $\mathfrak M$ (Section~\ref{la prop de Segal}), is isomorphic to its opposite category,
the linear dual $V^\vee$ of $V$ is again a homological conformal field theory.
Therefore, from Theorem~\ref{main theoreme homologie groupe de Lie}, we obtain:
\begin{theor}{\bf (Main coTheorem)}\label{main cotheoreme homologie groupe de Lie}
{\it Let $X$ be a simply-connected topological space such that
its based loop space homology with coefficient in a field,
$H_*(\Omega X,\mathbb F)$, is finite dimensional.
Then its free loop space cohomology taken with coefficients into a field,
$H^*(\mathcal{L}X,\mathbb{F})$, is a non-unital non-counital
homological field theory.}
\end{theor}

\subsection{Batalin-Vilkovisky algebra structure.}\label{section structure BV}
In Section~\ref{la prop de Segal}, we have recalled the Segal prop of
Riemann surfaces $\mathfrak{M}=\mathfrak{M}(p,q)_{p,q\geq 0}$.
Consider the sub-operad
$$\mathfrak{M}_0=\mathfrak{M}_0(p)_{p\geq 0}\subset\mathfrak{M}(p,1)_{p\geq 0}$$
where we consider only path-connected Riemann surfaces of genus $0$, with only one
out-going boundary component.
By~\cite[p. 282]{Getzler:BVAlg} (See also~\cite[Proposition 4.8]{Markl-Shnider-Stasheff:opeatp}),
there is a natural map of topological operads from the
 framed little $2$-discs operad $f\mathcal{D}_2$ to $\mathfrak{M}_0$ which is a
homotopy equivalence.
Alternatively, it is implicit in~\cite[Theorem 1.5.16]{thesedewahl}
or~\cite[Theorem 7.3 and Proposition 7.4]{Salvatore-Wahl:FrameddoBVa}, that the
 framed little $2$-discs operad $f\mathcal{D}_2(p)_{p\geq 0}$
is homotopy equivalent as operads to the operad 
$\{B\Gamma_{0,p+1}\}_{p\geq 0}\thickapprox BDiff^+(F_{0,p+1},\partial)_{p\geq 0}$.

Therefore any algebra over the linear prop $H_*(\mathcal{M})\cong H_*(BD)$
is an algebra over the linear operad $H_*(f\mathcal{D}_2)$: any homological
conformal field theory is in particular a Batalin-Vilkovisky algebra.
Therefore,  from Theorem~\ref{main cotheoreme homologie groupe de Lie}, we deduce
when the cohomology is taken with coefficient in a field:
\begin{cor}\label{Structure BV en cohomologie groupe de Lie}
Let ${\Bbbk}$ be any principal ideal domain.
Let $X$ be a simply connected space such

$\bullet$ For each $i\leq d$,  $H_i(\Omega X;{\Bbbk})$ is finite generated, 
 
$\bullet$ $H_d(\Omega X;{\Bbbk})\cong{\Bbbk}$, $H_{d-1}(\Omega X;{\Bbbk})$ is ${\Bbbk}$-free  and

$\bullet$ For $i>d$, $H_i(\Omega X;{\Bbbk})=\{0\}$.

Then the singular cohomology of $\mathcal {L}X$ with coefficients in ${\Bbbk}$,
${H}^*(\mathcal{L}X,{\Bbbk})$, is an algebra over the operad
$\oplus_{g\geq 0} H_*(BDiff^{+}(F_{g,p+1}))$, $p\geq 1$.
In particular,  the shifted cohomology
$\mathbb{H}^*(\mathcal{L}X,{\Bbbk}):=H^{*+d}(\mathcal{L}X,{\Bbbk})$
is a Batalin-Vilkovisky algebra (not necessarily with an unit).
\end{cor}
\begin{proof}
We have already proved above the Corollary when ${\Bbbk}$ is a field.
Over a principal ideal domain ${\Bbbk}$, we are going to indicate what are the differences.
Note also that the proof of Corollary~\ref{Structure BV en cohomologie groupe fini} will be similar.

Let $F_{g,p+q}$ be an oriented cobordism from $\coprod_{i=1}^p S^1$ to $\coprod_{i=1}^q S^1$.
Using the universal coefficient theorem for cohomology and Kunneth theorem, we have
that the cohomology of the fiber of $\rho_{in}$ in degre $-d\chi(F)$
$$H^{-d\chi(F)}map_*(F/\partial_{in} F,X)\cong H^{-d\chi(F)}(\Omega X^{\times -\chi(F)})\cong
 H^{-d}(\Omega X)^{\otimes -\chi(F)}\cong {\Bbbk}.$$
Therefore using integration along the fiber in cohomology, we can define the evaluation product
$$\mu(F)^*:
H^l(\mathcal{L}X^{\times q})\buildrel{H^l(\rho_{out})}\over\rightarrow
H^{l}(\mathcal M_{q,p+p}(X))\buildrel{\rho_{in}^!}\over\rightarrow
H^{l+d\chi_F}(BD(F_{g,p+q})\times\mathcal{L}X^{p}).
$$
Now Section~\ref{propic} implies that the evaluation products
$\mu(F)^*$ are symmetric, compatible with gluing and disjoint union of cobordisms.
Consider the composite, denoted $\mu(F)^{\#}$, of the tensor product of $\mu(F)^*$ with the identity
morphism
$$
H^l(\mathcal{L}X^{\times q})\otimes
H_{m}(BD(F))\buildrel{\mu(F)^*\otimes Id}\over\rightarrow
H^{l+d\chi_F}(BD(F)\times\mathcal{L}X^{p})\otimes
H_{m}(BD(F))
$$
and of the Slant product~\cite[p. 254]{Switzer}
$$/:H^{l+d\chi_F}(BD(F)\times\mathcal{L}X^{p})\otimes
H_{m}(BD(F))\rightarrow H^{l+d\chi_F-m}(\mathcal{L}X^{p}).$$
Again, the evaluation products $\mu(F)^{\#}$ are symmetric, compatible with gluing and disjoint
union. Therefore restricting to path-connected cobordisms $F_{g,1+q}$ with only $p=1$
incoming boundary component, the composite
$$
H^*(\mathcal{L}X)^{\otimes q}\otimes
H_{*}(BD(F_{g,1+q}))\buildrel{Kunneth\otimes Id}\over\rightarrow
H^*(\mathcal{L}X^{\times q})\otimes
H_{*}(BD(F_{g,1+q}))
\buildrel{\mu(F)^\#}\over\rightarrow
H^{*}(\mathcal{L}X)
$$
defines an action of the opposite of the operad
$H_*(BDiff^+(F_{g,1+q},\partial))_{q\geq 1}$ on $H^*(\mathcal{L}X)$.
But as recalled in Section~\ref{section main cotheorem}, the topological operad $BDiff^+(F_{g,1+q},\partial)_{q\geq 1}$
is isomorphic to the opposite of the operad $BDiff^+(F_{g,q+1},\partial))_{q\geq 1}$.
\end{proof}

In the next section, we prove similar results for finite groups. But our results for
finite groups are better than our results for connected topological groups.

-The main theorem for finite groups (Theorem~\ref{main theorem stable version}) is in the homotopy
category of spectra. In~\cite[(4.2)]{Becker-Gottlieb:transferfiberbundle}, there is a stable version of integration
along the fiber for fibre bundles with smooth fibers. So Theorem~\ref{main theoreme homologie groupe de Lie} should also hold in the stable category.

%-The Batalin-Vilkovisky algebra structure on free loop space singular cohomology holds over
%any commutative ring ${\Bbbk}$ for finite groups (Corollary~\ref{Structure BV en cohomologie groupe fini}).
%In fact, Corollary~\ref{Structure BV en cohomologie groupe de Lie} should hold also over any
%commutative ring  ${\Bbbk}$. But a separate proof is needed: using the integration along the fibre
%in cohomology, you need to prove directly without using the prop structure, that the evaluation maps
%are equivariant and compatible with the composition of the operad. This is more complicated than
%the prop action.

-The structure for free loop space homology (respectively cohomology) for finite groups has a counit,
(respectively an unit).
When $G$ is a connected compact Lie group, although we don't prove it, there is also an unit:
the element $(EG\times_G \eta)^!(1)\in H^d(EG\times_G G^{ad})\cong H^d(LBG)$ considered in the proof
of Theorem~\ref{Homologie de Hochschild groupe de Lie}.

To summarize, we believe that all our results for finite groups should extend
to connected compact Lie groups.

\section{The case of finite groups}
In this section,we consider a finite group $G$ instead of a connected topological group.
Using transfer instead of integration along the fibers, we prove that, for a finite group $G$,
the free loop homology on a $K(G,1)$, $H_*(\mathcal{L}(K(G,1)))$,
is a counital non-unital homological conformal field theory.
In order to define the transfert maps $map(in,K(G,1))_!$ and $\rho_{in!}$,
we need to check the finiteness of all the fibers of $\rho_{in}$.

\subsection{finiteness of all the fibres is preserved by:\newline}\label{finitude fibre preserve par}

-pull-back and homotopy equivalence:
Consider a commutative diagram
$$\xymatrix{
E_1\ar[r]^{g}\ar@{>>}[d]^{p_1}
&E_2\ar@{>>}[d]^{p_2}\\
B_1\ar[r]^{h}
&B_2
}$$
where $p_1$ and $p_2$ are two fibrations.
Suppose that the diagram is a pull-back or that $h$ and $g$ are homotopy equivalence.
Then for any $b_1\in B_1$, the fibre of $p_1$ over $b_1$, $p_1^{-1}(b_1)$,
is homotopic to the fiber of  $p_2$ over $h(b_1)$, $p_2^{-1}(h(b_1))$.

-composition:
Let $f:X\twoheadrightarrow Y$ and let $g:Y\twoheadrightarrow Z$ be two fibrations.
Let $z\in Z$ be any element of $Z$. By pull-back of $f$, we obtain the fibration
$f':(g\circ f)^{-1}(z)\twoheadrightarrow g^{-1}(z)$.
If the base space of $f'$, $g^{-1}(z)$, and if all the fibres of $f'$, $f'^{-1}(y)$,
$y\in g^{-1}(z)$, are homotopy equivalent to a finite CW-complex then the total space of $f'$,
$(g\circ f)^{-1}(z)$ is also homotopy equivalent to a finite CW-complex.

-Borel construction:
Let $p:E\twoheadrightarrow B$ be a fibration and also a $G$-equivariant map.
We have the pull-back of principal $G$-bundles
$$
\xymatrix{
EG\times E\ar[r]^{EG\times p}\ar[d]
&EG\times B\ar[d]\\
E_{hG}\ar[r]_{p_{hG}}
&B_{hG}
}
$$
Therefore the fiber of $p_{hG}$ over the class $[x,b]$, $x\in EG$, $b\in B$, is homeomorphic
to the fibre of $p$ over $b$, $p^{-1}(b)$.

\subsection{Finiteness of the fibres of $map(in,K(G,1)$}
 \begin{proposition}\label{toutes les fibres homotopes}
Let $X$ be a path-connected space.
Let $F_{g,p+q}$ be a path-connected oriented cobordism from
$\coprod_{i=1}^p S^1$ to $\coprod_{i=1}^q S^1$.
If $p\geq 1$ and $q\geq 1$ then all the fibers of
the fibration obtained by restriction to the in-boundary
components 
$$map(in,X):map(F_{g,p+q},X)\twoheadrightarrow \mathcal LX^{\times p}$$ 
are homotopy equivalent to the product, $\Omega X^{-\chi(F)}$.
\end{proposition}
\begin{proof}
The proof of this Proposition follows the same pattern as the proof of
Proposition~\ref{action du pi1 sur in est trivial}:
Consider the commutative diagram~(\ref{pull-back diagramme de cordes}).
As $X^{\# v(c)}$ is path connected, all the fibres of
the fibration
$$\prod_{v\in\sigma(c)}map(v,X)\twoheadrightarrow\prod_{v\in\sigma(c)} X^{\#\mu(v)}=X^{\# v(c)}$$
are homotopy equivalent.
By pull-back and then by homotopy equivalence (µSection~(\ref{finitude fibre preserve par})),
all the fibres of
the fibration $map(in,X):map(F_{g,p+q},X)\twoheadrightarrow({\mathcal LX})^{\times p}$
are homotopy equivalent.
Let $x_0\in X$. Denote by $\bar{x_0}$ the constant map from $\partial_{in}F$ to $X$.
By Proposition~\ref{cofibre de in}, the fibre of  $map(in,X)$ over $\bar{x_0}$, the pointed mapping
space $map_*((F/\partial_{in} F,\partial_{in} F)(X,x_0))$ is homotopy equivalent
to the product of pointed loop spaces,  $\Omega X^{-\chi(F)}$. 
\end{proof}
\begin{proposition}\label{finite fibre}
Let $G$ be a finite group.
Let $X$ be a $K(G,1)$.
Let $F_{g,p+q}$ be a path-connected oriented cobordism from
$\coprod_{i=1}^p S^1$ to $\coprod_{i=1}^q S^1$.

1) If $p\geq 1$ and $q\geq 0$ then all the fibers of
the fibration obtained by restriction to the in-boundary
components 
$$map(in,X):map(F_{g,p+q},X)\twoheadrightarrow \mathcal LX^{\times p}$$ 
are homotopy equivalent to a discrete finite set.

2 ) Let $R$ be a ring spectrum.
If $X$ is $R$-small then for any $q\geq 0$, the mapping space
$map(F_{g,0+q},X)$ is $R$-small
\end{proposition}
\begin{proof}

1) The case $p\geq 1$ and $q\geq 1$ follows from the Proposition~\ref{toutes les fibres homotopes}, since $\Omega X$ is homotopy equivalent
to $G$.

Case $F_{0,1+0}$: Existence of the counit (to be compared with~\cite[Section 3]{1095.55006} for the free loop space homology $H_*(\mathcal{L}M)$ of a manifold).
Consider the disk $D^2$ as an oriented cobordism $F_{0,1+0}$ witn one incoming boundary and zero outgoing boundary.
Let $l\in\mathcal {L}X$ be a free loop.
The loop $l$ is homotopic to a constant loop if and only if $l$ can be extend over $D^2$, the cone of $S^1$.
This means exactly that $l$ belongs to the image of
$$map(in,X):map(D^2,X)\twoheadrightarrow \mathcal LX.$$
Let $x_0\in X$. Let $\bar{x_0}$ be the constant free loop equal to $x_0$.
The fibre of  $map(in,X)$ over $\bar{x_0}$ is the double loop space
$map_*((D^2/S^1,S^1)(X,x_0))=\Omega^2(X,x_0)$.
The double loop space $\Omega^2(X,x_0)$ is homotopy equivalent to $\Omega G$, which is a point since
$G$ is discrete.
Therefore all the fibres of $map(in,X):map(D^2,X)\twoheadrightarrow \mathcal LX$ are either empty or
contractile.

Case $p\geq 1$ and $q=0$.
By removing the interior of an embedded disk $D^2$ from the oriented cobordism $F_{g,p+0}$, we obtain
an oriented cobordism $F_{g,p+1}$.
According to the case  $p\geq 1$ and $q\geq 1$, the operation associated to $F_{g,p+1}$
$$
H_*(\mathcal{L}X)^{\times p}\buildrel{map(in,X)_!}\over\longrightarrow H_*(map(F_{g,p+1},X))
\buildrel{map(out,X)}\over\longrightarrow H_*(\mathcal{L}X)
$$
is defined. The operation associated to $D^2=F_{0,1+0}$, namely the counit,
$$
H_*(\mathcal{L}X)\buildrel{map(in,X)_!}\over\longrightarrow H_*(map(D^2,X))
\buildrel{map(out,X)}\over\longrightarrow H_*(map(\O,X))=H_*(point)=\mathbb{F}$$
is also defined.
Therefore by composition, the operation associated to  $F_{g,p+0}$
$$
H_*(\mathcal{L}X)^{\times p}\buildrel{map(in,X)_!}\over\longrightarrow H_*(map(F_{g,p+0},X))
\buildrel{map(out,X)}\over\longrightarrow  H_*(map(\O,X))=H_*(point)=\mathbb{F}
$$
is going to be defined.
More precisely, consider the commutative diagram
$$\xymatrix{
&&\mathcal{L}X^{\times p}\\
&map(F_{g,p+0},X)\ar[ld]_{map(out,X)}\ar[d] \ar[ru]^{map(in,X)}\ar[r]
& map(F_{g,p+1},X)\ar[d]^{map(out,X)}\ar[u]_{map(in,X)}\\
\mathcal{L}X^{\times 0}
&map(F_{0,1+0},X)\ar[l]^-{map(out,X)}\ar[r]_-{map(in,X)}
& \mathcal{L}X^{\times 1}}$$
where the square is the pull-back~(\ref{produit fibre gluing}).
Since all the fibres of 
$map(in,X):map(F_{0,1+0},X)\twoheadrightarrow  \mathcal{L}X^{\times 1}$
are homotopy equivalent a finite discrete set, by pull-back (Section~\ref{finitude fibre preserve par}),
all the fibres of the fibration $map(F_{g,p+0},X)\twoheadrightarrow map(F_{g,p+1},X)$
are also homotopy equivalent to a finite discrete set.
Therefore, by composition (Section~\ref{finitude fibre preserve par}), all the fibres of 
$map(in,X):map(F_{g,p+0},X)\twoheadrightarrow  \mathcal{L}X^{\times p}$ are homotopy equivalent to
a finite discrete set.

Note that the same proof (Proof of Lemma~\ref{composition lemma}) shows that the operation  associated to  $F_{g,p+0}$ is the composite
 of the operations associated to $F_{g,p+1}$ and $F_{g,1+0}$ (This is Proposition~\ref{compatible au gluing} without the $BDiff^+(F,\partial)$, compare also with Lemma 8 and Corollary 9 in~\cite{1095.55006}).

2) Consider this time, the disk $D^2$ as an oriented cobordism $F_{0,0+1}$ with zero incoming boundary and one outgoing boundary. We suppose that $X$ is $R$-small.
Since the fibre of the fibration
$map(in,X):map(D^2,X)\twoheadrightarrow map(\O,X)=point $ is homotopy equivalent to $X$,
the operation associated to $D^2=F_{0,0+1}$, namely the unit,
$$ R=R\wedge\Sigma^\infty map(\O,X)_+\buildrel{\tau_{map(in,X)}}
\over\longrightarrow R\wedge\Sigma^\infty map(D^2,X)_+
\buildrel{R\wedge\Sigma^\infty map(out,X)_+}\over\longrightarrow R\wedge\Sigma^\infty \mathcal{L}X_+$$
is defined using Dwyer's transfert.
By 1), the operation associated to $F_{g,1+q}$ is defined for all $q\geq 0$.
Therefore, by composition (same arguments as in the proof of 1), in the case $p\geq 1$ and $q=0$),
the operation associated to $F_{g,0+q}$,
$$ R\buildrel{\tau_{map(in,X)}}
\over\longrightarrow R\wedge\Sigma^\infty map(F_{g,0+q},X)_+
\buildrel{R\wedge\Sigma^\infty map(out,X)_+}\over\longrightarrow R\wedge(\Sigma^\infty \mathcal{L}X_+)^{\wedge q}$$
is also defined using Dwyer's transfert. That is the fibre of the fibration
$$map(in,X):map(F_{g,0+q},X)\twoheadrightarrow map(\O,X)=point $$ 
is $R$-small.
\end{proof}
\subsection{The results for finite groups}
Let $F_{g,p+q}$ be an oriented cobordism from $\amalg_{i=1}^p S^1$ to
  $\amalg_{i=1}^q S^1$.
Let $Diff^+(F;\partial)$ be the group of orientation preserving diffeomorphisms
that fix the boundaries pointwise.

Let $G$ be a finite discrete group.
Let $X$ be a $K(G,1)$.
Suppose that every path component of  $F_{g,p+q}$ has at least one in-boundary component.
By part 1) of Proposition~\ref{finite fibre}, all the fibers of the fibration
$map(in;X):map(F_{g,p+q},X)\twoheadrightarrow \mathcal{L}X^{\times p}
$ are homotopy equivalent to a finite CW-complex.
Therefore by Section~\ref{finitude fibre preserve par}, all the fibres of the fibration obtained by Borel construction $(-)_{hDiff^+(F;\partial)}$
$$
\rho_{in}:=map(in,X)_{hDiff^+(F;\partial)}:
map(F_{g,p+q},X)_{hDiff^+(F;\partial)}\twoheadrightarrow BDiff^+(F;\partial)\times\mathcal{L}X^{\times p}
$$
are homotopy equivalent to a finite CW-complex.

Therefore we can define the {\it evaluation product} associated to $F_{g,p+q}$.
$$
\mu(F):\Sigma^\infty BDiff^+(F;\partial)_+\wedge (\Sigma^\infty \mathcal{L}X_+)^{\wedge p}\rightarrow
 (\Sigma^\infty \mathcal{L}X_+)^{\wedge q}$$  
by the composite of the transfert map of $\rho_{in}$,
$$
\tau_{\rho_{in}}:
\Sigma^\infty BDiff^+(F;\partial)_+\wedge (\Sigma^\infty \mathcal{L}X_+)^{\wedge p}\rightarrow
\Sigma^\infty map(F_{g,p+q},X)_{hDiff^+(F;\partial)+}
$$
and of
$$
\Sigma^\infty \rho_{out+}:\Sigma^\infty map(F_{g,p+q},X)_{hDiff^+(F;\partial)+}\rightarrow
(\Sigma^\infty \mathcal{L}X_+)^{\wedge q}.
$$
Now the same arguments as in Section~\ref{propic} give a stable version of the main theorem:
\begin{theor}(Stable version)\label{main theorem stable version}
Let $G$ be a finite group.
Let $X$ be a $K(G,1)$.
Then the suspension spectrum $\Sigma^{\infty}\mathcal{L}X_+$ is an algebra over the stable prop of surfaces $\Sigma^{\infty} (\propBD)_+$ in the stable homotopy category (The topological prop $\propBD$
is defined in Proposition~\ref{equivalence des trois props} except that we consider here the cobordisms whose path components all
have at least one in-boundary components).
\end{theor}
\begin{cor}(Stable Frobenius algebra)\label{Stable Frobenius algebra}
Let $G$ be a finite group.
Let $X$ be a $K(G,1)$.
Then the suspension spectrum $\Sigma^{\infty}\mathcal{L}X_+$ is a counital non-unital commutative Frobenius
object (in the sense of~\cite[3.6.13]{Kock:Frob2TQFT}) in the homotopy category of spectras.
In particular the suspension spectrum $\Sigma^{\infty}\mathcal{L}X_+$ is a non-unital commutative
associative ring spectrum.
\end{cor}
\begin{proof}
Consider the topological prop $BD$ (or Segal prop $\mathfrak M$ since they are homotopy equivalent).
Any map $\varphi$ from a discrete set $E$ to a topological space $Y$ is uniquely determined up
to homotopy by the composite
$E\buildrel{\varphi}\over\rightarrow Y\twoheadrightarrow \pi_0(Y)$.
Therefore the quotient map $BD\twoheadrightarrow \pi_0(BD)$ admits a section
$\sigma:\pi_0(BD)\hookrightarrow BD$, which is a morphism
of props up to homotopy since the quotient map $BD\twoheadrightarrow \pi_0(BD)$ is a morphism of props.
Recall from Sections~\ref{composantes connexes de la Segal prop} or~\ref{Tillmann prop} that
there is an isomorphim of discrete props $\pi_0(\mathfrak M)\cong \pi_0(BD)\cong sk(2-Cob)$ between the path-components of the two topological props $\mathfrak M$, $BD$
and the skeleton of the category of oriented $2$-dimensional cobordisms.
So we have a morphism of props $\Sigma^{\infty}sk(2-Cob)_+\hookrightarrow \Sigma^{\infty} (\propBD)_+$ in the
stable homotopy category.
Therefore by Theorem~\ref{main theorem stable version}, $\Sigma^{\infty}\mathcal{L}X_+$
 is an algebra over the stable prop  $\Sigma^{\infty}sk(2-Cob)_+$.
So by~\cite[Theorem 3.6.19]{Kock:Frob2TQFT}, $\Sigma^{\infty}\mathcal{L}X_+$ is a commutative Frobenius object
in the homotopy category of spectra.
\end{proof}
\begin{cor}(Batalin-Vilkovisky algebra)\label{Structure BV en cohomologie groupe fini}
Let $G$ be a finite group.
Let $X$ be a $K(G,1)$.
Let $h^*$ be a generalized cohomology theory coming from a commutative ring spectrum.
Then $h^*(\mathcal{L}X)$, is an algebra over the operad
$\oplus_{g\geq 0} h_*(BDiff^{+}(F_{g,p+1}))$, $p\geq 0$.
In particular, the singular free loop space cohomology of $X$, with coefficients in any commutative ring ${\Bbbk}$,
$H^*(\mathcal{L}X;{\Bbbk})$, is an unital Batalin-Vilkovisky algebra.
\end{cor}
\begin{propriete}\label{Slant product}
In the homotopy category of spectra, let $Y$ be a coalgebra over an operad $O=O(n)_{n\geq 0}$.
Let $\mu:O(n)\wedge Y\rightarrow Y^{\wedge n}$ be the evaluation product.
Then the composition of
$$
h^*(Y)^{\otimes n}\otimes h_*(O(n))\rightarrow h^*(Y^{\wedge n})\otimes h_*(O(n))
\buildrel{h^*(\mu)\otimes h_*(O(n))}\over\rightarrow h^*(O(n)\wedge Y)\otimes h_*(O(n))
$$
and of the slant product for generalized multiplicative cohomology~\cite[p. 270 iii)]{Switzer}
$$
/:h^*(O(n)\wedge Y)\otimes h_*(O(n))\rightarrow h^*(Y)
$$
makes $h^*(Y)$ into an algebra over the opposite of the operad $h_*(O)$.
\end{propriete}
\begin{proof}[Proof of Corollary~\ref{Structure BV en cohomologie groupe fini}]
By Theorem~\ref{main theorem stable version}, $\Sigma^{\infty}\mathcal{L}X_+$ is a coalgebra over the stable operad
$\Sigma^{\infty}\vee_{g\geq 0}BDiff^{+}(F_{g,1+q})_+$,
$q\geq 0$, or over the homotopy equivalent stable operad
$\Sigma^{\infty}\mathfrak{M}(1,q)_+$, $q\geq 0$.
By Property~\ref{Slant product},
$h^*(\mathcal{L}X)$, is an algebra over the operad
$h_*(\mathfrak{M}(1,q))^{op}$, $q\geq 0$. But the topological operad
$\mathfrak{M}(p,1)$, $p\geq 0$ is isomorphic to the opposite of the operad
$\mathfrak{M}(1,q)$, $q\geq 0$.
Therefore $h^*(\mathcal{L}X)$ is an algebra over the operad
$h_*(\mathfrak{M}(p,1))$, $p\geq 0$.

By taking the zero genus part, as in Section~\ref{section structure BV},
we obtain that $h^*(\mathcal{L}X)$, is an algebra over the operad
$h_*(f\mathcal{D}_2)$. If $h^*$ is any singular cohomology theory,
$h^*(\mathcal{L}X)$, is an unital Batalin-Vilkovisky algebra.
\end{proof}
From Theorem~\ref{main theorem stable version} and Corollary~\ref{Stable Frobenius algebra}, we get immediately:
\begin{theor}(Generalized homology version)\label{Generalized homology version groupe fini}
Let $G$ be a finite group.
Let $X$ be a $K(G,1)$.
Let $h_*$ be a generalized homology theory with a commutative product such that the Kunneth
morphism
$h_*(X)\otimes_{h_*(pt)}h_*(Y)\buildrel{\cong}\over\rightarrow h_*(X\times Y)$ is an isomorphism.
Then

i) The free loop space homology $h_*(\mathcal{L}X)$ is an algebra over the prop $h_*(BD)$
in the category of graded modules over the graded commutative algebra $h_*(pt)$.

ii) $h_*(\mathcal{L}X)$ is a non-unital counital Frobenius algebra in the category of $h_*(pt)$-modules.
\end{theor}
In particular from i), we have:
\begin{theor}(Singular homology version)\label{main theoreme homologie groupe fini}
Let $G$ be a finite group.
Let $X$ be a $K(G,1)$. Then the singular free loop space homology of $X$, with coefficients in a field $\mathbb{F}$,
$H_*(\mathcal{L}X;\mathbb{F})$, is a counital non-unital homological conformal field theory.(See Section~\ref{unital counital ou pas} for the definition).
\end{theor}
Of course, there is also a singular cohomology version.
In~\cite[p. 176, (B.7.3)]{Ravenel:Nilpotence}, Ravenel explains that Morava $K$-theory and singular homology with
field coefficients are essentially the only generalized homology theories where Kunneth is an isomorphism.
For these generalized homology theory with Kunneth isomorphism, it turns out that, in many case, the Frobenius
algebra $h_*(\mathcal{L}X)$ of Theorem~\ref{Generalized homology version groupe fini} ii) has an unit:
\begin{cor}(unit)\label{unite pour les groupes finis}
Let $G$ be a finite group. Let $X$ be a $K(G,1)$. Then

1) (Dijkgraaf-Witten) if $char(\mathbb F)$ does not divide card(G) then $H_*(\mathcal {L}BG;\mathbb F)=H_0(\mathcal {L}BG;\mathbb F)$ is an
unital and counital Frobenius algebra.

2) (Comparison with Strickland~\cite{Strickland:duality} below) if $K(n)$ is the even periodic Morava $K$-theory spectrum at an odd prime then $K(n)_*(\mathcal LBG)$ is an unital and counital Frobenius algebra in the category of $K(n)_*(pt)$-graded modules.
\end{cor}
\begin{proof}
Let $R$ be the Eilenberg-Mac Lane spectrum $H\mathbb{F}$ in case 1) and let $R$ be $K(n)$
in case 2). By part 2) of Proposition~\ref{finite fibre}, it suffices to show that
$X$ is $R$-small.
As recalled in Section~\ref{Dwyer transfer}, this is the case if the homology
$R_*(X)$ is finitely generated as $R_*(pt)$-module.

Case 1) Since the cardinal of $G$ is invertible in $\mathbb{F}$,
$H_*(X;\mathbb{F})$ is concentrated in degre $0$~\cite[dual of Chap III Corollary 10.2]{Brown:cohgro}.
Therefore $H_*(X;\mathbb{F})=H_0(X;\mathbb{F})\cong \mathbb{F}$ is a finite dimension
vector space over $\mathbb{F}$.

Case 2) If $p$ is odd, the two-periodic Morava $K$-theory $K(n)$ is a
commutative ring spectrum~\cite[p. 764]{Strickland:duality}.
By~\cite{RavenelMexique}, $K(n)_*(X)$ is finitely generated as $K(n)_*(pt)$-modules.
\end{proof}
\begin{rem}
1) When $\mathbb{F}$ is the field of complex numbers $\mathbb{C}$,
we have not checked that our Frobenius algebra
$H_0(\mathcal {L}BG;\mathbb C)$ coincides with the Frobenius algebra of Dijkgraaf-Witten.
But it should!

2) Let $\mathcal G$ be a finite groupoid. Let $B\mathcal G$ its classifying space.
In~\cite[Theorem 8.7]{Strickland:duality}, Strickland showed that the suspension spectrum of
$B\mathcal G$ localized with respect to $K(n)$ is an unital and counital Frobenius object.
Roughly, the comultiplication is the diagonal map $B\mathcal G\rightarrow
B\mathcal G\otimes B\mathcal G$.
The counit is the projection map $B\mathcal G\rightarrow *$.
On the contrary, the multiplication is given by the transfer map of the evaluation fibration
$(ev_0,ev_1):B\mathcal G^{[0,1]}\twoheadrightarrow B\mathcal G\otimes B\mathcal G$.
The unit is the transfert of the projection map  $B\mathcal G\rightarrow *$.
As pointed by Strickland~\cite[p. 733]{Strickland:duality},
this Frobenius structure has ``striking formal similarities''
with the case of manifolds (See 3) of example ~\ref{exemples d'algebres de Frobenius}).
In particular, $K(n)_*(B\mathcal G)$ is an unital and counital Frobenius algebra in
the category of $K(n)_*(pt)$-modules.

Let $G$ be a finite group. The inertia groupoid $\Lambda G$ of $G$ (~\cite[Definition 8.8]{Strickland:duality}
or~\cite[Definition 2.49]{Adem-Chen-Ruan:livre}) is a finite groupoid whose classifying space $B\Lambda G$
is homotopy equivalent to the free loop space on $BG$, ${\mathcal L}BG$.
Therefore applying Strickland results, we obtain that $K(n)_*(\mathcal{L}BG)$ is an unital and counital
Frobenius algebra like in part 2) of our Corollary~\ref{unite pour les groupes finis}.
We have not checked that Strickland Frobenius algebra coincides with ours.
But Strickland definitions of the comultiplication, of the counit, of the multiplication and of the unit
are very different from the definitions using cobordism given in this paper.
\end{rem}
% LUX NE PARLE PAS DE TQFT.
%{\bf Remarks.} In \cite{Lupercio-Uribe-Xicotencatl:Orbistring} Lupercio, Uribe and Xicotencatl prove that the TQFT $H_0(\mathcal LBG;\mathbb Q)$ is isomorphic to Dijkgraaf-Witten Frobenius algebra. Let $\tilde{G}$ be the set of conjugacy classes of $G$, we recall that 
%$$H_0(\mathcal LBG,\mathbb Q)\cong \mathbb Q[\tilde{G}].$$
%The product $\mu$ on $H_0(\mathcal LBG,\mathbb Q)$ is given by the formula
%$$\mu(\tilde{g}\otimes \tilde{g'})=\sum_{h\in G}\widetilde{hgh^{-1}g'}$$
%and the coproduct $\delta$ by
%$$\delta(\tilde{g})=\sum_{(g_1,g_2)\in G^{\times 2}\diagup g_1g_2=g}\tilde{g_1}\otimes\tilde{g_2}.$$
%%%%%%%%%%%%%%%%%%%%%%%%%%%%%%%%%%%%%%%%%%%%%%%%%%%%%%%%%%%%%%%%%%%%%%%%%%%%%%%%%%%%%%%%%%%%%%%%%%%%%%%%%%%%%%%%%%%
\begin{rem}
By Proposition~\ref{toutes les fibres homotopes}, Theorem~\ref{main theorem stable version} can be extended to path-connected spaces $X$ such that $\Omega X$ is (stably) equivalent to a finite CW-complex.
In this case, we are only able to get a non-counital non-unital homological conformal
field theory. Although we have not prove it, we believe that this structure is trivial
except when $H_*(\Omega X;\mathbb{F})$ is concentrated in degre 0.
\end{rem}
\section{Frobenius algebras and symmetric Frobenius algebras}
In this section, we recall the notion of symmetric Frobenius and Frobenius algebras and prove that the homology of a connected compact Lie group is a symmetric Frobenius algebra.
\subsection{Frobenius algebras}
Frobenius algebras arise in the representation theory of algebras. 
\\
{\it A Frobenius algebra is a finite dimensional unitary associative algebra $A$ over a field $R$, equipped with a bilinear form called the Frobenius form
$$<-,->:A\otimes A\rightarrow \mathbb{F}$$  
that satisfies the Frobenius identity 
$$<a,bc>=<ab,c>$$
and is non-degenerate.}  
\\
They can also be characterized by the existence of an isomorphism
$\lambda_L:A\cong A^\vee$ of left-$A$-modules. In fact the existence of such an isomorphism implies 
the existence of a coassociative counital coproduct 
$$\delta:A\rightarrow A\otimes A$$
which is a morphism of $A$-bimodules~\cite[Thm 2.1]{Abrams:modulesfrobenius}.
Here the $A$-bimodule structure on $A\otimes A$ is the {\it outer} bimodule structure
given by
$$
a\dot(b\otimes b')\dot c:=ab\otimes b'c
$$
for $a$, $b$, $b'$ and $c\in A$.
%This mixing between the multiplicative and the comultiplicative structure will be treated later in terms of algebras over a prop.   
\\
When $\lambda_L:A\cong A^\vee$ is an isomorphism of $A$-bimodule, the algebra $A$ is called symmetric Frobenius, this is equivalent to requiring that the Frobenius form is symmetric $<a,b>=<b,a>$. Let us notice that a commutative Frobenius algebra is always a symmetric Frobenius algebra.
\begin{ex}\label{example d'algebre symetrique}
 1) A classical example is given by algebras of matrices $M_n(\mathbb{F})$ where 
$$<A,B>=tr(A.B).$$
2) Let $G$ be a finite group then its group algebra $\mathbb{F}[G]$ is a non commutative symmetric Frobenius algebra.By definition, the group ring $\mathbb{F}[G]$ admits the set $\{g\in G\}$ as a basis.
Denote by $\delta_g$ the dual basis in $\mathbb F[G]^\vee$.
The linear isomorphism $\lambda_L:\mathbb{F}[G]\rightarrow \mathbb{F}[G]^\vee$,
sending $g$ to $\delta_{g^{-1}}$ is an isomorphism of $\mathbb{F}[G]$-bimodules.
\\
3) Let $M^d$ be a compact oriented closed manifold of dimension $d$ then the singular homology $H_{*+d}(M^d,\mathbb F)$ is a commutative Frobenius algebra of
(lower) degree $+d$. The product is the intersection product, the coproduct is induced by the diagonal $\Delta:M^d\rightarrow M^d\times M^d$. The counit is induced by the projection map $M\rightarrow *$. The unit is the orientation class
$[M]\in H_d(M)$.   
\end{ex}
\subsection{Hopf algebras}
Let $H$ be a finite dimensional Hopf algebra over a field ${\mathbb F}$.
A left (respectively right) {\it integral} in $H$ is an element $l$ of $H$ such
that $\forall h\in H$, $h\times l=\varepsilon (h) l$ (respectively $l\times h=\varepsilon(h) l$).
A Hopf algebra $H$ is {\it unimodular} if there exists a non-zero element
$l\in H$ which is both a left and a right integral in $H$.
\begin{ex}\label{exemples d'algebres de Frobenius}
If $G$ is a finite group,
$\sum_{g\in G} g$ is both a left and right integral in the group algebra
${\mathbb F}[G]$.
\end{ex}
The set $\int$ of left (respectively right)
integrals in the dual Hopf algebra $H^\vee$
is a ${\mathbb F}$-vector space of dimension $1$~\cite[Corollary 5.1.6 2)]{Sweedler:livre}.
Let $\lambda$ be any non-zero left (respectively right) integral in $H^\vee$.
The morphism of left (respectively right) $H$-modules, $H\buildrel{\cong}\over\rightarrow H^{\vee}$
sending $1$ to $\lambda$ is an
isomorphism~\cite[Proof of Corollary 5.1.6 2)]{Sweedler:livre}.
So a finite dimensional Hopf algebra is always a Frobenius algebra,
but not always a symmetric Frobenius algebra as the following theorem shows.
\begin{theor}(Due to~\cite{Oberst-Schneider}. Other proofs are given
in ~\cite{Farnsteiner} and ~\cite[p. 487 Proposition]{Lorenz:representationHopf}.
See also~\cite{Humphreys:symHopf}.)\label{symmetrique ssi unimodulaire}
A Hopf algebra $H$ is a symmetric Frobenius algebra if and only if
$H$ is unimodular and its antipode $S$ satisfies $S^2$ is an inner
automorphism of $H$.
\end{theor}
Assume that $H$ is unimodular and that
$S^2$ is an inner automorphism of $H$.
Let $u$ be an invertible element
$u\in H$ such that
$\forall h\in H$, $S^2(h)=uhu^{-1}$.
Let $\lambda$ be any non-zero left integral in $H^\vee$.
Then $\beta(h,k):=\lambda (hku)$ is a non-degenerate symmetric
bilinear form~\cite[p. 487 proof of Proposition]{Lorenz:representationHopf}.
\begin{ex} Let $G$ be a finite group.
Since $S^2=Id$ and since $\delta_1$ is a left integral for
${\mathbb F}[G]^\vee$, we recover that the
the linear isomorphism $\mathbb{F}[G]\rightarrow\mathbb{F}[G]^\vee$, sending $g$ to $\delta_1(-g)=\delta_{g^{-1}}$ is an isomorphism of $\mathbb{F}[G]$-bimodules.
\end{ex}
All the previous results extend to graded Hopf algebras.
We need the following.
\begin{proposition}
Let $H$ be a cocommutative (lower) graded Hopf algebra such that
\\
i) $H_0\cong\mathbb{F}$ ($H$ is connected),
\\
ii) $H$ is concentrated in degrees between $0$ and $d$ and
\\
iii) $H_d\neq {0}$.
\\
Then there exists an isomorphism $H\buildrel{\cong}\over\rightarrow
H^\vee$
of $H$-bimodules (necessarily of (lower) degree $-d$:$H_p\buildrel{\cong}\over\rightarrow
(H_{d-p})^\vee$), i.e.
 $H$ is a symmetric Frobenius algebra of (lower) degree $-d$.
\end{proposition}
\begin{proof}
Since $H$ is cocommutative, $S^2=Id$.
For degree reasons, any element $l\in H_d$ is both a left and a right integral.
Since $H_d\neq {0}$, $H$ is unimodular.
Therefore, by Theorem~\ref{symmetrique ssi unimodulaire}, $H$ is a symmetric Frobenius algebra.
\end{proof}
Notice that an ungraded cocommutative Hopf algebra
can be non unimodular~\cite[p. 487-8, Remark and Examples (1) and (4)]{Lorenz:representationHopf}.
Therefore the previous Proposition is false without condition i).

\subsection{The case of compact Lie groups} Let us come to our motivational example. Let us take a connected compact Lie group $G$ of dimension $d$.
Let $m$ denotes the product and $Inv$ the inverse map of $G$.
As a manifold one knows that $H_{*+d}(G,\mathbb F)$ together with the intersection product is a commutative Frobenius algebra of lower degree $+d$.
Moreover its homology  together with the coproduct $\Delta_*$ and the Pontryagin product $m_*$ is a finite dimensional connected cocommutative Hopf algebra, the antipode map $S$ is given by $S=Inv_*$.
Therefore using the above Proposition, $H_*(G)$ together with the
Pontryagin product
is a symmetric Frobenius algebra of (lower) degree $-d$.
Denote by 
$$\Theta:H_p(G)\buildrel{\cong}\over\rightarrow
(H_{d-p}(G))^\vee\cong H^{d-p}(G)$$ an isomorphism of $H_*(G)$-bimodules of (lower) degree $-d$.
Let $\eta_!$ be the Poincar\'e dual of the canonical inclusion
$\eta:\{1\}\subset G$.
Since $\eta_!$ is a non-zero element of $H_d(G)^\vee$, there exists a non-zero
scalar $\alpha\in\mathbb{F}$ such that $\eta_!=\alpha.\Theta(1)$.
Therefore, we have obtained
\begin{theor}\label{homologie groupe de Lie algebre symetrique}
The singular homology of a connected compact Lie group $G$ taken with coefficients in a field and equipped with the bilinear pairing 
$$< a, b >:=\eta_!(m_*(a\otimes b))$$
is a symmetric Frobenius algebra.
\end{theor}
\begin{proof}[Second Proof] Let us give a more topological proof. The Frobenius relation is automatically satisfied.
This bilinear form is closely related to the intersection product.
In fact one can consider the following pull-back diagram :
$$\xymatrix{G\ar[d]^{\Delta}\ar[r]^p&\{e\}\ar[d]^\eta\\
G\times G\ar[r]_{m'} &G}
$$
where $m'(h,g)=h^{-1}g$.
Using properties of Poincar\'e duality with respect to Pull-back diagrams and denoting by $\bullet$ the intersection product (Poincar\'e dual of $\Delta$) one gets :
$$p_*(a\bullet b)=\eta_!\circ m_*(S(a)\otimes b).$$
Using $m''(h,g)=hg^{-1}$, we find the relation
$$p_*(a\bullet b)=\eta_!\circ m_*(a\otimes S(b)).$$
Therefore, 
$$< a, b >:=\eta_!(m_*(a\otimes b))=p_*(S^{-1}(a)\bullet b)=
p_*(a\bullet S^{-1}(b)).$$
Since the intersection product is commutative,
the pairing $<a,b>$ is symmetric.
We recall that $p_*(a\bullet b)$ is the Frobenius form associated to the Frobenius structure on $H_*(G,\mathbb F)$ given by Poincar\'e duality. Therefore
since $S^{-1}$ is an isomorphism,
the pairing $<a,b>$ is non-degenerate. 
\end{proof}

\section{Hochschild cohomology}\label{hochschild cohomology}
Let $G$ be a finite group. The group ring $\mathbb{F}[G]$ is equipped with an isomorphism
$$\lambda_L:\mathbb{F}[G]\buildrel{\cong}\over\rightarrow\mathbb{F}[G]^\vee$$ of $\mathbb{F}[G]$-bimodules and is therefore
a symmetric Frobenius algebra (Example~\ref{example d'algebre symetrique} 2)).
So we have the induced isomorphism in Hochschild cohomology
$$
HH^*(\mathbb{F}[G];\lambda_L):HH^*(\mathbb{F}[G];\mathbb{F}[G])\buildrel{\cong}\over\rightarrow
HH^*(\mathbb{F}[G];\mathbb{F}[G]^\vee).
$$
Our inspirational theorem in section 2 says that the Gerstenhaber algebra
$HH^*(\mathbb{F}[G];\mathbb{F}[G])$ is a Batalin-Vilkovisky algebra.
Here the $\Delta$ operator is the Connes coboundary map $H(B^\vee)$
on $HH^*(\mathbb{F}[G];\mathbb{F}[G]^\vee)$.
In this section, we extends our inspirational theorem for finite groups to connected compact
Lie groups:
%and we lift this structure at the chain level. Then we review some recent results on the Hochschild cohomology of symmetric Frobenius algebras and ask for a geometric interpretation in the case when the symmetric Frobenius algebra is the group algebra of a finite group or the singular chains of a Lie group.  
\begin{theor}\label{Homologie de Hochschild groupe de Lie}
Let $G$ be a connected compact Lie group of dimension $d$.
Denote by $S_*(G)$ the algebra of singular chains of $G$.
Consider Connes coboundary map $H(B^\vee)$
on the Hochschild cohomology of $S_*(G)$ with coefficients in its dual,
$HH^{*}(S_*(G);S^*(G))$.
then there is an isomorphism of graded vector spaces of upper degree $d$
$$\mathbb{D}^{-1}:HH^p(S_*(G);S_*(G))\buildrel{\cong}\over\rightarrow
HH^{p+d}(S_*(G);S^*(G))$$
such that the Gerstenhaber algebra $HH^*(S_*(G);S_*(G))$
equipped with the operator
$\Delta=\mathbb{D}\circ H(B^\vee)\circ \mathbb{D}^{-1}$ is a Batalin-Vilkovisky
algebra. 
\end{theor}
The proof of this theorem relies on Propositions 10 and 11 of~\cite{Menichi:BV_Hochschild} and
on the following three Lemmas.
Denote by $\eta:\{1\}\hookrightarrow G$ the inclusion of the trivial
group into $G$.
\begin{lem}\label{homologie groupe algebre de Frobenius}
The morphism of left $H_*(G)$-modules
$$
H_p(G)\rightarrow H_{d-p}(G)^\vee,a\mapsto a.\eta_!
$$
is an isomorphism.
\end{lem}
This Lemma is a particular case of Theorem~\ref{homologie groupe de Lie algebre symetrique}
in the previous section.
But we prefer to give an independent and more simple proof of this Lemma.
In fact in this section, we implicitly~\cite[Proof of Proposition 10]{Menichi:BV_Hochschild}
give two morphisms of $S_*(G)$-bimodules
$$
S_*(G)\buildrel{\simeq}\over\rightarrow P\buildrel{\simeq}\over\leftarrow S_*(G)^\vee
$$
which induce isomorphisms in homology.
In particular, passing to homology, we obtain a third proof of
Theorem~\ref{homologie groupe de Lie algebre symetrique}.
\begin{proof}[Proof of Lemma~\ref{homologie groupe algebre de Frobenius}]
By~\cite[Proof of Corollary 5.1.6 2)]{Sweedler:livre}, since $H_*(G)$ is a finite dimensional
Hopf algebra, $H_*(G)$ together with the Pontryagin product is a Frobenius algebra:
there exists an isomorphism of left $H_*(G)$-modules
$$\Theta:H_*(G)\buildrel{\cong}\over\rightarrow
(H_{*}(G))^\vee.$$
Since $H_*(G)$ is concentrated in degrees between $0$ and $d$ and since
$H_0(G)$ and $H_d(G)$ are two non trivial vector spaces,
the isomorphism $\Theta$ must be of (lower) degree $-d$.
Let $\eta_!$ be the Poincar\'e dual of the canonical inclusion $\eta:\{1\}\subset G$.
Since $\eta_!$ is a non-zero element of $H_d(G)^\vee$, there exists a non-zero
scalar $\alpha\in\mathbb{F}$ such that $\eta_!=\alpha.\Theta(1)$.
Therefore the
morphism of left $H_*(G)$-modules
$$\alpha\Theta:H_*(G)\buildrel{\cong}\over\rightarrow
(H_{*}(G))^\vee$$ is an isomorphism.
This is the desired isomorphism since $\alpha\Theta(1)=\eta_!$.
\end{proof}
Let $M$ and $N$ be two oriented closed smooth manifolds of dimensions $m$ and $n$. Let $G$ be a connected compact Lie group acting smoothly
on $M$ and $N$.
Let $f:M\rightarrow N$ be a smooth $G$-equivariant map.
Then we have a {\it Gysin equivariant map} in homology
$$
(EG\times_G f)_!:H_*(EG\times_G N)\rightarrow H_{*+m-n}(EG\times_G M)
$$
and a Gysin equivariant map in cohomology~\cite[Theorem 6.1]{Kawakubo:transformationgroups}
$$
(EG\times_G f)^!:H^*(EG\times_G M)\rightarrow H^{*+n-m}(EG\times_G N).
$$
Similarily to integration along the fiber,
Gysin equivariant maps are natural with respect to pull-backs and products.
% J'aimerai bien une reference !!!!!
\begin{lem}\label{shriek commute avec delta}
Let $K$ be a connected compact Lie group.
Suppose that $EG\times_G M$ and $EG\times_G N$ are two left $K$-spaces.
Suppose also that $$EG\times_G f:EG\times_G M\rightarrow EG\times_G N$$ is $K$-equivariant.
Then the Gysin equivariant map
$$
(EG\times_G f)_!:H_*(EG\times_G N)\rightarrow H_{*+m-n}(EG\times_G M)
$$
is a morphism of left $H_*(K)$-modules.
In particular, if
$K$ is the circle, $$\Delta\circ (EG\times_G f)_!=(EG\times_G f)_!\circ\Delta.$$
\end{lem}
\begin{proof}
Consider the pull-back diagram:
$$
\xymatrix{
K\times EG\times_G M\ar[rr]^{K\times EG\times_G f}\ar[d]_{action}
&& K\times EG\times_G N\ar[d]^{action}\\
EG\times_G M\ar[rr]_{EG\times_G f}
&& EG\times_G N
}
$$
where $action$ are the actions of $K$ on $EG\times_G M$ and $EG\times_G N$.
By naturality of Gysin equivariant map with respect to this pull-back and to products,% Preuve ?
we obtain the commutative diagram
$$
\xymatrix{
H_*(K)\otimes H_*(EG\times_G M)\ar[d]
&& H_*(K)\otimes H_*(EG\times_G N)\ar[d]\ar[ll]^{H_*(K)\otimes (EG\times_G f)_!}\\
H_*(K\times EG\times_G M)\ar[d]_{H_*(action)}
&& H_*(K\times EG\times_G N)\ar[d]^{H_*(action)}\ar[ll]^{(K\times EG\times_G f)_!}\\
H_*(EG\times_G M)
&& H_*(EG\times_G N)\ar[ll]_{(EG\times_G f)_!}
}$$
\end{proof}
Let us denote by $G^{ad}$ the left $G$-space obtained by the conjugation
action of $G$ on itself. The inclusion $\eta:\{1\}\hookrightarrow G^{ad}$ is a
$G$-equivariant embedding of dimension $d$. Therefore, we have a Gysin equivariant
morphism
$$(EG\times_G\eta)^!:H^*(BG)\rightarrow H^{*+d}(EG\times_G G^{ad}).$$
\begin{lem}\label{preserve les unites des algebres}
The morphism
$$
H^d(E\eta\times_\eta G^{ad}):H^d(EG\times_G G^{ad})\rightarrow H^d(G)
$$
maps $(EG\times_G\eta)^!(1)$ to $\eta!\in H_d(G)^\vee$.
\end{lem}
\begin{proof}
Consider the two commutative squares
$$
\xymatrix{
\{1\}\ar[d]_\eta\ar[rr]
&& BG\ar[d]^{EG\times_G\eta}\\
 G^{ad}\ar[rr]_-{E\eta\times_\eta G^{ad}}\ar[d]
&& EG\times_G G^{ad}\ar[d]^p\\
\{1\}\ar[rr]
&& BG
}
$$
The lower square is a fiber product since $G^{ad}$ is the fiber of the fiber
bundle $p:EG\times_G G^{ad}\twoheadrightarrow BG$.
The total square is a fiber product since $EG\times_G\eta$ is a section of $p$.
Therefore the upper square is also a fiber product.
By naturality of Gysin equivariant morphism with respect to fiber products,
we obtain the commutative square
$$
\xymatrix{
\mathbb{F}\ar[d]_{\eta^!}
&& H^*(BG)\ar[ll]\ar[d]^{(EG\times_G\eta)^!}\\
 H^*(G^{ad})
&& EG\times_G G^{ad}\ar[ll]^{H^*(E\eta\times_\eta G^{ad})}
}
$$
Therefore $H^*(E\eta\times_\eta G^{ad})\circ (EG\times_G\eta)^!(1)=\eta^!(1)$.
\end{proof}
\begin{proof}[Proof of Theorem~\ref{Homologie de Hochschild groupe de Lie}]
Let $M$ be a left $G$-space.
Let $B(*;G;M)$ denote the simplicial Bar
construction~\cite[p. 31]{MayClassifying}.
Recall that its space of $n$-simplices $B(*;G;M)_n$ is the product
$G^{\times n}\times M$.
The realisation $\vert-\vert$ of this simplicial space,
$\vert B(*;G;M)\vert$ is homeomorphic to the Borel Construction
$EG\times_G M$ when we set
$EG:=\vert B(*;G;G)\vert$~\cite[p. 40]{MayClassifying}.
Let $\Gamma G$ be the cyclic Bar construction of $G$~\cite[7.3.10]{LodayJ.:cych}.
The continuous application
$$\Phi:\Gamma G\rightarrow B(*;G;G^{ad})$$
$$(m,g_1,\cdots,g_n)\mapsto (g_1,\cdots,g_n,g_1\dots g_n m )$$
is an isomorphism of simplicial spaces.
(In homological algebra, the same isomorphism~\cite[7.4.2]{LodayJ.:cych} proves that
Hochschild homology and group homology are isomorphic).
The simplicial space $\Gamma G$ is in fact a cyclic space.
Let us consider the structure of cyclic space on $B(*;G;G^{ad})$
such that $\Phi$ is an isomorphisn of cyclic spaces.
Recall that $\eta:\{1\}\hookrightarrow G$ denote the inclusion of the trivial
group into $G$.
The composite
$$
B(*;G;*)\buildrel{B(*;G;\eta)}\over\rightarrow B(*;G;G^{ad})
\buildrel{\Phi^{-1}}\over\rightarrow \Gamma G
$$
which maps the $n$-simplex $[g_1,\dots,g_n]$ of $B(*;G;*)$
to the $n$-simplex $(g_1\dots g_n)^{-1},g_1,\dots,g_n)$ of $\Gamma G$,
is an injective morphism of cyclic spaces~\cite[7.4.5]{LodayJ.:cych}.
Here the simplicial space $B(*;G;*)$ is equipped with the structure
of cyclic space called twisted nerve and denoted $B(G,1)$
in~\cite[7.3.3]{LodayJ.:cych}.
Since realisation is a functor from cyclic spaces to $S^1$-spaces,
we obtain that

-the homeomorphism $\vert\Phi\vert:\vert\Gamma G\vert\buildrel{\cong}\over\rightarrow
EG\times_G G^{ad}$ is $S^1$-equivariant and

-the inclusion $EG\times_G \eta: BG\hookrightarrow EG\times_G G^{ad}$ is also $S^1$-equivariant.

Therefore in homology we have 
$H_*(\vert\Phi\vert)\circ\Delta=\Delta\circ H_*(\vert\Phi\vert)$.
By Lemma~\ref{shriek commute avec delta}, we also have
$(EG\times_G \eta)_!\circ\Delta=\Delta\circ (EG\times_G \eta)_!$.
Finally by dualizing,
in cohomology, we have
\begin{equation}\label{shriek de la section commute avec Delta}
 \Delta^\vee\circ H^*(\vert\Phi\vert)=H^*(\vert\Phi\vert)\circ\Delta^\vee
\quad\text{and}\quad
\Delta^\vee \circ (EG\times_G \eta)^!=(EG\times_G \eta)^!\circ \Delta^\vee.
\end{equation}
Let $j:G\rightarrow\Gamma G$,
$g\mapsto (g,1,\dots,1)$, the inclusion of the constant simplicial space $G$ into $\Gamma G$.
Consider the morphism of simplicial spaces
$$
B(*;\eta;G^{ad}):G^{ad}=B(*;*;G^{ad})\rightarrow B(*;G;G^{ad})
$$
$$
g\mapsto (1,\dots,1,g).
$$
Obviously $\Phi\circ j=B(*;\eta;G^{ad})$. Therefore we have the commutative diagram
of topological spaces
$$
\xymatrix{
& G \ar[dr]^{\vert j\vert}\ar[dl]_{E\eta\times_\eta G^{ad}}\\
EG\times_G G^{ad}
&& \Gamma G\ar[ll]^{\vert\Phi\vert}
}
$$
In~\cite{Burghelea-Fiedorowicz:chak,Goodwillie:cychdfl}, Burghelea, Fiedorowicz and Goodwillie proved that
there is an isomorphism of vector spaces between $H^*(\mathcal{L}BG)$ and
$HH^*(S_*(G);S^*(G))$.
More precisely, they give a $S^1$-equivariant weak homotopy equivalence
$\gamma:\vert \Gamma G\vert\rightarrow \mathcal{L}BG$~\cite[7.3.15]{LodayJ.:cych}.
And they construct an isomorphism
$BFG: H^*(\vert\Gamma G\vert)\rightarrow HH^*(S_*(G);S^*(G))$.
Denote by $\eta_{S_*(G)}:\mathbb{F}\hookrightarrow S_*(G)$ the unit of the algebra
 $S_*(G)$.
It is easy to check that
$$
HH^*(\eta_{S_*(G)};S^*(G))\circ BFG= H^*(\vert j\vert).
$$
Therefore we have the commutative diagram.
$$
\xymatrix{
&& H^*(G) \\
H^*(EG\times_G G^{ad})\ar[rr]_{H^*(\vert\Phi\vert)}\ar[urr]^{H^*(E\eta\times_\eta G^{ad})}
&& H^*(\Gamma G)\ar[u]_{H^*(\vert j\vert)}\ar[rr]_-{BFG}
&&  HH^*(S_*(G);S^*(G)) \ar[ull]_{HH^*(\eta_{S_*(G)};S^*(G))}
}
$$
Note that this diagram is similar to  the diagram in the proof of Theorems 20 and 21
in~\cite{Menichi:BV_Hochschild}.

Let $m\in HH^d(S_*(G);S^*(G))$ be $BFG\circ H^*(\vert\Phi\vert)\circ (EG\times_G \eta)^! (1)$.
By the above commutative diagram and Lemma~\ref{preserve les unites des algebres},
$$HH^d(\eta_{S_*(G)};S^*(G))(m)=H^d(E\eta\times_\eta G^{ad})\circ (EG\times_G \eta)^! (1)
=\eta_!.$$
Therefore, by Lemma~\ref{homologie groupe algebre de Frobenius},
The morphism of left $H_*(G)$-modules
$$
H_*(G)\rightarrow H_*(G)^\vee,a\mapsto a.HH^d(\eta_{S_*(G)};S^*(G))(m)
$$
is an isomorphism.
Therefore by Proposition 10 of~\cite{Menichi:BV_Hochschild},
we obtain that the morphism of $HH^*(S_*(G);S_*(G))$-modules
$$\mathbb{D}^{-1}:HH^p(S_*(G);S_*(G))\buildrel{\cong}\over\rightarrow
HH^{p+d}(S_*(G);S^*(G),$$
$$ a\mapsto a.m$$
is an isomorphism.

The isomorphism of Burghelea, Fiedorowicz and Goodwillie
$$BFG: H^*(\vert\Gamma G\vert)\rightarrow HH^*(S_*(G);S^*(G))$$
is compatible with the action of the circle on $\vert\Gamma G\vert$
and the dual of Connes boundary map $H(B^\vee)$:
this means that $BFG\circ \Delta^\vee= H(B^\vee)\circ BFG$.
Since $\Delta^\vee$ is a derivation for the cup product, $\Delta^\vee (1)=0$.
Therefore using (\ref{shriek de la section commute avec Delta}),
\begin{multline*}
H(B^\vee)(m)=BFG\circ\Delta^\vee \circ H^*(\vert\Phi\vert)\circ (EG\times_G \eta)^!(1)\\
=BFG\circ H^*(\vert\Phi\vert)\circ (EG\times_G \eta)^!\circ\Delta^\vee (1)=0.
\end{multline*}
So, by applying Proposition 11 of~\cite{Menichi:BV_Hochschild}, we obtain the desired
Batalin-Vilkovisky algebra structure.
\end{proof}
\begin{rem}
Denote by $s:BG\hookrightarrow\mathcal{L}BG$ the inclusion of the constants loops into
$\mathcal{L}BG$. If we equipped $BG$ with the trivial $S^1$-action then $s$ is $S^1$-equivariant.
The problem is that we don't know how to define $s_!$ directly.
Instead, in the proof of Theorem~\ref{Homologie de Hochschild groupe de Lie}, we define
$(EG\times_G\eta)_!$. And using a simplicial model of $EG$, we give $S^1$-actions on $BG$
and $EG\times_G G^{ad}$ such that
$EG\times_G\eta :BG\hookrightarrow EG\times_G G^{ad}$ is $S^1$-equivariant.
\end{rem}

\section{a string bracket in cohomology}
In this section,

-We show that the $\Delta$ operators of the Batalin-Vilkovisky algebras given by
Corollaries~\ref{Structure BV en cohomologie groupe de Lie} and~\ref{Structure BV en cohomologie groupe fini}, coincide with the  $\Delta$ operator induced by the action of $S^1$ on
$\mathcal{L}X$ (Proposition~\ref{comparaison des deltas})

-from the Batalin-Vilkovisky algebra given by
Corollaries~\ref{Structure BV en cohomologie groupe de Lie} and~\ref{Structure BV en cohomologie groupe fini},
we define a Lie bracket on the $S^1$-equivariant
cohomology $H^*_{S^1}(\mathcal{L}X)$ when $X$ satisfying the hypothesis of the main theorem
(Theorem~\ref{string bracket pour les classifiants}).
The definition of this string bracket in  $S^1$-equivariant cohomology
%(Theorem~\ref{string bracket pour les classifiants})
is similar but not identical
to the definition of the Chas-Sullivan string bracket in homology
(Theorem~\ref{Chas-Sullivan string bracket}).

-from the Batalin-Vilkovisky algebra given by our inspirational theorem in section 2 when $G$ is a finite
group or given by Theorem~\ref{Homologie de Hochschild groupe de Lie} when $G$ is a connected compact Lie group,
we define a Lie bracket on the cyclic cohomology $HC^*(S_*(G))$ of the singular chains on $G$
(Theorem~\ref{string bracket sur la cohomologie cyclique}).

We consider 
$$act:S^1\times\mathcal LX\rightarrow \mathcal LX$$ 
the reparametrization map defined by $act(\theta,\gamma(-)):=\gamma(-+\theta)$.
Let $[S^1]\in H_1(S^1)$ be the fundamental class of the circle.
\begin{proposition}\label{comparaison des deltas}
The operator $\Delta: H^*(\mathcal{L}X)\rightarrow H^{*-1}(\mathcal{L}X)$
of the Batalin-Vilkovisky algebras given by
Corollaries~\ref{Structure BV en cohomologie groupe de Lie} and~\ref{Structure BV en cohomologie groupe fini}, is the dual of the composite
$$
H_*(\mathcal{L}X)\buildrel{[S^1]\otimes -}\over\rightarrow H_*(S^1)\otimes H_*(\mathcal{L}X)
\buildrel{act_*}\over\rightarrow
H_*(\mathcal{L}X),\quad x\mapsto act_*([S^1]\otimes x).
$$
\end{proposition}
\begin{proof}
For $\varepsilon=0$ or $1$, let $i_\varepsilon:S^1\hookrightarrow S^1\times [0,1]$,
$x\mapsto (x,\varepsilon)$ be the two canonical inclusions.
Consider the cylinder $C:=S^1\times [0,1]$ as the cobordism $F_{0,1+1}$:
$$S^1\buildrel{i_0}\over\hookrightarrow S^1\times [0,1]
\buildrel{i_1}\over\hookleftarrow S^1.
$$
Let $\mu(C): H_*BDiff^+(C,\partial)\otimes H_*(\mathcal{L}X)\rightarrow H_*(\mathcal{L}X)$
be the evaluation map associated to $C$.

Recall from Section~\ref{section structure BV}, that there is a canonical
homotopy equivalence
$$f\mathcal{D}_2(1)\buildrel{\thickapprox}\over\rightarrow BDiff^+(C,\partial).$$
It is also easy to construct a canonical homotopy equivalence 
$S^1\buildrel{\thickapprox}\over\rightarrow f\mathcal{D}_2(1)$.
Denote by $B(\sigma):S^1\buildrel{\thickapprox}\over\rightarrow BDiff^+(C,\partial)$ the composite of these two homotopy equivalences.

The operator $\Delta: H^*(\mathcal{L}X)\rightarrow H^{*-1}(\mathcal{L}X)$
of the Batalin-Vilkovisky algebras given by
Corollaries~\ref{Structure BV en cohomologie groupe de Lie} and~\ref{Structure BV en cohomologie groupe fini}, is
by definition the dual of the composite
$$
H_*(\mathcal{L}X)\buildrel{[S^1]\otimes -}\over\rightarrow H_*(S^1)\otimes
H_*(\mathcal{L}X)
\buildrel{B(\sigma)_*\otimes Id}\over\rightarrow
H_*(BDiff^+(C,\partial))\otimes H_*(\mathcal{L}X)\buildrel{\mu(C)}\over\rightarrow
H_*(\mathcal{L}X).
$$
By definition, $\rho_{in}$ is
$EDiff^+(C,\partial)\times_{Diff^+(C,\partial)}map(i_0,X)$.
Since 
$$map(i_0,X):map(C,X)\buildrel{\thickapprox}\over\rightarrow\mathcal{L}X$$
is a homotopy equivalence, $\rho_{in}$ is also a homotopy equivalence.
Therefore the shriek of  $\rho_{in}$, $\rho_{in!}$, is equal to the inverse of
$\rho_{in*}$, $\rho_{in*}^{-1}$ and $$\mu(C):=\rho_{out*}\circ \rho_{in!}=
\rho_{out*}\circ \rho_{in*}^{-1}.$$

We identify $S^1$ with $\mathbb R/\mathbb Z$.
The Dehn twist $D\in Diff^+(C,\partial)$ defined by $D(\theta,a)=(\theta+a,a)$,
is the generator of the mapping class group
$\Gamma_{0,1+1}=\pi_0(Diff^+(C,\partial))\cong\mathbb{Z}$.
By~\cite{Earle-Schatz}, the morphism of groups $\sigma:\mathbb{Z}\buildrel{\thickapprox}\over\rightarrow Diff^+(C,\partial)$ sending $n\in\mathbb{Z}$ to the $n$-th composite $D^n$
of $D$, is a homotopy equivalence.
By applying the classifying construction, we obtain the map that we denoted before
$B(\sigma)$. The morphism of groups $\sigma$ induces the commutative diagram
$$
\xymatrix{
BDiff^+(C,\partial)\times\mathcal{L}X
& EDiff^+(C,\partial)\times_{Diff^+(C,\partial)}map(C,X)
\ar[l]_-{\rho_{in}}^-{\thickapprox}\ar[r]^-{\rho_{out}}
&\mathcal{L}X\\
\mathbb{R}/\mathbb{Z}\times\mathcal{L}X\ar[u]^{B(\sigma)\times\mathcal{L}X }_{\thickapprox}
& \mathbb{R}\times_{\mathbb{Z}}map(C,X)\ar[u]^{E(\sigma)\times_\sigma map(C,X)}_{\thickapprox}
\ar[l]_{r_{0}}^{\thickapprox}\ar[r]^-{r_{1}}
&\mathcal{L}X\ar[u]_{Id}
}$$
where $r_0$ and $r_1$ are the maps defined by
$r_0([\theta,f]):=([\theta],f(-,0))$ and $r_1([\theta,f]):=f\circ i_1=f(-,1)$
for $\theta\in\mathbb{R}$ and $f\in map(C,X)$.
Since $r_0$ is equal to $\mathbb{R}\times_{\mathbb{Z}}map(i_0,X)$,
$r_0$ is also a homotopy equivalence.
Therefore by the commutativity of the above diagram
$$
\mu(C)\circ (B(\sigma)_*\otimes H_*(\mathcal{L}X))=
\rho_{out*}\circ \rho_{in*}^{-1}\circ  (B(\sigma)_*\otimes H_*(\mathcal{L}X))=
r_{1*}\circ r_{0*}^{-1}.
$$
Let $\Phi:I\times\mathbb R\times map(C,X)\rightarrow \mathcal LX$
by the map defined by $\Phi(t,\theta,f(-,-)):=f(-+t\theta,1-t)$.
Since
$\Phi(t,\theta+n,f(-,-))=\Phi(t,\theta,f\circ D^{-n}(-,-))$  for $n\in\mathbb{Z}$,
 $\Phi$ induces a well-defined homotopy
$\overline{\Phi}:I\times\mathbb R\times_\mathbb{Z} map(C,X)\rightarrow \mathcal LX$
between $r_1$ and $act\circ r_0$. Therefore
$\mu(C)\circ (B(\sigma)_*\otimes H_*(\mathcal{L}X))=
r_{1*}\circ r_{0*}^{-1}=act_*$.
\end{proof}
\begin{proposition}\label{shriek fibre principal}
Let $G$ be a topological group.
Let $p:E\twoheadrightarrow B$ be a $G$-principal bundle
(or more generally a $G$-Serre fibration in the
sense of~\cite[p. 28]{Felix-Halperin-Thomas:ratht}).
Then $p$ is a Serre fibration.
Suppose that $B$ is path-connected and that $p$ is oriented
with orientation class $w\in H_n(G)\cong H_n(p^{-1}(p(*))$.
Then the composite
$$
p_!\circ H_*(p):
H_*(E)\rightarrow H_*(B)\rightarrow H_{*+n}(E)
$$
is given by the action of $w\in H_n(G)$ on $H_*(E)$.
\end{proposition}
\begin{proof}
Consider the pull-back
\xymatrix{
G\times E\ar[r]^-{action}\ar[d]_{proj_1}
& E\ar[d]^{p}\\
E\ar[r]_p
& B
}
\\
where $proj_1$ is the projection on the first factor
and $action$ is the action of $G$ on $E$.
By naturality with respect to pull-backs,
$$
p_!\circ H_*(p)=H_*(action)\circ proj_{1!}.
$$
Let $\varepsilon:G\rightarrow *$ be the constant map to a point.
If we orient $\varepsilon$ with the orientation class $w\in H_n(G)$,
 $\varepsilon_!(*)=w$.
Since $proj_1=\varepsilon\times id_G$,
$proj_{1!}=\varepsilon_!(*)\otimes id_{G!}=w\otimes id$.
So finally,
$$
p_!\circ H_*(p)(a)=H_*(action)(w\otimes a).
$$
Under the assumption that $G$ is path-connected,
an alternative proof is to interpret $H_*(p)$ as an edge
homomorphism~\cite[XIII.(7.2)]{Whitehead:eltsoht} and to use that
the Serre spectral is a spectral sequence of $H_*(G)$-modules.
\end{proof}
By propositions~\ref{shriek fibre principal} and~\ref{comparaison des deltas}, we have:
\begin{cor} \label{Delta en fonction de Gysin}
Consider the $S^1$-principal bundle
$p:ES^1\times \mathcal LX{\rightarrow} ES^1\times_{S^1} \mathcal LX$.
The composite $p_!\circ H_*(p)$ coincides with the operator $\Delta$.
\end{cor}
\begin{lem}\label{Definition du String Bracket}
Let $\varepsilon\in\mathbb{Z}$ be an integer.
Let $\mathbb{H}$ be a Batalin-Vilkovisky algebra (not necessarily with an unit)
and $\mathcal{H}$ be a graded module.
Consider a long exact sequence of the form
$$
\cdots\rightarrow \mathbb{H}_n\buildrel{E}\over\rightarrow
\mathcal{H}_{n+\varepsilon}
\rightarrow \mathcal{H}_{n+\varepsilon-2}\buildrel{M}\over\rightarrow
\mathbb{H}_{n-1}
\rightarrow\cdots
$$
Suppose that the operator $\Delta:\mathbb{H}_i\rightarrow \mathbb{H}_{i+1}$ is equal to $M\circ E$. Then
$$\{a,b\}:=(-1)^{\vert a\vert}E\left(M(a)\cup M(b)\right),
\quad\forall a,b\in\mathcal{H}$$
defines a Lie bracket of degree $2-\varepsilon$ on $\mathcal{H}$
such that the morphism of degree $1-\varepsilon$,
$M:\mathcal{H}_n\rightarrow\mathbb{H}_{n+1-\varepsilon}$
is a morphism of graded Lie algebras:
$$\{M(a),M(b)\}=(-1)^{1-\varepsilon} M(\{a,b\}).$$
\end{lem}
\begin{proof}
Case $\varepsilon=0$.
The Lie algebra structure is proved by Chas and Sullivan in the proof of~\cite[Theorem 6.1]{Chas-Sullivan:stringtop}.
In their proof, the relation
$$
\{M(a),M(b)\}=(-1)^{\vert a\vert+1}M\circ E\left(M(a)\cup M(b)\right)=-M(\{a,b\}).
$$
appears (See also~\cite[p. 136]{Westerland:equivariant}).

General Case $\varepsilon\neq 0$. If we replace $\mathcal{H}$ by its desuspension
$s^{-\varepsilon}\mathcal{H}$, we can apply the case  $\varepsilon=0$
since $(s^{-\varepsilon}\mathcal{H})_n=\mathcal{H}_{n+\varepsilon}$.
Alternatively, to check the signs carefully, it is simpler to generalise
the computations of Chas and Sullivan.
\end{proof}
\begin{theor}\cite[Theorem 6.1]{Chas-Sullivan:stringtop}\label{Chas-Sullivan string bracket}
Let $M$ be a compact oriented smooth manifold of dimension $d$.
Then
$$\{a,b\}:=(-1)^{\vert a\vert-d}H_*(p)\left(p_!(a)\cup p_!(b)\right),
\quad\forall a,b\in H_{*}^{S^1}(\mathcal{L}M)$$
defines a Lie bracket of degree $2-d$ on $H_{*}^{S^1}(\mathcal{L}M)$
such that
$$
p_!:H_{*}^{S^1}(\mathcal{L}M)\rightarrow H_{*+1}(\mathcal{L}M)
$$
is a morphism of Lie algebras (between the string bracket and the loop bracket).
\end{theor}
\begin{proof}
By \cite[Theorem 5.4]{Chas-Sullivan:stringtop}, $\mathbb{H}_n:=H_{n+d}(\mathcal{L}M)$ is a Batalin-Vilkovisky algebra.
Consider the Gysin exact sequence in homology
$$
\cdots\rightarrow H_{n+d}(\mathcal{L}M)\buildrel{H_{n+d}(p)}\over\rightarrow
H_{n+d}^{S^1}(\mathcal{L}M)
\rightarrow H_{n+d-2}^{S^1}(\mathcal{L}M)\buildrel{p_!}\over\rightarrow
 H_{n+d-1}(\mathcal{L}M)
\rightarrow\cdots
$$
By Corollary~\ref{Delta en fonction de Gysin}, we can apply Lemma~\ref{Definition du String Bracket}
to it in the case $\varepsilon=0$.
\end{proof}
\begin{theor}\label{string bracket pour les classifiants}
Let $X$ be a space satisfying the hypothesis of Corollary~\ref{Structure BV en cohomologie groupe de Lie}
or of Corollary~\ref{Structure BV en cohomologie groupe fini}.
Denote by $d$ the top degre of $H_*(\Omega X)$.
Then
$$\{a,b\}:=(-1)^{\vert a\vert-d}p^!\left(H^*(p)(a)\cup H^*(p)(b)\right),
\quad\forall a,b\in H^{*}_{S^1}(\mathcal{L}X)$$
defines a Lie bracket of (upper) degree $-1-d$ on $H^{*}_{S^1}(\mathcal{L}X)$
such that
$$
H^*(p):H^{*}_{S^1}(\mathcal{L}X)\rightarrow H^{*}(\mathcal{L}X)
$$
is a morphism of Lie algebras.
\end{theor}
\begin{proof}
By Corollaries~\ref{Structure BV en cohomologie groupe de Lie}
or~\ref{Structure BV en cohomologie groupe fini}, 
 $\mathbb{H}_{-n}:=H^{n+d}(\mathcal{L}X)$ is a Batalin-Vilkovisky algebra
not necessarily with an unit.
Consider the Gysin exact sequence in cohomology
$$
\cdots\rightarrow H^{n+d}(\mathcal{L}X)\buildrel{p^!}\over\rightarrow
H^{n+d-1}_{S^1}(\mathcal{L}X)
\rightarrow H^{n+d+1}_{S^1}(\mathcal{L}X)\buildrel{H^{n+d+1}(p)}\over\rightarrow
 H^{n+d+1}(\mathcal{L}X)
\rightarrow\cdots
$$
By Corollary~\ref{Delta en fonction de Gysin}, we can apply Lemma~\ref{Definition du String Bracket}
to it in the case $\varepsilon=1$.
\end{proof}
\begin{rem}
Under the hypothesis of our main theorem, one can also define a Lie bracket
of degree $2+d$ on the equivariant homology $H^{S^1}_*(\mathcal{L}X)$,
exactly as Chas-Sullivan string bracket. But this bracket will
often be zero since the product on $H_*(\mathcal{L}X)$ is often trivial.
\end{rem} 
\begin{theor}\label{string bracket sur la cohomologie cyclique}
a) Let $G$ be a finite group.
Then
$$\{a,b\}:=(-1)^{\vert a\vert}\partial\left(I(a)\cup I(b)\right),
\quad\forall a,b\in HC^*(\mathbb{F}[G])$$
defines a Lie bracket of (upper) degree $-1$ on
the cyclic cohomology of the group ring of $G$
such that the composite
$$
HC^{*}(\mathbb{F}[G])\buildrel{I}\over\rightarrow
HH^{*}(\mathbb{F}[G];\mathbb{F}[G]^\vee)\buildrel{\cong}\over\rightarrow
HH^*(\mathbb{F}[G];\mathbb{F}[G])
$$
is a morphism of Lie algebras.

b) Let $G$ be a compact connected Lie group of dimension $d$.
Then
$$\{a,b\}:=(-1)^{\vert a\vert-d}\partial\left(I(a)\cup I(b)\right),
\quad\forall a,b\in HC^*(S_*(G))$$
defines a Lie bracket of (upper) degree $-1-d$ on
the cyclic cohomology of the algebra of singular chains on $G$
such that the composite
$$
HC^{*+d}(S_*(G))\buildrel{I}\over\rightarrow
HH^{*+d}(S_*(G);S^*(G))\build\rightarrow_{\cong}^{\mathbb{D}} HH^*(S_*(G);S_*(G))
$$
is a morphism of Lie algebras.
\end{theor}
\begin{proof}
For a), let $A:=\mathbb{F}[G]$ and $d:=0$. For b), let $A:=S_*(G)$.
By our inspirational theorem in section 2 or
Theorem~\ref{Homologie de Hochschild groupe de Lie}, 
 $\mathbb{H}_{-n}:=HH^{n+d}(A;A^\vee)$ is a Batalin-Vilkovisky algebra
(with an unit).

Consider Connes long exact sequence in homology~\cite[2.2.1]{LodayJ.:cych}
$$
\cdots\rightarrow HH_{n+d}(A;A) \buildrel{I}\over\rightarrow
HC_{n+d}(A)
\buildrel{S}\over\rightarrow HC_{n+d-2}(A)\buildrel{\partial}\over\rightarrow
 H_{n+d-1}(A;A)
\rightarrow\cdots
$$
Usually, the map $\partial$ is unfortunately denoted $B$.
The composite
$$
HH_n(A;A)\buildrel{I}\over\rightarrow HC_n(A)\buildrel{\partial}\over\rightarrow HH_{n+1}(A;A)
$$
coincides with Connes boundary map
$H_*(B)$ (See~\cite[Notational consistency, p. 348-9]{Weibel:inthomalg}
or~\cite[proof of Proposition 7.1]{MenichiL:BValgaccoHa} where the mixed complex considered should be the Hochschild chain
complex of $A$).
By dualizing, we have Connes long exact sequence in cohomology~\cite[2.4.4]{LodayJ.:cych}
$$
\cdots\rightarrow HH^{n+d}(A;A^\vee) \buildrel{\partial}\over\rightarrow
HC^{n+d-1}(A)
\buildrel{S}\over\rightarrow HC^{n+d+1}(A)\buildrel{I}\over\rightarrow
 H^{n+d+1}(A;A^\vee)
\rightarrow\cdots
$$
Since $H(B^\vee)=I\circ\partial$,
we can apply Lemma~\ref{Definition du String Bracket}
to it in the case $\varepsilon=1$.
\end{proof}
In part a) of Theorem~\ref{string bracket sur la cohomologie cyclique}, the group ring
$\mathbb{F}[G]$ can be replaced by any symmetric Frobenius algebra $A$.
In~\cite[Corollary 1.5]{MenichiL:BValgaccoHa}, the second author defines a Lie bracket of (upper) degree $-2$
on the negative cyclic cohomology $HC^*_-(A)$ of any  symmetric Frobenius algebra $A$.

Let $M$ be a simply-connected manifold of dimension $d$.
Let $S^*(M)$ be the algebra of singular cochains on $M$.
In~\cite[Corollary 23]{Menichi:BV_Hochschild}, the second author defines a Lie bracket of lower degree
$2-d$ on the negative cyclic cohomology $HC^*_-(S^*(M))$.
In~\cite[Conjecture 24]{Menichi:BV_Hochschild}, the second author conjectures that the Jones isomorphism
is an isomorphism of graded Lie algebras between this bracket and the Chas-Sullivan bracket
on $H^{S^1}_*(LM)$.
Dually, we conjecture
\begin{conjecture}\label{conjecture iso d'algebres de Lie}
Let $G$ be a connected compact Lie group of dimension $d$.

i)The composite of the isomorphism due to
Burghelea, Fiedorowicz and Goodwillie~\cite{Burghelea-Fiedorowicz:chak,Goodwillie:cychdfl} 
$$H^{*+d}(LBG)\buildrel{\cong}\over\rightarrow
 HH^{*+d}(S_*(G),S^*(G))$$
and of the isomorphism given by Theorem~\ref{Homologie de Hochschild groupe de Lie}
$$\mathbb{D}:HH^{*+d}(S_*(G);S^*(G))\buildrel{\cong}\over\rightarrow
HH^*(S_*(G);S_*(G))$$
is a morphism of graded algebras
between the algebra given by
Corollary~\ref{Structure BV en cohomologie groupe de Lie}
and the underlying algebra on the Gerstenhaber algebra $HH^*(S_*(G);S_*(G))$.

ii)The isomorphism due to Burghelea, Fiedorowicz and Goodwillie~\cite{Burghelea-Fiedorowicz:chak,Goodwillie:cychdfl}
$$
H_{S^1}^*(LBG)\buildrel{\cong}\over\rightarrow HC^*(S_*(G)) 
$$
is a morphism of graded Lie algebras between the Lie brackets defined by
Theorem~\ref{string bracket pour les classifiants}
and Theorem~\ref{string bracket sur la cohomologie cyclique}.
\end{conjecture}
Note that part i) of the conjecture implies part ii) (by the same arguments as in the last paragraph
of~\cite{Menichi:BV_Hochschild}).

Following Freed, Hopkins and Teleman~\cite{FreedHopkinsTeleman3,freedverlindektheory} in twisted equivariant $K$-theory ${}^\tau \! K^*_G(G^{ad})$, one can easily define
a {\it fusion product} on the equivariant cohomology
$H^{*+d}_G(G^{ad})\cong H^{*+d}(EG\times_G G^{ad})\cong H^{*+d}(LBG)$~\cite{Gruher:duality}.
In~\cite{Gruher-Westerland}, Gruher and Westerland gives an isomorphism of graded algebras
between $H^{*+d}(LBG)$ and $HH^*(S_*(G);S_*(G))$ compatible with this fusion product. 
Note that to prove this isomorphism of algebras,
they used the following theorem of Felix, the second author and Thomas.
\begin{theor}~\cite[Corollary 2]{Felix-Menichi-Thomas:GerstduaiHochcoh}\label{dualite in gerstenhaber}
Let $X$ be a simply connected space such that $H_*(X)$ is finite dimensional in each degree.
Then there is an isomorphism of Gerstenhaber algebras
$$HH^{*}(S^{*}(X),S^{*}(X))\cong HH^{*}(S_{*}(\Omega X),S_{*}(\Omega X)).$$
Here $\Omega X$ is the topological monoid of Moore pointed loops.  
\end{theor}
Let $G$ be any topological group.
By~\cite[Proposition 2.10 and Theorem 4.15]{Felix-Halperin-Thomas:ratht}, the differential graded
algebras $S_*(G)$ and $S_*(\Omega BG)$ are weakly equivalent, therefore by~\cite[Theorem 3]{Felix-Menichi-Thomas:GerstduaiHochcoh},
there is an isomorphism of Gerstenhaber algebras
$$HH^{*}(S_{*}(G),S_{*}(G))\cong HH^{*}(S_{*}(\Omega BG),S_{*}(\Omega BG)).$$
By applying Theorems~\ref{Homologie de Hochschild groupe de Lie} and~\ref{dualite in gerstenhaber}, we obtain 
\begin{theor}
Let $G$ be a connected compact Lie group.
Denote by $S^*(BG)$ the algebra of singular cochains on the classifying space of $G$.
The Gerstenhaber algebra $HH^*(S^*(BG);S^*(BG))$
is a Batalin-Vilkovisky algebra. 
\end{theor}
\section{Appendix:(not) choosing an orientation class and signs problems}\label{choix classe orientation}
In this section, we explain that in fact, we do not choose an orientation class $w_F$ for each
cobordism $F$.
Instead, we put all the possible choice of an orientation class in a prop, the prop
$\text{Det}H_1(F,\partial_{in}F;\mathbb{Z})$,
to ensure the compatibility with gluing, disjoint union, etc ....
This prop appeared first in~\cite{Costello:tcftcalabiyaucat} and~\cite{Godin:higherstring}.
As a consequence, we explain that in fact, in our main theorem for simply connected spaces, Theorem~\ref{main theoreme homologie groupe de Lie}, $H_*(\mathcal{L}X)$ is a degree $d$ (non-counital non-unital) homological
conformal field theory.

\subsection{Integration along the fiber without orienting.}
Let $F\hookrightarrow E\buildrel{p}\over\twoheadrightarrow B$ an orientable fibration that
we don't orient for the moment.
The Serre spectral sequence (Compare with Section~\ref{integration suite spectrale})
gives the linear application of degre $0$
$$
\int_{p}:H_n(F;\mathbb{F})\otimes H_*(B;\mathbb{F})\rightarrow H_{*+n}(E;\mathbb{F})
$$
which is independant of any choice of orientation class. If we choose an orientation
class $w\in H_n(F;\mathbb{F})$, then we have an oriented fibration $p$ whose integration along the fibre
$p_!$ is given by $p_!(b):=\int_{p}(w\otimes b)$ for any $b\in H_*(B;\mathbb{F})$.

\subsection{The prop $H_{-d\chi(F)}(map_*(F_{p+q}/\partial_{in}F,X))$.}\label{prop des fibres}
Let $X$ be a simply connected space such that $H_*(\Omega X)$ is finite dimensional.
In Section~\ref{proprietes umkehr maps}, we explain the formulas that orientation classes
must satisfy, in order for the integration along the fibre to be natural, compatible
with composition and product. With theses rules, the family of direct sum of graded vector spaces
$$\bigoplus_{F_{p+q}} H_{-d\chi(F)}(map_*(F_{p+q}/\partial_{in}F,X);\mathbb{F})$$
form a $\mathbb{F}$-linear graded prop.
Here the direct sum is taken over a set of representatives $F_{p+q}$ of the oriented cobordisms classes
from $\coprod_{i=1}^p S^1$ to $\coprod_{i=1}^q S^1$, whose path components
have at least one incoming-boundary component and one outgoing-boundary component
(Compare with the topological prop defined in Proposition~\ref{equivalence des trois props}).
For example, let us explain what is the composition of this prop using the notations
of Proposition~\ref{compatible au gluing} and of the paragraph Composition in Section~\ref{proprietes umkehr maps}

Let $F_{g,p+q}$ and $F_{g',q+r}'$ be two oriented cobordisms.
Let $F_{g",p+r}''$ be the oriented cobordism obtained by gluing.
Consider the two orientable fibrations
$$
f:map(F_{g",p+r}'',X)\twoheadrightarrow map(F_{g,p+q},X)
$$
given in the pull back~(\ref{produit fibre gluing}),
and $$g:=map(in,X):map(F_{g,p+q},X)\twoheadrightarrow\mathcal{L}X^{\times p}.$$
By pull back, we obtain an orientable fibration

$$
f':map_*(F_{g,p+r}''/\partial_{in}, X)\twoheadrightarrow map_*(F_{g,p+q}/\partial_{in},X)
$$
with fibre $map_*(F_{g',q+r}'/\partial_{in},X)$.
Therefore the Serre spectral sequence gives the linear isomorphism of degre $0$
$$
\int_{f'}:H_{-d\chi(F')}map_*(F_{g',q+r}'/\partial_{in},X)
\otimes H_{-d\chi(F)}  map_*(F_{g,p+q}/\partial_{in},X)
\buildrel{\cong}\over\rightarrow H_{-d\chi(F'')}map_*(F_{g,p+r}''/\partial_{in}, X)
$$
which is the composition of the prop.
Suppose that we choose an orientation class for $f$, $w_f\in H_{-d\chi(F')}map_*(F_{g',q+r}'/\partial_{in},X)$,
and an orientation class for $g$, $w_g\in  H_{-d\chi(F)}  map_*(F_{g,p+q}/\partial_{in},X)$.
Then $(g\circ f)_!$ is equal to $f_!\circ g_!$ if the orientation class chosen $w_{g\circ f}$ for $g\circ f$
is equal to $f'_!(w_g)$, i. e. to $\int_{f'}(w_f\otimes w_g)$(See paragraph Composition in Section~\ref{proprietes umkehr maps}).

Using $\int_{\rho_{in}}$ instead of $\rho_{in!}$, we obtain an linear evaluation product of degre $0$
$$
\int_F: H_{-d\chi(F)}map_*(F_{g,p+q}/\partial_{in},X)\otimes H_*(BDiff^+(F,\partial))\otimes H_*(\mathcal{L}X)^{\otimes p}
\rightarrow  H_*(\mathcal{L}X)^{\otimes q}.
$$
If we choose an orientation class $w_F\in H_{-d\chi(F)}map_*(F_{g,p+q}$, of course this new evaluation product is related to the old one
by $\mu(F)(a\otimes v)=\int_F(w_F\otimes a\otimes v)$ for any
$a\otimes v\in  H_*(BDiff^+(F,\partial))\otimes H_*(\mathcal{L}X)^{\otimes p}$.

The section~\ref{propic} shows in fact that, with $\int_F$, $H_*(\mathcal{L}X)$ is an algebra over the tensor product of props
$$
H_{-d\chi(F)}map_*(F_{g,p+q}/\partial_{in},X)\otimes H_*(BDiff^+(F,\partial)).
$$

\subsection{The orientation of a finitely generated free $\mathbb{Z}$-module $V$, $\text{Det}V$}\label{orientation groupe abelien}
Let $V$ be a free abelian group of rank $n$.
Let $\Theta=(e_1,\cdots,e_n)$ et $\Theta'=(e_1',\cdots,e_n')$ be two basis of $V$.
Let $\varphi:V\rightarrow V$ be the $\mathbb{Z}$-linear automorphism sending $e_i$ to $e_i'$.
By definition, $\Theta$ and $\Theta'$ belong to the same orientation class
if the determinant of $\varphi$, $det\varphi$, is equal to $+1$.
Equivalently, the application induced by $\varphi$, on the $n$-th exterior powers,
$\Lambda^n\varphi:\Lambda^n V\rightarrow\Lambda^n V$ is the identity map, this means that
$e_1'\wedge\cdots\wedge e_n'=e_1\wedge\cdots\wedge e_n$.
Therefore the application which maps the orientation class $[\Theta]$ of a basis
 $\Theta=(e_1,\cdots,e_n)$, on the generator $e_1\wedge\cdots\wedge e_n$ of $\Lambda^n V\cong\mathbb{Z}$ is a bijection.
So a choice of an generator of $\Lambda^n V\cong\mathbb{Z}$ is a choice of an orientation on $V$.
Let us set $$\text{det}V:=\Lambda^n V.$$

\subsection{The prop $\text{det}H_1(F,\partial_{in};\mathbb{Z})$}
In~\cite[p. 183]{Costello:tcftcalabiyaucat} and~\cite[Lemma 13]{Godin:higherstring},
Costello and Godin explain that the familly of graded abelians groups
$$\bigoplus_{F_{p+q}} \text{det}H_1(F,\partial_{in};\mathbb{Z})$$
form a $\mathbb{Z}$-linear graded prop.
Again, here the direct sum is taken over a set of representatives $F_{p+q}$ of the oriented cobordisms classes
from $\coprod_{i=1}^p S^1$ to $\coprod_{i=1}^q S^1$, whose path components
have at least one incoming-boundary component and one outgoing-boundary component.

Let us explain what is the composition of this prop.
Let $F_{g,p+q}$ and $F_{g',q+r}'$ be two oriented cobordisms.
Let $F_{g'',p+r}''$ be the oriented cobordism obtained by gluing.
By excision, $H_*(F_{g'',p+r}'',F_{g,p+q})\cong H_*(F_{g',q+r}',\partial_{in} F_{g',q+r}')$.
By Proposition~\ref{cofibre de in}
$
\tilde{H}_2(F_{g',q+r}'/\partial_{in}F_{g',q+r}')=0
$
and
$
\tilde{H}_0(F_{g,p+q}/\partial_{in}F_{g,p+q})=0
$. Therefore the long exact sequence associated to the triple
$(F_{g'',p+r}'',F_{g,p+q},\partial_{in}F_{g,p+q})$ reduces to the short exact sequence:
$$
\xymatrix@1{
0\ar[r]
&H_1(F_{g,p+q},\partial_{in}F_{g,p+q})\ar[r]
&H_1(F_{g'',p+r}'',\partial_{in}F_{g,p+q})\ar[r]
&H_1(F_{g',q+r}',\partial_{in}F_{g',q+r}')\ar[r]
&0.}
$$
Godin's situation~\cite[(33)]{Godin:higherstring} is more complicated because she considers open-closed
cobordisms and in this paper, we consider only closed cobordisms.

A short exact sequence of finite type free abelian groups
$
\xymatrix@1{
0\ar[r]
&U\ar[r]
&V\ar[r]
&W\ar[r]
&0.}
$
gives~\cite[Lemma 1 p. 1176]{Conant-Vogtmann:theokontsevich}) a natural isomorphism
$\text{det} U\otimes\text{det} V\buildrel{\cong}\over\rightarrow \text{det} W$.
Therefore, we have a canonical isomorphism
$$
\text{det} H_1(F_{g,p+q},\partial_{in}F_{g,p+q};\mathbb{Z})
\otimes\text{det} H_1(F_{g',q+r}',\partial_{in}F_{g',q+r}';\mathbb{Z})
\buildrel{\cong}\over\rightarrow \text{det} H_1(F_{g'',p+r}'',\partial_{in}F_{g'',p+q}'';\mathbb{Z}).
$$
This is the composition of the prop.

\subsection{The prop isomorphim
$\text{det}H_1(F,\partial_{in};\mathbb{Z})^{\otimes d}\otimes_\mathbb{Z}\mathbb{F}
\cong H_{-d\chi(F)}(map_*(F/\partial_{in},X);\mathbb{F})$.}

Obviously, our conformal field theory structure on $H_*(\mathcal{L}X;\mathbb{F})$ when $X$ is simply-connected
topological space, depends of a choice of a generator $w\in H_d(\Omega X;\mathbb{F})$.
So let us choose a fixed generator $w\in H_d(\Omega X;\mathbb{F})$.

Let $F_{p+q}$ be an oriented cobordism whose path connected components have at least one incoming boundary
component and also at least one outgoing component.
By Proposition~\ref{cofibre de in}, the quotient space $F/\partial_{in}$ is homotopy equivalent
to a wedge $\vee_{-\chi(F)}S^1$.
Let $f:F/\partial_{in}\buildrel{\approx}\over\rightarrow\vee_{-\chi(F)}S^1$ be a pointed homotopy
equivalence. Consider the composite of the Kunneth map, $Kunneth$, and of $H_*(map_*(f, X))$.
$$
H_*(\Omega X)^{\otimes -\chi(F)}\build\rightarrow_\cong^{Kunneth}
H_*(map_*(\vee_{-\chi(F)}S^1,X))\build\rightarrow_\cong^{H_*(map_*(f, X))}
H_*(map_*(F/\partial_{in},X)).
$$
Let $w_f$ denote the image of $w^{\otimes -\chi(F)}$ by this isomorphism.
On the other hand, let $\Theta_f$ be the image of the canonical basis of $\mathbb{Z}^{\times -\chi(F)}$
by the inverse of $H_1(f;\mathbb{Z})$:
$$
\mathbb{Z}^{\times -\chi(F)}=H_1(\vee_{-\chi(F)}S^1;\mathbb{Z})
\buildrel{H_1(f;\mathbb{Z})^{-1}}\over\rightarrow H_1(F/\partial_{in};\mathbb{Z})\cong
  H_1(F,\partial_{in};\mathbb{Z}).
$$
\begin{proposition}\label{lien entre les deux generateurs}
Let $f$ and $g:F/\partial_{in}\buildrel{\approx}\over\rightarrow\vee_{-\chi(F)}S^1$ be two pointed homotopy
equivalences.
Let $w_f$ and $w_g$ be the two corresponding generators of
$H_{-d\chi(F)}(map_*(F/\partial_{in},X);\mathbb{F})$.
Let $\Theta_f$ and $\Theta_g$ be the two associated basis of $H_1(F,\partial_{in};\mathbb{Z})$. Then $$w_f=det_{\Theta_f}(\Theta_g)^d w_g,$$
where $det_{\Theta_f}(\Theta_g)$ is the $d$-th power of the determinant of the basis $\Theta_g$
with respect to the basis $\Theta_f$.
\end{proposition}
\begin{proof}
Let $h:\vee_{-\chi(F)}S^1\buildrel{\approx}\over\rightarrow\vee_{-\chi(F)}S^1$
be a pointed homotopy equivalence such that $f$ is homotopic to the composite
$h\circ g$. Since $H_1(f;\mathbb{Z})=H_1(h;\mathbb{Z})\circ H_1(g;\mathbb{Z})$,
\begin{multline*}
det_{\Theta_f}(\Theta_g)=
det_{H_1(f;\mathbb{Z})(\Theta_f)}H_1(f;\mathbb{Z})(\Theta_g)\\
=det_{(\text{canonical basis})}H_1(h;\mathbb{Z})(\text{canonical basis})
=det H_1(h;\mathbb{Z})
\end{multline*}
where canonical basis denotes the canonical basis of $\mathbb{Z}^{\times -\chi(F)}$.
By Proposition~\ref{action sur la top class},
\begin{multline*}
w_f:=H_*(map_*(f,X))\circ Kunneth (w^{\otimes -\chi(F)})\\
=H_*(map_*(g,X))\circ H_*(map_*(h,X))\circ Kunneth (w^{\otimes -\chi(F)})\\
=det H_1(h;\mathbb{Z})^d H_*(map_*(g,X))\circ Kunneth (w^{\otimes -\chi(F)})
=det H_1(h;\mathbb{Z})^d w_g
\end{multline*}
\end{proof}
Let $\Theta_f=(e_1,\cdots,e_{-\chi(F)})$ be the basis of
$H_1(F,\partial_{in};\mathbb{Z})$ associated to a pointed homotopy
equivalence $f:F/\partial_{in}\buildrel{\approx}\over\rightarrow\vee_{-\chi(F)}S^1$.
As recalled in Section~\ref{orientation groupe abelien}, the orientation class
of $\Theta_f$, $[\Theta_f]$, corresponds to the generator
$e_1\wedge\cdots\wedge e_{-\chi(F)}$ of $\text{det} H_1(F,\partial_{in};\mathbb{Z})$.
Therefore $[\Theta_f]^{\otimes d}$ is a generator of the tensor product
 $\text{det} H_1(F,\partial_{in};\mathbb{Z})]^{\otimes d}$.
Consider the unique $\mathbb{Z}$-linear map
$$
Or(F):\text{det} H_1(F,\partial_{in};\mathbb{Z})]^{\otimes d}\rightarrow
H_{-d\chi(F)}(map_*(F/\partial_{in},X);\mathbb{F})
$$
sending the generator  $[\Theta_f]^{\otimes d}$ to the generator $w_f$.
\begin{cor}
The morphism $Or(F)$ is independant of the pointed homotopy equivalence
$f:F/\partial_{in}\buildrel{\approx}\over\rightarrow\vee_{-\chi(F)}S^1$.
\end{cor}
\begin{proof}
Let $g:F/\partial_{in}\buildrel{\approx}\over\rightarrow\vee_{-\chi(F)}S^1$
be another pointed homotopy equivalence.
Let $\Theta_g$ be the basis of $H_1(F,\partial_{in};\mathbb{Z})$
associated to $g$. By definition of orientation classes,
$[\Theta_g]=det_{\Theta_f}(\Theta_g)[\Theta_f]$.
Therefore, by Proposition~\ref{lien entre les deux generateurs},
$$
Or(F)([\Theta_g]^{\otimes d})=Or(F)\left( det_{\Theta_f}(\Theta_g)^d [\Theta_f]^{\otimes d}        \right)=det_{\Theta_f}(\Theta_g)^d w_f=w_g.
$$
\end{proof}
We claim that the familly of $\mathbb{Z}$-linear map
$$
Or(F):\text{det} H_1(F,\partial_{in};\mathbb{Z})]^{\otimes d}\rightarrow
H_{-d\chi(F)}(map_*(F/\partial_{in},X);\mathbb{F})
$$
gives a morphism of props from the tensor product of prop
$
\text{det}H_1(F,\partial_{in};\mathbb{Z})^{\otimes d}
$
to the prop
$H_{-d\chi(F)}(map_*(F/\partial_{in},X);\mathbb{F})$.
At the end of section~\ref{prop des fibres}, we explain that
$H_*(\mathcal{L}X)$ is an algebra over the tensor product of props
$$
H_{-d\chi(F)}map_*(F_{g,p+q}/\partial_{in},X)\otimes_\mathbb{F} H_*(BDiff^+(F,\partial)).
$$
Therefore $H_*(\mathcal{L}X)$ is an algebra over the tensor product of props
$$
\text{det}H_1(F,\partial_{in};\mathbb{Z})^{\otimes d}\otimes_\mathbb{Z} H_*(BDiff^+(F,\partial)).
$$
That is $H_*(\mathcal{L}X)$ is a {\it $d$-dimensional (non-unital non-counital)
conformal field theory} (in the sense of~\cite[Definition p. 169]{Costello:tcftcalabiyaucat}
or~\cite[Definition 3, Section 4.1]{Godin:higherstring}).

\bibliography{Bibliographie}

\providecommand{\bysame}{\leavevmode\hbox to3em{\hrulefill}\thinspace}
\providecommand{\MR}{\relax\ifhmode\unskip\space\fi MR }
\providecommand{\MRhref}[2]{%
  \href{http://www.ams.org/mathscinet-getitem?mr=#1}{#2}
}
\providecommand{\href}[2]{#2}
\begin{thebibliography}{10}

\bibitem{Abrams:modulesfrobenius}
Lowell Abrams, \emph{Modules, comodules, and cotensor products over {F}robenius
  algebras}, J. Algebra \textbf{219} (1999), no.~1, 201--213.

\bibitem{Adams:infiniteloop}
John~Frank Adams, \emph{Infinite loop spaces}, Annals of Mathematics Studies,
  vol.~90, Princeton University Press, Princeton, N.J., 1978.

\bibitem{Adem-Chen-Ruan:livre}
Alejandro Adem, Johann Leida, and Yongbin Ruan, \emph{Orbifolds and stringy
  topology}, Cambridge Tracts in Mathematics, vol. 171, Cambridge University
  Press, Cambridge, 2007.

\bibitem{Atiyah:TQFTsIHES}
Michael Atiyah, \emph{Topological quantum field theories}, Inst. Hautes
  \'Etudes Sci. Publ. Math. (1988), no.~68, 175--186 (1989).

\bibitem{Becker-Gottlieb:transferfiberbundle}
J.~C. Becker and D.~H. Gottlieb, \emph{The transfer map and fiber bundles},
  Topology \textbf{14} (1975), 1--12.

\bibitem{Becker-Gottlieb:transferfibration}
\bysame, \emph{Transfer maps for fibrations and duality}, Compositio Math.
  \textbf{33} (1976), no.~2, 107--133.

\bibitem{Becker-Gottlieb:history}
James~C. Becker and Daniel~Henry Gottlieb, \emph{A history of duality in
  algebraic topology}, History of topology, North-Holland, Amsterdam, 1999,
  pp.~725--745.

\bibitem{Behrend-Ginot-Noohi-Xu}
Kai Behrend, Gr{\'e}gory Ginot, Behrang Noohi, and Ping Xu, \emph{String
  topology for loop stacks}, C. R. Math. Acad. Sci. Paris \textbf{344} (2007),
  no.~4, 247--252.

\bibitem{Brown:cohgro}
K.~S. Brown, \emph{Cohomology of groups}, Graduate Texts in Mathematics,
  no.~87, Springer-Verlag, New York, 1994, Corrected reprint of the 1982
  original.

\bibitem{Burghelea-Fiedorowicz:chak}
D.~Burghelea and Z.~Fiedorowicz, \emph{Cyclic homology and algebraic
  {$K$}-theory of spaces. {II}}, Topology \textbf{25} (1986), no.~3, 303--317.

\bibitem{Chas-Sullivan:stringtop}
M.~Chas and D.~Sullivan, \emph{String topology}, preprint: math.GT/991159,
  1999.

\bibitem{Chas-Sullivan:closedLiebialgebra}
Moira Chas and Dennis Sullivan, \emph{Closed string operators in topology
  leading to {L}ie bialgebras and higher string algebra}, The legacy of {N}iels
  {H}enrik {A}bel, Springer, Berlin, 2004, pp.~771--784.

\bibitem{Clapp:duality}
M{\'o}nica Clapp, \emph{Duality and transfer for parametrized spectra}, Arch.
  Math. (Basel) \textbf{37} (1981), no.~5, 462--472.

\bibitem{Cohen-Hess-Voronov:stringtopacyclhom}
R.~Cohen, K.~Hess, and A.~Voronov, \emph{String topology and cyclic homology},
  Advanced Courses in Mathematics - {CRM} Barcelona, Birkh\"auser, 2006, Summer
  School in Almeria, 2003.

\bibitem{1095.55006}
Ralph~L. Cohen and V{\'e}ronique Godin, \emph{A polarized view of string
  topology}, Topology, geometry and quantum field theory, London Math. Soc.
  Lecture Note Ser., vol. 308, Cambridge Univ. Press, Cambridge, 2004,
  pp.~127--154.

\bibitem{Conant-Vogtmann:theokontsevich}
James Conant and Karen Vogtmann, \emph{On a theorem of {K}ontsevich}, Algebr.
  Geom. Topol. \textbf{3} (2003), 1167--1224 (electronic).

\bibitem{Costello:tcftcalabiyaucat}
K.~Costello, \emph{Topological conformal field theories and {C}alabi-{Y}au
  categories}, Adv. Math. \textbf{210} (2007), no.~1, 165--214.

\bibitem{Dwyer:transfer}
W.~G. Dwyer, \emph{Transfer maps for fibrations}, Math. Proc. Cambridge Philos.
  Soc. \textbf{120} (1996), no.~2, 221--235.

\bibitem{Dwyer-Wilkerson:hfpmLg}
W.~G. Dwyer and C.~W. Wilkerson, \emph{Homotopy fixed-point methods for {L}ie
  groups and finite loop spaces}, Ann. of Math. (2) \textbf{139} (1994), no.~2,
  395--442.

\bibitem{Dwyer-Henn:barcelone}
William~G. Dwyer and Hans-Werner Henn, \emph{Homotopy theoretic methods in
  group cohomology}, Advanced Courses in Mathematics. CRM Barcelona,
  Birkh\"auser Verlag, Basel, 2001.

\bibitem{Earle-Schatz}
C.~J. Earle and A.~Schatz, \emph{Teichm\"uller theory for surfaces with
  boundary}, J. Differential Geometry \textbf{4} (1970), 169--185.

\bibitem{Farnsteiner}
Rolf Farnsteiner, \emph{On {F}robenius extensions defined by {H}opf algebras},
  J. Algebra \textbf{166} (1994), no.~1, 130--141.

\bibitem{Felix-Halperin-Thomas:ratht}
Y.~F{\'e}lix, S.~Halperin, and J.-C. Thomas, \emph{Rational homotopy theory},
  Graduate Texts in Mathematics, vol. 205, Springer-Verlag, 2000.

\bibitem{Felix-Menichi-Thomas:GerstduaiHochcoh}
Y.~F{\'e}lix, L.~Menichi, and J.-C. Thomas, \emph{Gerstenhaber duality in
  {H}ochschild cohomology}, J. Pure Appl. Algebra \textbf{199} (2005), no.~1-3,
  43--59.

\bibitem{freedverlindektheory}
Daniel~S. Freed, \emph{The {V}erlinde algebra is twisted equivariant
  {$K$}-theory}, Turkish J. Math. \textbf{25} (2001), no.~1, 159--167.

\bibitem{FreedHopkinsTeleman3}
Daniel~S. Freed, Michael~J. Hopkins, and Constantin Teleman, \emph{Twisted
  {$K$}-theory and loop group representations}, preprint: math.AT/0312155,
  2003.

\bibitem{galatiusmadsentillmannweiss}
Soren Galatius, Ib~Madsen, Ulrike Tillmann, and Michael Weiss, \emph{The
  homotopy type of the cobordism category}, preprint: math.AT/0605249, 2006.

\bibitem{Getzler:BVAlg}
E.~Getzler, \emph{Batalin-{V}ilkovisky algebras and two-dimensional topological
  field theories}, Comm. Math. Phys. \textbf{159} (1994), no.~2, 265--285.

\bibitem{Godin:higherstring}
V.~Godin, \emph{Higher string topology operations}, preprint:
  math.AT/0711.4859, 2007.

\bibitem{Goldmancrochet}
William~M. Goldman, \emph{Invariant functions on {L}ie groups and {H}amiltonian
  flows of surface group representations}, Invent. Math. \textbf{85} (1986),
  no.~2, 263--302.

\bibitem{Goodwillie:cychdfl}
T.~G. Goodwillie, \emph{Cyclic homology, derivations, and the free loopspace},
  Topology \textbf{24} (1985), no.~2, 187--215.

\bibitem{Gottlieb:bundleEuler}
Daniel~Henry Gottlieb, \emph{Fibre bundles and the {E}uler characteristic}, J.
  Differential Geometry \textbf{10} (1975), 39--48.

\bibitem{gruher-salvatore:genstringtopop}
K.~Gruher and P.~Salvatore, \emph{Generalized string topology operations},
  Proc. London Math. Soc. (3) \textbf{96} (2008), no.~1, 78--106.

\bibitem{Gruher:duality}
Kate Gruher, \emph{A duality between string topology and the fusion product in
  equivariant {$K$}-theory}, Math. Res. Lett. \textbf{14} (2007), no.~2,
  303--313.

\bibitem{Gruher-Westerland}
Kate Gruher and Craig Westerland, \emph{String topology prospectra and
  hochschild cohomology}, preprint: math.AT/07101445, 2007.

\bibitem{Hatcher:algtop}
A.~Hatcher, \emph{Algebraic topology}, Cambridge University Press, 2002.

\bibitem{Hilton-Roitbergcriminal}
Peter Hilton and Joseph Roitberg, \emph{On principal {$S\sp{3}$}-bundles over
  spheres}, Ann. of Math. (2) \textbf{90} (1969), 91--107.

\bibitem{Humphreys:symHopf}
J.~E. Humphreys, \emph{Symmetry for finite dimensional {H}opf algebras}, Proc.
  Amer. Math. Soc. \textbf{68} (1978), no.~2, 143--146.

\bibitem{Kawakubo:transformationgroups}
Katsuo Kawakubo, \emph{The theory of transformation groups}, japanese ed., The
  Clarendon Press Oxford University Press, New York, 1991.

\bibitem{thesedeKitchloo}
Nitu Kitchloo, \emph{Phd thesis}, MIT, 1998.

\bibitem{Kock:Frob2TQFT}
Joachim Kock, \emph{Frobenius algebras and 2{D} topological quantum field
  theories}, London Mathematical Society Student Texts, vol.~59, Cambridge
  University Press, Cambridge, 2004.

\bibitem{Lichnerowicz:integrationfibre}
Andr{\'e} Lichnerowicz, \emph{Un th\'eor\`eme sur l'homologie dans les espaces
  fibr\'es}, C. R. Acad. Sci. Paris \textbf{227} (1948), 711--712.

\bibitem{LodayJ.:cych}
J.~Loday, \emph{Cyclic homology}, second ed., Grundlehren der Mathematischen
  Wissenschaften, vol. 301, Springer-Verlag, Berlin, 1998.

\bibitem{Lorenz:representationHopf}
Martin Lorenz, \emph{Representations of finite-dimensional {H}opf algebras}, J.
  Algebra \textbf{188} (1997), no.~2, 476--505.

\bibitem{Lupercio-Uribe-Xicotencatl:Orbistring}
Ernesto Lupercio, Bernardo Uribe, and Miguel~A. Xicotencatl, \emph{Orbifold
  string topology}, Geom. Topol. \textbf{12} (2008), no.~4, 2203--2247.

\bibitem{Lyndon-Schupp:livre}
Roger~C. Lyndon and Paul~E. Schupp, \emph{Combinatorial group theory},
  Springer-Verlag, Berlin, 1977, Ergebnisse der Mathematik und ihrer
  Grenzgebiete, Band 89.

\bibitem{Madsen-Tillmann:stablemapping}
Ib~Madsen and Ulrike Tillmann, \emph{{The stable mapping class group and
  $Q(\Bbb C P\sp \infty\sb +)$.}}, Invent. Math. \textbf{145} (2001), no.~3,
  509--544.

\bibitem{Markl-Shnider-Stasheff:opeatp}
M.~Markl, S.~Shnider, and J.~Stasheff, \emph{Operads in algebra, topology and
  physics}, Mathematical Surveys and Monographs, vol.~96, Amer. Math. Soc.,
  2002.

\bibitem{Markl:operadprop}
Martin Markl, \emph{Operads and props}, preprint: math.AT/0601129, to appear in
  Handbook of Algebra, 2006.

\bibitem{MayClassifying}
J.~Peter May, \emph{Classifying spaces and fibrations}, Mem. Amer. Math. Soc.
  \textbf{1} (1975), no.~1, 155, xiii+98.

\bibitem{MenichiL:BValgaccoHa}
L.~Menichi, \emph{Batalin-{V}ilkovisky algebras and cyclic cohomology of {H}opf
  algebras}, $K$-Theory \textbf{32} (2004), no.~3, 231--251.

\bibitem{Menichi:BV_Hochschild}
\bysame, \emph{Batalin-{V}ilkovisky algebra structures on {H}ochschild
  cohomology}, \`a paraitre, Bull. Soc. Math. France \textbf{137} (2009),
  no.~2, 361--379.

\bibitem{Milnor-Stasheff}
John~W. Milnor and James~D. Stasheff, \emph{Characteristic classes}, Princeton
  University Press, Princeton, N. J., 1974, Annals of Mathematics Studies, No.
  76.

\bibitem{Mimura-Toda:topliegroups}
Mamoru Mimura and Hirosi Toda, \emph{Topology of {L}ie groups. {I}, {II}},
  Translations of Mathematical Monographs, vol.~91, American Mathematical
  Society, Providence, RI, 1991, Translated from the 1978 Japanese edition by
  the authors.

\bibitem{Morita:characteristic}
Shigeyuki Morita, \emph{Geometry of characteristic classes}, Translations of
  Mathematical Monographs, vol. 199, American Mathematical Society, Providence,
  RI, 2001, Translated from the 1999 Japanese original, Iwanami Series in
  Modern Mathematics.

\bibitem{Oberst-Schneider}
Ulrich Oberst and Hans-J{\"u}rgen Schneider, \emph{\"{U}ber {U}ntergruppen
  endlicher algebraischer {G}ruppen}, Manuscripta Math. \textbf{8} (1973),
  217--241.

\bibitem{RavenelMexique}
Douglas~C. Ravenel, \emph{Morava {$K$}-theories and finite groups}, Symposium
  on {A}lgebraic {T}opology in honor of {J}os\'e {A}dem ({O}axtepec, 1981),
  Contemp. Math., vol.~12, Amer. Math. Soc., Providence, R.I., 1982,
  pp.~289--292.

\bibitem{Ravenel:Nilpotence}
\bysame, \emph{Nilpotence and periodicity in stable homotopy theory}, Annals of
  Mathematics Studies, vol. 128, Princeton University Press, Princeton, NJ,
  1992, Appendix C by Jeff Smith.

\bibitem{Salvatore-Wahl:FrameddoBVa}
P.~Salvatore and N.~Wahl, \emph{Framed discs operads and {B}atalin-{V}ilkovisky
  algebras}, Q. J. Math. \textbf{54} (2003), no.~2, 213--231.

\bibitem{Segal:defCFT}
Graeme Segal, \emph{The definition of conformal field theory}, Topology,
  geometry and quantum field theory, London Math. Soc. Lecture Note Ser., vol.
  308, Cambridge Univ. Press, Cambridge, 2004, pp.~421--577.

\bibitem{Siegel-Witherspoon:Hochschildcoh}
S.~Siegel and S.~Witherspoon, \emph{The {H}ochschild cohomology ring of a group
  algebra}, Proc. London Math. Soc. \textbf{79} (1999), no.~1, 131--157.

\bibitem{Spanier:livre}
Edwin~H. Spanier, \emph{Algebraic topology}, Springer-Verlag, New York, 1981,
  Corrected reprint.

\bibitem{Strickland:duality}
N.~P. Strickland, \emph{{$K(N)$}-local duality for finite groups and
  groupoids}, Topology \textbf{39} (2000), no.~4, 733--772.

\bibitem{Sweedler:livre}
Moss~E. Sweedler, \emph{Hopf algebras}, Mathematics Lecture Note Series, W. A.
  Benjamin, Inc., New York, 1969.

\bibitem{Switzer}
Robert~M. Switzer, \emph{Algebraic topology---homotopy and homology},
  Springer-Verlag, New York, 1975, Die Grundlehren der mathematischen
  Wissenschaften, Band 212.

\bibitem{Tillmann:htpysmcg}
Ulrike Tillmann, \emph{On the homotopy of the stable mapping class group},
  Invent. Math. \textbf{130} (1997), no.~2, 257--275.

\bibitem{TillmannICM02}
\bysame, \emph{Strings and the stable cohomology of mapping class groups},
  Proceedings of the {I}nternational {C}ongress of {M}athematicians, {V}ol.
  {II} ({B}eijing, 2002) (Beijing), Higher Ed. Press, 2002, pp.~447--456.

\bibitem{VogtmannICM06}
Karen Vogtmann, \emph{The cohomology of automorphism groups of free groups},
  International {C}ongress of {M}athematicians. {V}ol. {II}, Eur. Math. Soc.,
  Z\"urich, 2006, pp.~1101--1117.

\bibitem{thesedewahl}
Nathalie Wahl, \emph{Phd thesis}, Oxford university, 2001.

\bibitem{Weibel:inthomalg}
C.~Weibel, \emph{An introduction to homological algebra}, Cambridge University
  Press, 1994.

\bibitem{Westerland:equivariant}
Craig Westerland, \emph{Equivariant operads, string topology, and {T}ate
  cohomology}, Math. Ann. \textbf{340} (2008), no.~1, 97--142.

\bibitem{Whitehead:eltsoht}
G.~Whitehead, \emph{Elements of homotopy theory}, Graduate Texts in
  Mathematics, vol.~61, Springer-Verlag, 1978.

\end{thebibliography}
\bibliographystyle{amsplain}

 \vspace{1cm}

\noindent
\begin{tabular*}{4in}[t]{lcr}
david.chataur at math.univ-lille1.fr     & \hspace{1cm} &          luc.menichi at univ-angers.fr \\
D\'epartement de math\'ematiques Paul Painlev\'e&&     LAREMA UMR
6093, CNRS\\
UFR de math\'ematique&&                               Facult\'e des Sciences\\
USTL&&                                              2, Boulevard Lavoisier\\
59655 Villeneuve d'Ascq C\'edex, France. &&          49045 Angers, France.
\end{tabular*}

\end{document}